\documentclass[letterpaper,11pt]{article}
\usepackage{amsmath,amssymb,geometry,natbib,bbm,enumitem}
\usepackage{xcolor}
\usepackage{graphicx}
\usepackage{placeins}
\definecolor{lightgray}{gray}{0.85}
\usepackage[latin1]{inputenc}
\usepackage[T1]{fontenc}
\usepackage[colorlinks=true,linkcolor=red,citecolor=blue]{hyperref}
\newtheorem{theorem}{Theorem}[section]

\newtheorem{corollary}[theorem]{Corollary}
\newtheorem{definition}[theorem]{Definition}

\newtheorem{proposition}[theorem]{Proposition}
\newtheorem{remark}[theorem]{Remark}
\newtheorem{theo}{Theorem}

\newenvironment{proof}{\noindent\textbf{Proof.} }{\hfill$\Box$}
\numberwithin{equation}{section}
\newcommand{\lqn}[1]{\noalign{\noindent $\displaystyle#1$}}

\begin{document}

\title{Strong approximations for the $p$-fold integrated empirical process
\\
with applications to statistical tests}
\author{Sergio Alvarez-Andrade\footnote{e-mail: sergio.alvarez@utc.fr}~,
~ Salim Bouzebda\footnote{e-mail: salim.bouzebda@utc.fr}
~ and Aim\'e Lachal\footnote{e-mail: aime.lachal@insa-lyon.fr}
\\~\\
$^{*\dag}$ Sorbonne Universit\'es, Universit\'e de Technologie de Compi\`egne
\\
\textit{Laboratoire de Math\'ematiques Appliqu\'ees de Compi\`egne}
\\
$^\ddag$ Universit\'e de Lyon, Institut National des Sciences Appliqu\'ees de Lyon
\\
\textit{Institut Camille Jordan}}
\date{}
\maketitle

%%%%%%%%%%%%%%%%%%%%
\begin{abstract}
\noindent
The main purpose of this paper is to investigate the strong approximation of
the $p$-fold integrated empirical process, $p$ being a fixed positive integer.
More precisely, we obtain the exact rate of the
approximations by a sequence of weighted Brownian bridges and a weighted Kiefer
process. Our arguments are based in part on the \cite{KMT1975}'s results.
We  obtain an exponential bound for the tail probability of the
weighted approximation to the $p$-fold integrated empirical process.
Applications include the two-sample testing procedures together with the
change-point problems. We also consider the strong approximation of integrated
empirical processes when the parameters are estimated. We study the
behavior of the self-intersection local time of the partial sum process
representation of integrated empirical processes. Finally, simulation results
are provided to illustrate the finite sample performance of the proposed
statistical tests based on the integrated empirical processes.
\end{abstract}
%%%%%%%%%%%%%%%%%%%%

%%%%%%%%%%%%%%%%%%%%%%%%%%%%%%%%%%%%%%%%%%%%%%%%%%%%%%%%%%%%%%%%%%%%%%%%%%%%%%%
\section{Introduction}
%%%%%%%%%%%%%%%%%%%%%%%%%%%%%%%%%%%%%%%%%%%%%%%%%%%%%%%%%%%%%%%%%%%%%%%%%%%%%%%

Let $F=\{F(t),t\in\mathbb{R}\}$ be a continuous distribution function [d.f.] and denote by
$Q=\{Q(u),u\in[0,1]\}$ the usual quantile function (generalized inverse) pertaining to $F$
defined as
$$
Q(u) := \inf\{t\in\mathbb{R}: F(t)\geq u\} \quad\text{for}\quad u\in(0,1),
$$
$$
Q(0):=\lim_{u\downarrow 0}Q(u)\quad\text{and}\quad Q(1):=\lim_{u\uparrow 1}Q(u).
$$
The function $Q$ is strictly increasing and we have $F(Q(u))=u$ for any $u\in[0,1]$.
We denote the sets $\lbrace 0, 1, 2,\ldots\rbrace$ and $\lbrace 1, 2, \ldots\rbrace$
respectively by $\mathbb{N}$ and $\mathbb{N}^*$.
Consider now a sequence of independent, identically
distributed [i.i.d.] random variables [r.v.'s] $\{U_i: i\in\mathbb{N}^*\}$
uniformly distributed on $[0,1]$ and, for each $i\in\mathbb{N}^*$, set $X_i:=Q(U_i)$.
The sequence $\{X_i: i\in\mathbb{N}^*\}$ consists of i.i.d. r.v.'s with
d.f. $F$: $F(t)=\mathbb{P}\{X_1\leq t\}$ for $t\in\mathbb{R}$
(cf., e.g., \cite{Shorack1986}, p.~3 and the references therein). Moreover,
we conversely have $U_i=F(X_i)$.

For each $n\in\mathbb{N}^*$, let $\mathbb{F}_n$  and $\mathbb{U}_n$ be the
empirical d.f.'s based upon the respective samples $X_1,\ldots,X_n$ and
$U_1,\ldots,U_n$ defined by
\begin{align*}
\mathbb{F}_n(t)
&:=\frac{1}{n}\#\{i\in\{1,\dots,n\}:X_i\leq t\}
=\frac{1}{n}\sum_{i=1}^n \mathbbm{1}_{\{X_i\leq t\}}
\quad\text{for}\quad t\in\mathbb{R},
\\
\mathbb{U}_n(u)
&:= \frac{1}{n}\#\{i\in\{1,\dots,n\}:U_i \leq u\}
=\frac{1}{n}\sum_{i=1}^n \mathbbm{1}_{\{U_i\leq u\}}
\quad\text{for}\quad u\in[0,1],
\end{align*}
where $\#$ denotes cardinality. For each $n\in\mathbb{N}^*$,
we introduce the \emph{empirical process} $\alpha_n$ and the
\emph{uniform empirical process} $\beta_n$ defined by
\begin{align}
\alpha_n(t)
&:=\sqrt{n}\,(\mathbb{F}_n(t)-F(t))
\quad\text{for}\quad t\in\mathbb{R},
\label{beta1}\\
\beta_n(u)
&:= \sqrt{n}\left(\mathbb{U}_n(u)-u\right)
\quad\text{for}\quad u\in[0,1].
\label{uniformprocesse}
\end{align}
We have of course the usual relations between the empirical process and uniform empirical
process:
\begin{align}
\alpha_n(t)
&=\beta_n(F(t))\quad\text{for}\quad t\in\mathbb{R},\,n\in\mathbb{N}^*,
\label{identity}
\\
\beta_n(u)
&=\alpha_n(Q(u))\quad\text{for}\quad u\in\mathbb[0,1],\,n\in\mathbb{N}^*.
\label{identity2}
\end{align}

In this paper, we consider integrated empirical d.f.'s based upon the samples
$X_1,\ldots,X_n$ and $U_1,\ldots,U_n$ together with the corresponding integrated
empirical processes in the following sense.
%%%%%%%%%%%%%%%%%%%%
\begin{definition}\label{def}
We define the families of integrated d.f.'s and integrated empirical d.f.'s
associated with the d.f.~$F$, for any $p\in\mathbb{N}$, any $n\in\mathbb{N}^*$
and any $t\in\mathbb{R}$, as
$$
F^{(0)}(t):=F(t),\quad \mathbb{F}_n^{(0)}(t):=\mathbb{F}_n(t),
$$
$$
F^{(1)}(t):=\int_{-\infty}^{t} F(s)\,dF(s),\quad
\mathbb{F}_n^{(1)}(t):=\int_{-\infty}^{t} \mathbb{F}_n(s)\,d\mathbb{F}_n(s),
$$
and for $p\geq 2$,
$$
F^{(p)}(t):=\int_{-\infty}^{t} dF(s_1)\int_{-\infty}^{s_1}\,dF(s_2)\dots
\int_{-\infty}^{s_{p-1}} F(s_p)\,dF(s_p),
$$
$$
\mathbb{F}_n^{(p)}(t):=\int_{-\infty}^{t}\,d\mathbb{F}_n(s_1)
\int_{-\infty}^{s_1}\,d\mathbb{F}_n(s_2)\dots\int_{-\infty}^{s_{p-1}}
\mathbb{F}_n(s_p)\,d\mathbb{F}_n(s_p),
$$
together with the corresponding family of integrated empirical processes as
\begin{equation}\label{alpha}
\alpha_n^{(p)}(t):=\sqrt{n}\left(\mathbb{F}_n^{(p)}(t)-F^{(p)}(t)\right)\!.
\end{equation}
\end{definition}
%%%%%%%%%%%%%%%%%%%%

Notice that $F^{(p)}$ (resp. $\mathbb{F}_n^{(p)}$) is a kind of \emph{$p$-fold
integral} with respect to the measure $dF$ (resp. $d\mathbb{F}_n$). Hence,
we will call $F^{(p)}$ (resp. $\mathbb{F}_n^{(p)}$, $\alpha_n^{(p)}$)
throughout the paper \emph{$p$-fold integrated d.f.} (resp. \emph{$p$-fold
integrated empirical d.f., $p$-fold integrated empirical process}).
Finally, we define exactly in the same manner the
\emph{$p$-fold integrated uniform empirical d.f.} $\mathbb{U}_n^{(p)}$
and the \emph{$p$-fold integrated uniform empirical process} $\beta_n^{(p)}$.

\begin{figure}[!ht]
\begin{center}
\centerline{
\includegraphics[width=8cm]{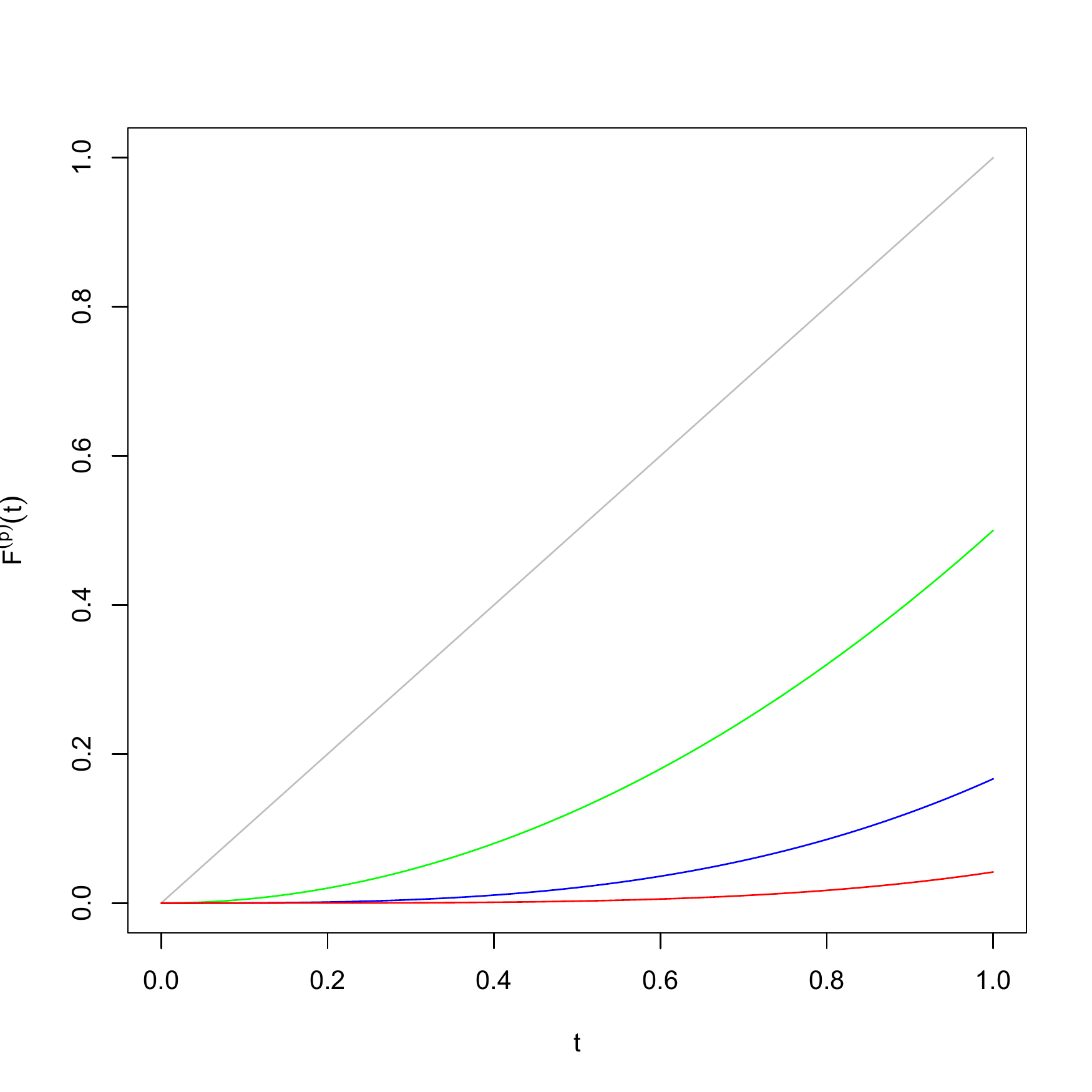}
\includegraphics[width=8cm]{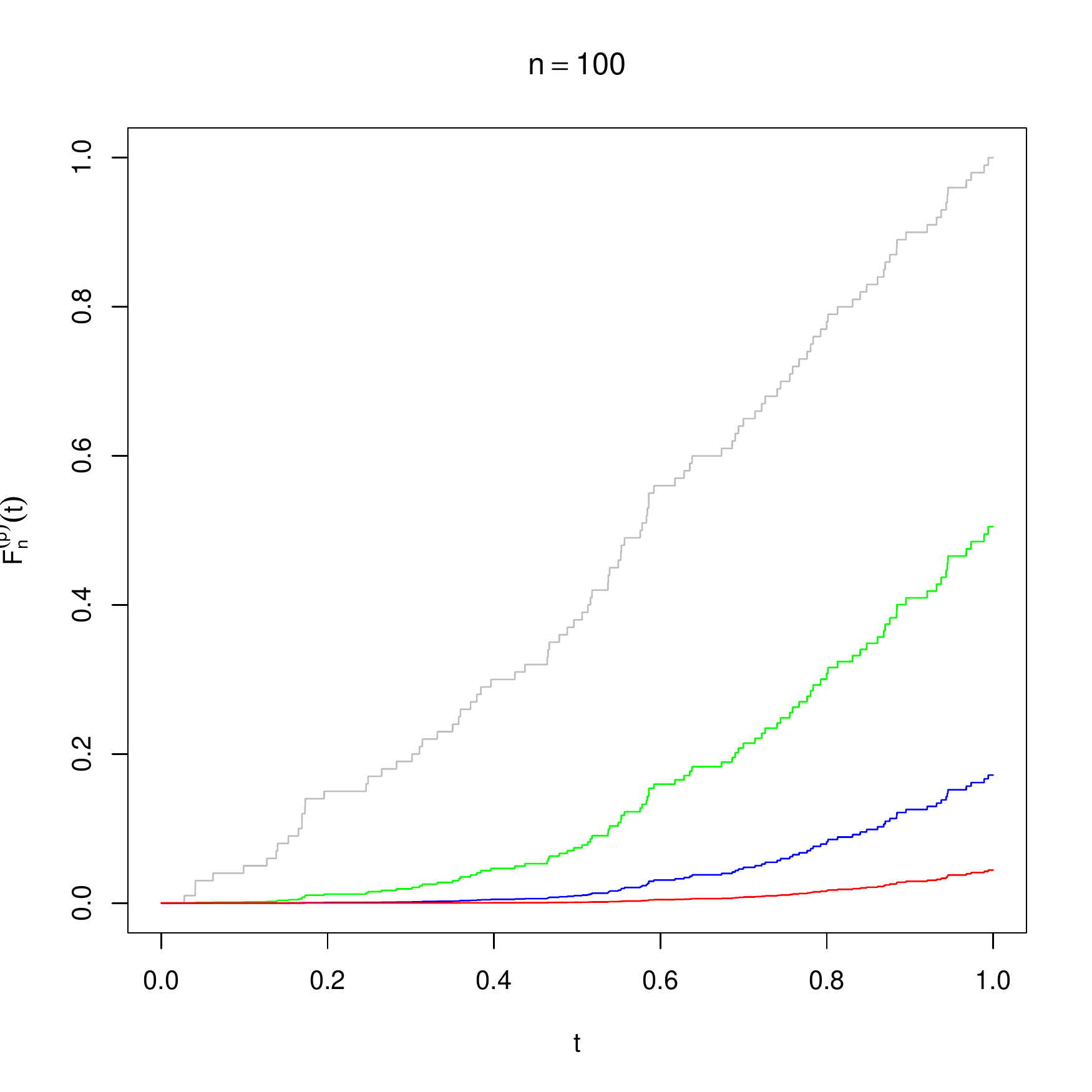}}
\end{center}
\caption{On the left: the theoretical functions $F^{(p)}, p=0,1,2,3$;
on the right: the empirical counterparts for the uniform distribution.
(``{\textcolor{gray}{--}}'': $\mathbb{F}_{n}, F$ ),
(``{\textcolor{green}{--}}'':  $\mathbb{F}_{n}^{(1)},~~F^{(1)}$),
(``{\textcolor{blue}{--}}'':  $\mathbb{F}_{n}^{(2)}, ~~F^{(2)}$),
(``{\textcolor{red}{--}}'':  $\mathbb{F}_{n}^{(3)}, ~~F^{(3)}$)} \label{normallocation2}
\end{figure}

\begin{figure}[!ht]
\begin{center}
\centerline{\includegraphics[width=9cm]{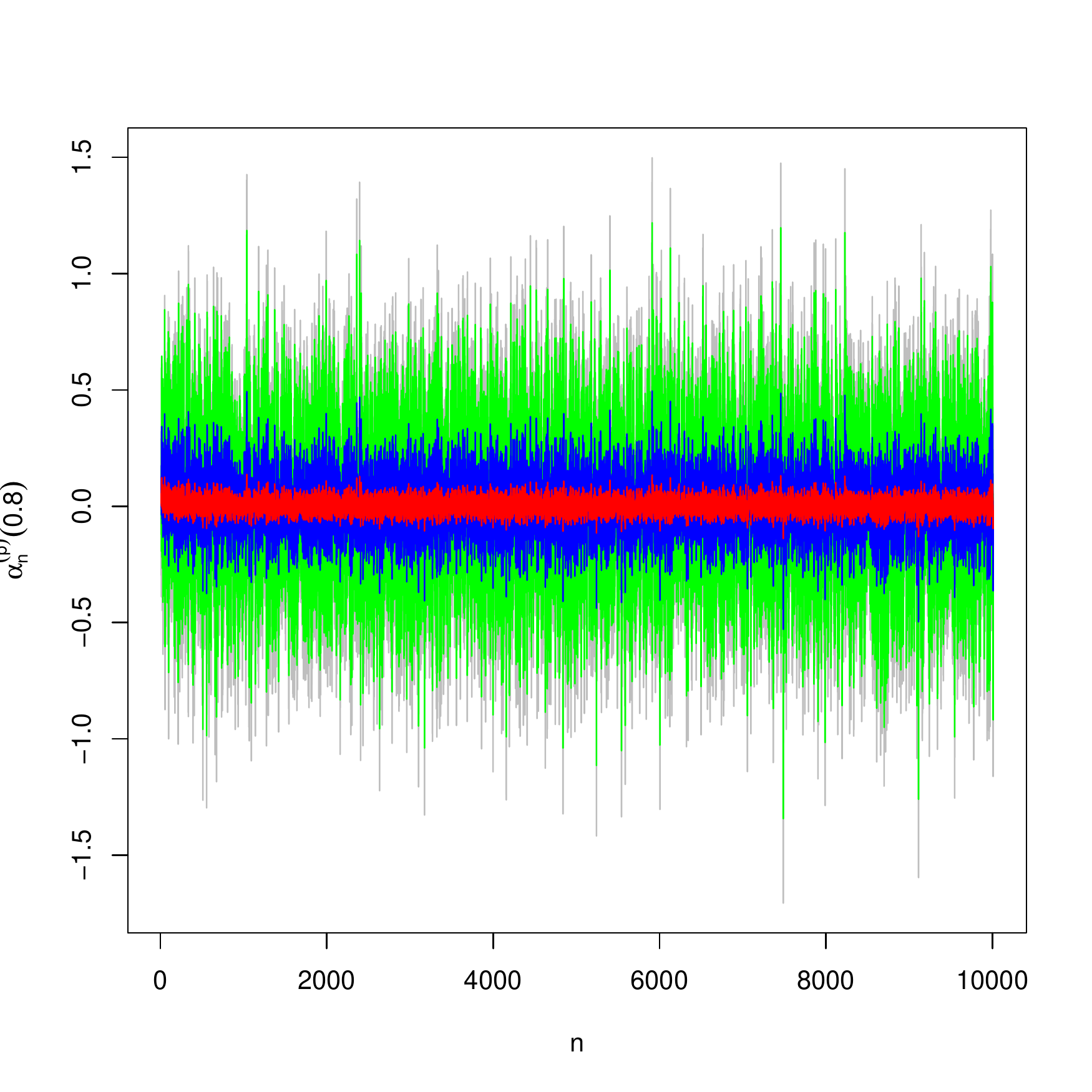}}
\end{center}
\caption{The uniform empirical processes for $p=0,1,2,3$ as functions of the sample size.
(``{\textcolor{gray}{--}}'': $\alpha_{n},$ ),
(``{\textcolor{green}{--}}'':  $\alpha^{(1)}_{n},$),
(``{\textcolor{blue}{--}}'':  $\alpha_{n}^{(2)},$),
(``{\textcolor{red}{--}}'':  $\alpha_{n}^{(3)}$)} \label{normallocation2}
\end{figure}
Below, we provide explicit expressions for $F^{(p)}$ and $\mathbb{F}_n^{(p)}$,
the proof of which are postponed to Section~\ref{appe}.
%%%%%%%%%%%%%%%%%%%%
\begin{proposition}\label{expFp}
For each $p\in\mathbb{N}$, we explicitly have, with probability~$1$,
\begin{equation}\label{repFnp1}
F^{(p)}(t)=\frac{F(t)^{p+1}}{(p+1)!},\quad
\mathbb{F}_n^{(p)}(t)=\frac{1}{n^{p+1}} {n\mathbb{F}_{n}(t)+p \choose p+1}
\quad\text{for}\quad t\in\mathbb{R},\,n\in\mathbb{N}^*.
\end{equation}
\end{proposition}
%%%%%%%%%%%%%%%%%%%%

The particular case where $p=1$ has often been considered in the literature.
\cite{HenzeNikitin2000,HenzeNikitin2002} introduced and deeply investigated
the goodness-of-fit testing procedures based on the integrated empirical process.
Indeed, the asymptotic properties of their procedures, Kolmogorov-Smirnov,
Cram\'er-von Mises and Watson-type statistics, can be derived from the limiting
behavior of the integrated empirical process. \cite{HenzeNikitin2003} considered
a two-sample testing procedure and focused on the approximate local Bahadur
efficiencies of their statistical tests. It is noteworthy to point out that
tests based on some integrated empirical processes turn out to be more efficient
for certain distributions, we may refer at this point to
\cite{HenzeNikitin2003,HenzeNikitin2002,HenzeNikitin2000}.
In \cite{HenzeNikitin2003}, it is shown that
the statistics based on the integrated empirical processes  perform better
the classical ones, under asymmetric alternatives. For instance,
Kolmogorov-Smirnov test based on these integrated empiricals has maximal
Bahadur efficiency if the underlying distribution is skew-Laplace.
We mention also that in the paper \cite{HenzeNikitin2000} where a new approach
to goodness-of-fit testing was proposed by using integrated
empirical processes. In this paper, the authors have established that
the integrated Kolmogorov- Smirnov test is locally Bahadur optimal
for the logistic distribution and the statistic
$$
\sqrt{n}\int_{\mathbb{R}} \left(\mathbb{F}_n^{(p)}(t)
-F_0^{(p)}(t)\right) dF_0(t), ~~\mbox{for} ~~ p=1,
$$
turns out to be locally Bahadur optimal for the ``root-logistic distribution''.
In the paper by \cite{HenzeNikitin2002}, one of the new tests is locally
Bahadur optimal for the hyperbolic cosine distribution, we may refer to Theorem 6.1
therein for exact formulation. In \cite{Aime2001}, another version of the $p$-fold
integrated empirical process ($p\in\mathbb{N}^*$) was introduced. For the extension
to the multivariate framework, we may refer to \cite{Jing2006} and \cite{Jing2007}
where some projected integrated empirical processes for testing the equality of
two multivariate distributions are considered. Inspired by the work of
\cite{HenzeNikitin2003}, \cite{BouzebdaElfaouzy2012} developed multivariate
two-sample testing procedures based on the integrated empirical copula process
that are extended to the $K$-sample problem in \cite{Bouzebdaetal2011}.
Emphasis is placed on the explanation of the strong approximation methodology.
The asymptotic behavior of weighted multivariate Cram\'er-von Mises-type
statistics under contiguous alternatives was characterized by
\cite{BouzebdaZari2013}. For more recent references, we refer to
\cite{DurioNikitin2016}, \cite{Bouzebda2016} and \cite{Bouzebdaetal2015}.

The main purpose of this paper is to investigate the strong approximation of
the $p$-fold integrated empirical process. Next we use the obtained results
for studying the asymptotic properties of statistical tests based on this process.
We point out that strong approximations are quite useful and have
received considerable attention in probability theory. Indeed, many
well-known and important probability theorems can be considered as consequences of
results about strong approximation of sequences of sums by corresponding
Gaussian sequences.

We will first obtain an upper bound in probability for the distance
between the $p$-fold integrated empirical process and a sequence of appropriate
Brownian bridges (see Theorem~\ref{lem2a}). This is the key point of our study.
From this, we will deduce a strong approximation of the $p$-fold integrated
empirical process by this sequence of { transformed} Brownian bridges (see Corollary~\ref{corollaryapprox}).
As an application, we will derive the rates of convergence for the
distribution of smooth functionals of each $p$-fold integrated empirical process
(see Corollary~\ref{corol22}). Moreover, we will deduce
strong approximations for the Kolmogorov-Smirnov and Cram\'er-von Mises-type
statistics associated with the $p$-fold integrated empirical processes
(see Corollary~\ref{ecooooo}).

Second, we will obtain a strong approximation of the $p$-fold integrated
empirical process by a { transformed} Kiefer processes (see Theorem~\ref{kieferapproximation}).
This latter is of particular interest; indeed, for instance, any kind of law
of the iterated logarithm which holds for the partial sums of Gaussian
processes may then be transferred to the $p$-fold integrated empirical processes
(see Corollary~\ref{corollarylli}).
We may refer to \cite{Dasgupta2008} (Chapter 12), \cite{Csorgoho1993} (Chapter 3),
\cite{Csorgo1981REVESZ} (Chapters 4-5) and \cite{Shorack1986} (Chapter 12) for
expositions, details and references about this problem.

Third, we will apply our theoretical results to some statistical applications.
More precisely, we will consider the famous two-sample and change-point problems
for which we will develop procedures based on statistics involving
$p$-fold integrated empirical distribution functions.

We refer to \cite{Hall1983S}, \cite{csorgo2007} and \cite{Mason2012} for a
survey of possible applications of the strong approximation and many references.
There is a huge literature on the strong approximations and their applications.
It is not the purpose of this paper to survey this extensive literature.

The layout of the article is as follows. In Section~\ref{sectionone}, we first
present some strong approximation results for the $p$-fold integrated empirical
process; our main tools are the results of \cite{KMT1975}.
In section \ref{section2-1}, we consider weighted approximations.
Sections~\ref{sectiontwo} and \ref{sectionchnae} are devoted to statistical applications,
namely the two-sample and change-point problems respectively. In Section~\ref{Section5},
we deal with the strong approximation of the $p$-fold integrated
empirical process when parameters are estimated. Section~\ref{section343} is
concerned with the behavior of the self-intersection local time of the partial
sum process representation of the $p$-fold integrated empirical process.
In section \ref{simulation}, simulation results are performed in order
to illustrate the finite sample performances of the proposed statistical tests
for testing uniformity based on the integrated empirical processes.
Finally, in the Appendix, we suggest some possible extensions of our work for further investigations.

To prevent from interrupting the flow of the presentation, all mathematical
developments are postponed to Section~\ref{appe}.

%%%%%%%%%%%%%%%%%%%%%%%%%%%%%%%%%%%%%%%%%%%%%%%%%%%%%%%%%%%%%%%%%%%%%%%%%%%%%%%
\section{Strong approximation}\label{sectionone}
%%%%%%%%%%%%%%%%%%%%%%%%%%%%%%%%%%%%%%%%%%%%%%%%%%%%%%%%%%%%%%%%%%%%%%%%%%%%%%%

%%%%%%%%%%%%%%%%%%%%%%%%%%%%%%%%%%%%%%%%%%%%%%%%%%%%%%%%%%%%%%%%%%%%%%%%%%%%%%%
\subsection{Some processes}

First, we introduce some definitions and notations. Let $\mathbb{W}=
\{\mathbb{W}(s): s\geq 0\}$ and $\mathbb{B}=\{\mathbb{B}(u) : u\in[0,1]\}$ be
the standard Wiener process and Brownian bridge, that is, the centered Gaussian
processes with continuous sample paths, and covariance functions
$$
\mathbb{E}(\mathbb{W}(s)\mathbb{W}(t))=s\wedge t\quad\text{for}\quad s,t\geq 0
$$
and
$$
\mathbb{E}(\mathbb{B}(u)\mathbb{B}(v))=u\wedge v- uv\quad\text{for}\quad u,v\in[0,1].
$$
A Kiefer process $\mathbb{K}=\{\mathbb{K}(s,u) : s \geq 0, u\in[0,1]\}$ is a two-parameters
centered Gaussian process, with continuous sample paths, and covariance function
$$
\mathbb{E}(\mathbb{K}(s,u)\mathbb{K}(t,v))=(s\wedge t)\,(u\wedge v-uv)
\quad\text{for}\quad s,t\geq 0\quad\text{and}\quad u,v\in[0,1].
$$
It satisfies the following distributional identities:
$$
\left\{\mathbb{K}(s,u): u\in[0,1]\right\}\stackrel{\mathcal{L}}{=}
\left\{\sqrt{s}\,\mathbb{B}(u): u\in[0,1]\right\}\quad\text{for}\quad s \geq 0
$$
and
$$
\left\{\mathbb{K}(s,u): s\geq 0\right\}\stackrel{\mathcal{L}}{=}
\left\{\sqrt{u(1-u)}\,\mathbb{W}(s): s\geq 0\right\}\quad\text{for}\quad u\in[0,1],
$$
where $\stackrel{\mathcal{L}}{=}$ stands for the equality in distribution.
The interested reader may refer to \cite{Csorgo1981REVESZ} for details on
the Gaussian processes mentioned above.

%%%%%%%%%%%%%%%%%%%%%%%%%%%%%%%%%%%%%%%%%%%%%%%%%%%%%%%%%%%%%%%%%%%%%%%%%%%%%%
\subsection{Brownian approximation}

It is well-known that the empirical
uniform process $\{\beta_n:n\in\mathbb{N}^*\}$ converges to $\mathbb{B}$ in
$D[0,1]$ (the space of all right-continuous real-valued functions defined on
$[0,1]$ which have left-hand limits, equipped with the Skorohod topology; see,
for details, \cite{Billingsley1968}). The rate of convergence of this process
to $\mathbb{B}$ is an important task in statistics as well as in probability
that has been investigated by several authors. We can and will assume without
loss of generality that all r.v.'s and processes introduced so far and later on
in this paper can be defined on the same probability space (cf. Appendix 2 in
\cite{Csorgoho1993}).

{  In \cite{CsCsHoMa1986}, it was stated the following Brownian bridge
approximation} for $\{\beta_n:n\in\mathbb{N}^*\}$ (Formula~(\ref{major})), along with a description of its proof with few details, which has been subsequently refined by \cite{MasonvanZwet1987}
(Formula~(\ref{mason-inequality})).

%%%%%%%%%%%%%%%%%%%%
\begin{theo}
On a suitable probability space, we may define the uniform empirical process
$\{\beta_n:n\in\mathbb{N}^*\}$, in combination with a sequence of Brownian
bridges $\{\mathbb{B}_n:n\in\mathbb{N}^*\}$, such that,
for any $d,n\in\mathbb{N}^*$ satisfying $d\leq n$ and any positive number $x$,
\begin{equation}\label{mason-inequality}
\mathbb{P}\bigg\{ \sup_{u\in[0,d/n]}\left|\beta_n(u)-\mathbb{B}_n(u)\right|
\geq \frac{1}{\sqrt{n}}\,(c_1\log d +x)\bigg\}\leq c_2\,\exp(-c_3x)
\end{equation}
where $c_1$, $c_2$ and $c_3$ are suitable absolute constants.
The same inequality holds when replacing the interval $[0,d/n]$ by $[1-d/n,1]$.
In particular, for $d=n$,
\begin{equation}\label{major}
\mathbb{P} \bigg\{\sup_{u\in[0,1]} |\beta_n(u)-\mathbb{B}_n(u)|
\geq \frac{1}{\sqrt{n}}\,(c_1\log n+x)\bigg\}\leq c_2 \,\exp(-c_3x).
\end{equation}
\end{theo}
%%%%%%%%%%%%%%%%%%%%
In~(\ref{major}), suitable explicit values for $c_1,c_2,c_3$ were exhibited by
\cite{BretagnolleMassart1989}, Theorem~1: $c_1=12,c_2=2,c_3=1/6$.
In his manuscript, \cite{Major2000} details the original proof of (\ref{major}).
\cite{Chatterjee2012} provided a new alternative approach for proving the
famous KMT theorem.

%%%%%%%%%%%%%%%%%%%%
\begin{remark}
In the sequel, the precise meaning of ``suitable probability space'' is that
an independent sequence of Wiener processes, which is independent of the
originally given sequence of i.i.d. r.v.'s, can be constructed on the assumed
probability space. This is a technical requirement which allows the construction
of the Gaussian processes displayed in our theorems, and which is not restrictive
since one can expand the probability space to make it rich enough (see, e.g.,
Appendix 2 in \cite{Csorgoho1993}, \cite{deAcosta1982}, \cite{Csorgo1981REVESZ}
and Lemma A1 in \cite{Berkes1979}). Throughout this paper, it will be assumed
that the underlying probability spaces are suitable in this sense.
\end{remark}
%%%%%%%%%%%%%%%%%%%%

In the following theorem, we state the key point to access the strong
Brownian approximation of the $p$-fold integrated uniform empirical process
$\big\{\beta_n^{(p)}:n\in\mathbb{N}^*\big\}$. {  The following Theorem can be seen as a version of \cite{CsCsHoMa1986} for iterated processes.}

%%%%%%%%%%%%%%%%%%%%
\begin{theorem}\label{lem2b}
Fix $p\in\mathbb{N}^*$. On a suitable probability space, we may define the
$p$-fold integrated uniform empirical process $\big\{\beta_n^{(p)}:n\in\mathbb{N}^*\big\}$,
in combination with a sequence of Brownian bridges
$\{\mathbb{B}_n:n\in\mathbb{N}^*\}$, such that,
for any $d,n\in\mathbb{N}^*$ satisfying $d\leq n$ and large enough $x$,
\begin{equation}\label{estimation-b}
\mathbb{P}\bigg\{ \sup_{u\in[0,d/n]} \left|\beta_n^{(p)}(u)
-\mathbb{B}_n^{(p)}(u)\right|
\geq \frac{1}{\sqrt{n}}\,(c_1\log d +x)\bigg\}
\leq B_p \sum_{k=2}^{p+1}\exp\!\left(-C_p\,x^{2/k}n^{1-2/k}\right)
\end{equation}
where $B_p$ and $C_p$ are positive constants depending on $p$, $c_1$
is the constant arising in (\ref{major}) and, for each $n\in\mathbb{N}^*$,
$\mathbb{B}_n^{(p)}$ is the process defined by
$$
\mathbb{B}_n^{(p)}(u):=\frac{1}{p!}\,u^p\,\mathbb{B}_n(u)
\quad\text{for}\quad u\in[0,1].
$$
The same inequality holds when replacing the interval $[0,d/n]$ by $[1-d/n,1]$.
\end{theorem}
%%%%%%%%%%%%%%%%%%%%
In particular, making $d=n$ in~(\ref{estimation-b}), we obtain the key estimate for the $p$-fold
integrated empirical process $\big\{\alpha_n^{(p)}:n\in\mathbb{N}^*\big\}$ below.
%%%%%%%%%%%%%%%%%%%%
\begin{theorem}\label{lem2a}
Fix $p\in\mathbb{N}^*$. On a suitable probability space, we may define the
$p$-fold integrated empirical process $\{\alpha_n^{(p)}:n\in\mathbb{N}^*\}$,
in combination with a sequence of Brownian bridges
$\left\{\mathbb{B}_n:n\in\mathbb{N}^*\right\}$, such that,
for large enough $x$ and all $n\in\mathbb{N}^*$,
\begin{equation}\label{estimation-a}
\mathbb{P}\!\left\{ \sup_{t\in\mathbb{R}} \left|\alpha_n^{(p)}(t)
-\mathbb{B}_n^{(p)}(F(t))\right|
\geq \frac{1}{\sqrt{n}}\,(c_1\log n+x)\right\}
\leq B_p \sum_{k=2}^{p+1}\exp\!\left(-C_p\,x^{2/k}n^{1-2/k}\right)\!.
\end{equation}
\end{theorem}
%%%%%%%%%%%%%%%%%%%%
One of the immediate consequences of Theorem~\ref{lem2a} is an upper bound for the
convergence of distributions of smooth functionals of $\alpha_n^{(p)}$.
Indeed, applying (\ref{estimation-a}) with $x=c\log n$
for a suitable constant $c$ yields the result below. Notice that the following corollary
is the analogous of the Corollary of \cite{KMT1975} page 113.
Let $\mathcal{D}(\mathbb{A})$ be the space of right-continuous real-valued
functions defined on $\mathbb{A}$ which have left-hand limits, equipped
with the Skorohod topology; refer to \cite{Billingsley1968} for further details on this problem.
%%%%%%%%%%%%%%%%%%%%
\begin{corollary}\label{corol22}
Fix $p\in\mathbb{N}^*$. Let $\mathbb{B}$ be a Brownian bridge and
$\mathbb{B}^{(p)}$ the process defined by
$$
\mathbb{B}^{(p)}(u):=\frac{1}{p!}\,u^p\,\mathbb{B}(u)
\quad\text{for}\quad u\in[0,1].
$$
Let $\Phi(\cdot)$ be a  functional defined
on  the space $\mathcal{D}(\mathbb{R})$, satisfying a Lipschitz condition
$$
|\Phi(v)-\Phi(w)|\leq L \sup_{t\in\mathbb{R}} |v(t)-w(t)|.
$$
Assume further that the distribution of the r.v. $\Phi\big(\mathbb{B}^{(p)}(F(\cdot))\big)$
 has a bounded density. Then, as $n\to\infty$,
\begin{equation}\label{estimationfunctional}
\sup_{x\in\mathbb{R}}\left|\mathbb{P}\!\left\{\Phi\big(\alpha_n^{(p)}(\cdot)\big)
\leq x\right\} -\mathbb{P}\!\left\{\Phi\big(\mathbb{B}^{(p)}(F(\cdot))\big)\leq x\right\}\right|
=\mathcal{O}\!\left(\frac{\log n}{\sqrt{n}}\right)\!.
\end{equation}
\end{corollary}
%%%%%%%%%%%%%%%%%%%%
For more comments on this kind of results, we may refer to \cite{Kokoszka2000},
Corollary 1.1 and p.~2459.

By applying (\ref{estimation-a}) to $x=c'\log n$ for a suitable constant $c'$
and appealing to Borel-Cantelli lemma, one can obtain the following almost sure
approximation of the process $\big\{\alpha_n^{(p)}:n\in\mathbb{N}^*\big\}$ based on a
sequence of Brownian bridges.
%%%%%%%%%%%%%%%%%%%%
\begin{corollary}\label{corollaryapprox}
The following bound holds, with probability~$1$, as $n\to\infty$:
\begin{equation}\label{approx}
\sup_{t\in\mathbb{R}} \left|\alpha_n^{(p)}(t)- \mathbb{B}_n^{(p)}(F(t))\right|
=\mathcal{O}\!\left(\frac{\log n}{\sqrt{n}}\right)\!.
\end{equation}
\end{corollary}
%%%%%%%%%%%%%%%%%%%%

The next result yields an almost sure approximation for
$\big\{\alpha_n^{(p)}:n\in\mathbb{N}^*\big\}$ based on a Kiefer process.
%%%%%%%%%%%%%%%%%%%%
\begin{theorem}\label{kieferapproximation}
On a suitable probability space, we may define the $p$-fold integrated
empirical process $\big\{\alpha_n^{(p)}:n\in\mathbb{N}^*\big\}$, in combination with
a Kiefer process $\{\mathbb{K}(s,u): s\geq 0, u\in[0,1]\}$, such that,
with probability~$1$, as $n\to\infty$,
$$
\max_{1\leq k\leq n}\sup_{t\in\mathbb{R}}\left|\sqrt{k}\,\alpha_k^{(p)}(t)
-\mathbb{K}^{(p)}(k,F(t))\right|=\mathcal{O}\!\left((\log n)^2\right)
$$
where $\mathbb{K}^{(p)}$ is the process defined by
$$
\mathbb{K}^{(p)}(s,u):=\frac{1}{p!}\,u^p\,\mathbb{K}(s,u)
\quad\text{for}\quad s\geq 0,\,u\in[0,1].
$$
\end{theorem}
%%%%%%%%%%%%%%%%%%%%
Let us mention that the ``extracted'' Kiefer process $\{\mathbb{K}(n,u):
n\in\mathbb{N}^*, u\in[0,1]\}$ may be viewed as the partial sums process of a
sequence of independent Brownian bridges $\left\{\mathbb{B}_i:i\in\mathbb{N}^*\right\}$:
$$
\mathbb{K}(n,u)=\sum_{i=1}^n \mathbb{B}_i(u)
\quad\text{for}\quad n\in\mathbb{N}^*,\, u\in[0,1].
$$

From Theorem~\ref{kieferapproximation}, we deduce the following law of iterated
logarithm (``a.s.'' stands for ``almost surely'').
%%%%%%%%%%%%%%%%%%%%
\begin{corollary}\label{corollarylli}
We have the following law of iterated logarithm for the $p$-fold integrated
empirical process:
\begin{equation}\label{lli}
\limsup_{n\to\infty}\frac{\sup_{t\in\mathbb{R}}\big|\alpha_n^{(p)}(t)\big|}{\sqrt{\log \log n}}
=\frac{(p+1/2)^{p+1/2}}{p!\,(p+1)^{p+1}}\quad\text{a.s.}
\end{equation}
\end{corollary}
%%%%%%%%%%%%%%%%%%%%

As a direct application of (\ref{approx}) and (\ref{lli}) to the problem of
goodness-of-fit, for testing the null hypothesis
$$
\mathcal{H}_0: F=F_0,
$$
we can use the following statistics: the \emph{$p$-fold integrated
Kolmogorov-Smirnov statistic}
$$
\mathbf{S}_n^{(p)}:=\sup_{t\in\mathbb{R}}
\left|\sqrt{n}\left(\mathbb{F}_n^{(p)}(t)-F_0^{(p)}(t)\right)\right|
$$
as well as the \emph{$p$-fold integrated Cram\'er-von Mises statistic}
$$
\mathbf{T}_n^{(p)}:=n\int_{\mathbb{R}} \left(\mathbb{F}_n^{(p)}(t)
-F_0^{(p)}(t)\right)^{\!2} dF_0(t).
$$
%%%%%%%%%%%%%%%%%%%%
\begin{corollary}\label{ecooooo}
Under $\mathcal{H}_0$, with probability~$1$, as $n\to\infty$, we have
\begin{align}
\left|\mathbf{S}_n^{(p)}-\sup_{t\in\mathbb{R}}\big|\mathbb{B}_n^{(p)}(F_0(t))\big|\right|
&
=\mathcal{O}\!\left(\frac{\log n}{\sqrt{n}}\right)\!,
\label{kol}\\[1ex]
\left|\mathbf{T}_n^{(p)}-\int_{\mathbb{R}}\big[\mathbb{B}_n^{(p)}(F_0(t))\big]^2\,dF_0(t)\right|
&
=\mathcal{O}\!\left(\!\sqrt{\frac{\log\log n}{n}}\,\log n\right)\!.
\label{vonmises}
\end{align}
\end{corollary}
%%%%%%%%%%%%%%%%%%%%

We finish this part by pointing out the possibility of considering the statistics, for $r\geq 1$,
$$
\omega_{n,p,r}=\sqrt{n}\left(\int_{\mathbb{R}}
\left|\mathbb{F}_n^{(p)}(t)-F_0^{(p)}(t)\right|^r dF_0(t)\right)^{\!1/r}\!.
$$
It is clear, however, that we have the following convergence in distribution as $n\to \infty$,
under $\mathcal{H}_{0}$:
$$
\sqrt{n}\left(\int_{\mathbb{R}}
\left|\mathbb{F}_n^{(p)}(t)-F_0^{(p)}(t)\right|^r dF_0(t)\right)^{\!1/r}\longrightarrow
\left(\int_{\mathbb{R}}\left|\mathbb{B}^{(p)}(F_0(t))\right|^rdF_0(t)\right)^{\!1/r}\!.
$$
Denoting by $X_{(1)} \leq X_{(2)}\leq  \cdots \leq X_{(n)}$ the order statistics of
$X_{1},X_{2},\dots,X_{n}$, and putting $X_{(0)} = -\infty, X_{(n+1)} =+\infty$,
straightforward manipulations of integrals yield the alternative representation,
for $p=0$ and $r\geq 1$,
\begin{align*}
\omega_{n,0,r}=&\;\frac{\sqrt{n}}{(r+1)^{1/r}}\left\{\sum_{i=1}^{n+1}\left[\left(F_0(X_{(i)})
-\frac{i-1}{n}\right)\left|F_0(X_{(i)})-\frac{i-1}{n}\right|^{r}\right.\right.\\
&-\left.\left.\left(F_0(X_{(i-1)})-\frac{i-1}{n}\right)\left|F_0(X_{(i-1)})
-\frac{i-1}{n}\right|^{r}\right]\right\}^{1/r}.
\end{align*}
In a future research, it would be of interest to deeply investigate such statistics.

%%%%%%%%%%%%%%%%%%%%%%%%%%%%%%%%%%%%%%%%%%%%%%%%%%%%%%%%%%%%%%%%%%%%%%%%%%%%%%%
\section{Weighted approximations}\label{section2-1}
%%%%%%%%%%%%%%%%%%%%%%%%%%%%%%%%%%%%%%%%%%%%%%%%%%%%%%%%%%%%%%%%%%%%%%%%%%%%%%%
In this part, we consider the weighted difference
$\big|\beta_n^{(p)}(u)-\mathbb{B}_n^{(p)}(u)\big|/\big(u(1-u)\big)^{1/2-\nu}$
for a power $\nu\in[0,1/2)$. Since this quantity is not defined at $0$ and $1$,
we will work with the supremum on an interval of the form $[d/n,1-d/n]$ for
$d,n\in\mathbb{N}^*$ such that $d\leq n/2, {n\geq 2}$. The weighted approximations
for the classical uniform empirical process
were deeply studied in the papers by
\cite{CsCsHoMa1986},  \cite{HaeuslerMaosn1999}, \cite{Mason1991},
\cite{Mason1999} and \cite{MasonvanZwet1987}.
For more details on the subject we may refer to  \cite{Csorgoho1993}.

We state here the analogue of Theorem 1.2 of \cite{Mason1999}.
%%%%%%%%%%%%%%%%%%%%
\begin{theorem}\label{theweithed1}
Fix $p\in\mathbb{N}^*$. On a suitable probability space, we may define the
$p$-fold uniform integrated empirical process $\big\{\beta_n^{(p)}:n\in\mathbb{N}^*\big\}$,
in combination with a sequence of Brownian bridges
$\left\{\mathbb{B}_n:n\in\mathbb{N}^*\right\}$ and a positive constant $C_p'$,
such that, for every $\nu\in[0,1/2)$, there exist a positive constant
$B_{p,\nu}$ for which we have, for any $d,n\in\mathbb{N}^*$ satisfying
$d\leq n/2$ and large enough $x$,
\begin{equation}
\mathbb{P}\bigg\{\sup_{u\in[d/n,1-d/n]}\frac{{n^{\nu}}\big|\beta_n^{(p)}(u)
-\mathbb{B}_n^{(p)}(u)\big|}{\big(u(1-u)\big)^{{1/2-}\nu}}\geq x\bigg\}
\leq B_{p,\nu} \sum_{k=2}^{p+1} \exp\!\left(-C_p'\,x^{2/k}d^{2\nu/k}n^{1-(1+2\nu)/k}\right)\!.
\label{ineg-Delta1}
\end{equation}
\end{theorem}
%%%%%%%%%%%%%%%%%%%%
This theorem may be proved by appealing to Theorem~\ref{lem2b}.
Applying (\ref{ineg-Delta1}) with $x=c''(\log n)/n^{1/2-\nu}$
for a suitable constant $c''$ yields the result below.
%%%%%%%%%%%%%%%%%%%%
\begin{corollary}\label{corollary-O}
On the same probability space of Theorem \ref{theweithed1}, we have,
for any $\nu\in[0,1/2)$ and any $d\in\mathbb{N}^*$, with probability $1$,
as $n\to\infty$,
$$
\sup_{u\in[d/n,1-d/n]}\frac{\big|\beta_n^{(p)}(u)-\mathbb{B}_n^{(p)}(u)\big|}{\big(u(1-u)\big)^{\nu}}
=\mathcal{O}\!\left(\frac{\log n}{n^{1/2-\nu}}\right)\!.
$$
\end{corollary}
%%%%%%%%%%%%%%%%%%%%
By using Theorem \ref{theweithed1} we have the following proposition.
\begin{corollary}
On the same probability space of Theorem \ref{theweithed1}, for all $0\leq \nu < 1/2$ there exists a constant $\gamma>0 $ such that
$$
\sup_{n\geq 2}\mathbb{E}\left\{\exp\left(\gamma \sup_{u\in[1/n,1-1/n]}\frac{{n^{\nu}}\big|\beta_n^{(p)}(u)
-\mathbb{B}_n^{(p)}(u)\big|}{\big(u(1-u)\big)^{{1/2-}\nu}}\right)\right\}<\infty.
$$

\end{corollary}

%%%%%%%%%%%%%%%%%%%%%%%%%%%%%%%%%%%%%%%%%%%%%%%%%%%%%%%%%%%%%%%%%%%%%%%%%%%%%%%
\section{The two-sample problem}\label{sectiontwo}
%%%%%%%%%%%%%%%%%%%%%%%%%%%%%%%%%%%%%%%%%%%%%%%%%%%%%%%%%%%%%%%%%%%%%%%%%%%%%%%

For each $m,n\in\mathbb{N}^*$, let $X_1,\ldots, X_m$ and $Y_1,\ldots,Y_n$
be independent random samples from continuous d.f.'s $F$ and $G$, respectively,
and let $\mathbb{F}_m^{(p)}$ and $\mathbb{G}_n^{(p)}$ denote their $p$-fold
integrated empirical d.f.'s. Tests for the null hypothesis
$$
\mathcal{H}_0': F=G,
$$
can be based on the \emph{$p$-fold integrated two-sample empirical process}
defined, for each $m,n\in\mathbb{N}^*$, by
$$
\boldsymbol{\xi}_{m,n}^{(p)}(t):=\sqrt{\frac{mn}{m+n}}
\left(\mathbb{F}_m^{(p)}(t)-\mathbb{G}_n^{(p)}(t)\right)
\quad\text{for}\quad t\in\mathbb{R}.
$$
%%%%%%%%%%%%%%%%%%%%

Actually, as in \cite{BouzebdaElfaouzy2012}, we will more generally consider
the following \emph{modified $p$-fold integrated two-sample empirical process}
(which includes the process $\boldsymbol{\xi}_{m,n}^{(p)}$). Fix a positive
integer~$q$ which will serve as a power. We define, for each $m,n\in\mathbb{N}^*$,
$$
\boldsymbol{\xi}_{m,n}^{(p,q)}(t):=\sqrt{\frac{mn}{m+n}}
\left[\left(\mathbb{F}_m^{(p)}(t)\right)^{\!q}-\left(\mathbb{G}_n^{(p)}(t)\right)^{\!q}\right]
\quad\text{for}\quad t\in\mathbb{R}.
$$
Set also, for any $m,n\in\mathbb{N}^*$,
$$
\varphi(m,n):=\max\left(\frac{\log m}{\sqrt{m}},\frac{\log n}{\sqrt{n}}\right)
\quad\text{and}\quad\phi(m,n):=\max\!\left(\sqrt{\frac{\log\log m}{m}}\,\log m,
\sqrt{\frac{\log\log n}{n}}\,\log n\right)\!.
$$

Reasonable statistics for testing $\mathcal{H}_0'$ would be the
\emph{modified $p$-fold integrated Kolmogorov-Smirnov statistic}
$$
\mathbf{S}_{m,n}^{(p,q)}:=\sup_{t\in\mathbb{R}}
\big|\boldsymbol{\xi}_{m,n}^{(p,q)}(t)\big|
$$
and the \emph{modified $p$-fold integrated Cram\'er-von Mises statistic}
$$
\mathbf{T}_{m,n}^{(p,q)}:=\int_{\mathbb{R}} \boldsymbol{\xi}_{m,n}^{(p,q)}(t)^2\,dF(t).
$$

The following results are consequences of Corollary~\ref{corollaryapprox}.
%%%%%%%%%%%%%%%%%%%%
\begin{corollary}\label{ximn}
On a suitable probability space, it is possible to define
$\big\{\boldsymbol{\xi}_{m,n}^{(p,q)}:m,n\in\mathbb{N}^*\big\}$,
jointly with two sequences of Brownian bridges $\big\{\mathbb{B}_m^1:
m\in\mathbb{N}^*\big\}$ and $\big\{\mathbb{B}_n^2:n\in\mathbb{N}^*\big\}$,
such that, under $\mathcal{H}_0'$, with probability~$1$, as $\min(m,n) \to\infty$,
$$
\sup_{t\in\mathbb{R}} \left|\boldsymbol{\xi}_{m,n}^{(p,q)}(t)
-\mathbb{B}_{m,n}^{(p,q)}(t)\right| =\mathcal{O}(\varphi(m,n)),
$$
where, for each $m,n\in\mathbb{N}^*$, $\mathbb{B}_{m,n}^{(p,q)}$ is
the Gaussian process defined by
$$
\mathbb{B}_{m,n}^{(p,q)}(t):=\frac{(p+1)q}{(p+1)!^q}\,F(t)^{pq+q-1}
\left(\sqrt{\frac{n}{m+n}}\,\mathbb{B}_m^1(F(t))-\sqrt{\frac{m}{m+n}}\,
\mathbb{B}_n^2(F(t))\right)\quad\text{for}\quad t\in\mathbb{R}.
$$
\end{corollary}
%%%%%%%%%%%%%%%%%%%%

%%%%%%%%%%%%%%%%%%%%
\begin{corollary}
Under $\mathcal{H}_0'$, with probability~$1$, as $\min(m,n) \to\infty$, we have
\begin{align*}
\left|\mathbf{S}_{m,n}^{(p,q)}-\sup_{t\in\mathbb{R}}
\big|\mathbb{B}_{m,n}^{(p,q)}(t)\big|\right|
&
=\mathcal{O}(\varphi(m,n)),
\\[1ex]
\left|\mathbf{T}_{m,n}^{(p,q)}-\int_{\mathbb{R}} \mathbb{B}_{m,n}^{(p,q)}(t)^2\,dF(t)\right|
&
=\mathcal{O}(\phi(m,n)).
\end{align*}
\end{corollary}
%%%%%%%%%%%%%%%%%%%%

%%%%%%%%%%%%%%%%%%%%
\begin{remark}
The family of statistics indexed by $q$ may be used to maximize the power of the
statistical test for a specific alternative hypothesis as argued in \cite{Ahmad2000}
in the case $p=1$.
\end{remark}
%%%%%%%%%%%%%%%%%%%%

Now, we fix a positive integer $K$ and we describe the more general $K$-sample problem.
For each $k\in\{1,\dots,K\}$, we consider a setting made of independent observations
$\big\{X_i^k:i\in\{1,\dots,n_k\}\big\}$ of a real-valued r.v. $X^k$.
The d.f.'s of $X_i^k$, $i\in\{1,\ldots,n_k\}$, are denoted by $F^k$ and they
are assumed to be continuous. We would like to test, $F_0$ being a fixed continuous
d.f., the null hypothesis
$$
\mathcal{H}_0^K: F^1= F^2=\dots=F^K=F_0.
$$
For any $K$-tuple of positive integers $\boldsymbol{n}=(n_1,\dots,n_K)$, set
$\boldsymbol{|n|}=\sum_{k=1}^K n_k$ and let
$$
(Z_1,\ldots,Z_{\boldsymbol{|n|}}):=\big(X_1^1,\ldots,X_{n_1}^1,X_1^2,\ldots,
X_{n_2}^2,\ldots,X_1^K,\ldots,X_{n_K}^K\big)
$$
be the pooled sample of total size $\boldsymbol{|n|}$,
$\mathbb{D}_{K,\boldsymbol{n}}^{(p)}$ be the $p$-fold integrated empirical
d.f. based upon $Z_1,\ldots,Z_{\boldsymbol{|n|}}$, and, for each
$k\in\{1,\dots,K\}$, $\mathbb{F}_{n_k}^{k,(p)}$ be the $p$-fold integrated
empirical d.f. based upon $X_1^k,\ldots,X_{n_k}^k$. Of course,
we have the following identity:
\begin{equation}\label{DD}
\mathbb{D}_{K,\boldsymbol{n}}^{(p)}=\frac{1}{\boldsymbol{|n|}}
\sum_{k=1}^K n_k\,\mathbb{F}_{n_k}^{k,(p)}.
\end{equation}

Next, we define the \emph{$p$-fold integrated $K$-sample empirical process} in the
following way: for any $K$-tuple $\boldsymbol{n}=(n_1,\dots,n_K)\in(\mathbb{N}^*)^K$,
$$
\boldsymbol{\xi}_{K,\boldsymbol{n}}^{(p)}(t):=\sum_{k=1}^K n_k
\left(\mathbb{F}_{n_k}^{k,(p)}(t)-\mathbb{D}_{K,\boldsymbol{n}}^{(p)}(t)\right)^{\!2}
\quad\text{for}\quad t\in\mathbb{R}.
$$
Obvious candidates for testing Hypothesis $\mathcal{H}_0^K$ are the
\emph{$p$-fold integrated $K$-sample Kolmogorov-Smirnov statistic}
$$
\mathbf{S}_{K,\boldsymbol{n}}^{(p)}:=\sup_{t\in\mathbb{R}}
\boldsymbol{\xi}_{K,\boldsymbol{n}}^{(p)}(t)
$$
and the \emph{$p$-fold integrated $K$-sample Cram\'er-von Mises functional}
(the usual square being included in the definition of $\boldsymbol{\xi}_{K,\boldsymbol{n}}^{(p)}$)
$$
\mathbf{T}_{K,\boldsymbol{n}}^{(p)}:=\int_{\mathbb{R}}
\boldsymbol{\xi}_{K,\boldsymbol{n}}^{(p)}(t)\,dF_0(t).
$$
Set
$$
\phi_K(\boldsymbol{n}):=\max_{1\leq k\leq K}
\left\{\sqrt{\frac{\log\log n_k}{n_k}}\,\log n_k \right\}\!.
$$

As a consequence of Corollary~\ref{corollaryapprox} and by using similar arguments
to those used in \cite{Bouzebdaetal2011}, we obtain the following results.
%%%%%%%%%%%%%%%%%%%%
\begin{theorem}\label{theoremK}
On a suitable probability space, it is possible to define
$\big\{\boldsymbol{\xi}_{K,\boldsymbol{n}}^{(p)}:\boldsymbol{n}\in(\mathbb{N}^*)^K\big\}$,
jointly with $K$ sequences of Brownian bridges $\big\{\mathbb{B}_m^k:
m\in\mathbb{N}^*\big\}$, $k\in\{1,\dots,K\}$, such that, under $\mathcal{H}_0^K$, with
probability~$1$, for $\boldsymbol{n}=(n_1,\dots,n_K)$ such that $\min_{1\leq k\leq K}n_k\to\infty$,
$$
\sup_{t\in\mathbb{R}} \left| \boldsymbol{\xi}_{K,\boldsymbol{n}}^{(p)}(t)
-\mathbb{B}_{K,\boldsymbol{n}}^{(p)}(t)\right|=\mathcal{O}\big(\phi_K(\boldsymbol{n})\big),
$$
where, for each $\boldsymbol{n}=(n_1,\dots,n_K)\in(\mathbb{N}^*)^K$,
$\mathbb{B}_{K,\boldsymbol{n}}^{(p)}$ is the process defined by
$$
\mathbb{B}_{K,\boldsymbol{n}}^{(p)}(t):=\frac{F_0(t)^{2p}}{p!^2}
\!\left[\sum_{k=1}^K\mathbb{B}_{n_k}^k\!\big(F_0(t)\big)^2
-\left(\sum_{k=1}^K\sqrt{\frac{n_k}{\boldsymbol{|n|}}}\,
\mathbb{B}_{n_k}^k\!\big(F_0(t)\big)\right)^{\!\!2}\right]
\quad\text{for}\quad t\in\mathbb{R}.
$$
\end{theorem}
%%%%%%%%%%%%%%%%%%%%
In the particular case $K=2$ (i.e. the two-sample problem), the corresponding
settings are related to the previous ones according as
$$
\boldsymbol{\xi}_{2,(n_1,n_2)}^{(p)}(t)=\left(\boldsymbol{\xi}_{n_1,n_2}^{(p)}(t)\right)^{\!2},
\quad \phi_2((n_1,n_2))=\phi(n_1,n_2),
\quad \mathbf{S}_{2,(n_1,n_2)}^{(p)}=\left(\mathbf{S}_{n_1,n_2}^{(p,1)}\right)^{\!2},
\quad \mathbf{T}_{2,(n_1,n_2)}^{(p)}=\mathbf{T}_{n_1,n_2}^{(p,1)}.
$$
Notice that $\mathbb{B}_{K,\boldsymbol{n}}^{(p)}(t)\geq 0$ for any $t\in\mathbb{R}$
and $\boldsymbol{n}\in(\mathbb{N}^*)^K$ as it is easily seen with the aid of
the Cauchy-Schwarz inequality. For $K=2$, this process writes
$\mathbb{B}_{2,(n_1,n_2)}^{(p)}=\left(\mathbb{B}_{n_1,n_2}^{(p,1)}\right)^{\!2}$,
this is the square of a Gaussian process.

The next result, which is an immediate consequence of the previous theorem
(observe that $\mathbf{S}_{K,\boldsymbol{n}}^{(p)}$ and
$\mathbf{T}_{K,\boldsymbol{n}}^{(p)}$ are bounded linear functionals of the
process $\boldsymbol{\xi}_{K,\boldsymbol{n}}^{(p)}$), gives the limit
null distributions of the statistics under consideration.
%%%%%%%%%%%%%%%%%%%%
\begin{corollary}
Under $\mathcal{H}_0^K$, with probability~$1$, for $\boldsymbol{n}=(n_1,\dots,n_K)$
such that $\min_{1\leq k \leq K} n_k \to\infty$, we have
\begin{align*}
\left|\mathbf{S}_{K,\boldsymbol{n}}^{(p)}-\sup_{t\in\mathbb{R}}
\mathbb{B}_{K,\boldsymbol{n}}^{(p)}(t)\right|
&
=\mathcal{O}\big(\phi_K(\boldsymbol{n})\big),
\\[1ex]
\left|\mathbf{T}_{K,\boldsymbol{n}}^{(p)}-\int_{\mathbb{R}}
\mathbb{B}_{K,\boldsymbol{n}}^{(p)}(t)\,dF_0(t)\right|
&
=\mathcal{O}\big(\phi_K(\boldsymbol{n})\big).
\end{align*}
\end{corollary}
%%%%%%%%%%%%%%%%%%%%

%%%%%%%%%%%%%%%%%%%%%%%%%%%%%%%%%%%%%%%%%%%%%%%%%%%%%%%%%%%%%%%%%%%%%%%%%%%%%%%
\section{The change-point problem}\label{sectionchnae}
%%%%%%%%%%%%%%%%%%%%%%%%%%%%%%%%%%%%%%%%%%%%%%%%%%%%%%%%%%%%%%%%%%%%%%%%%%%%%%%

Here and elsewhere, $\lfloor t\rfloor$ denotes the largest integer not
exceeding $t$. In many practical applications, we assume the structural
stability of statistical models and this fundamental assumption needs to be
tested before it can be applied. This is called the analysis of structural
breaks, or change-points, which has led to the development of a variety of
theoretical and practical results. For good sources of references to research
literature in this area along with statistical applications, the reader may
consult \cite{Brodsky1993}, \cite{Csorgo1997} and \cite{ChenGupta2000}.
For recent references on the subject we may refer, among many others, to
\cite{Bouzebda2012MMS}, \cite{Alexander2012}, \cite{Julian2013}, \cite{Lajos2014},
\cite{Bouzebda2014AN} and \cite{Bouzebda2014MMS}.

In this section, we deal with testing changes in d.f.'s for a sequence of
independent real-valued r.v.'s $X_1,\ldots,X_n$. The corresponding null
hypothesis that we want to test is
$$
\mathcal{H}_0'': X_1,\ldots,X_n \text{~~have d.f.~~}F.
$$
As frequently done, the behavior of the derived tests will be investigated
under the alternative hypothesis of a single change-point
$$
\mathcal{H}_1'': \exists ~k^*\in\{1,\ldots,n-1\}\text{~~such that~~}
X_1,\ldots, X_{k^*}\text{~~have d.f.~~}F \text{~~and~~}X_{k^*+1},\ldots,X_n\text{~~have d.f.~~}G.
$$
The d.f.'s $F$ and $G$ are assumed to be continuous. The critical integer
$k^*$ can be written as $\lfloor ns\rfloor$ for a certain $s\in(0,1)$. Then,
testing the null hypothesis $\mathcal{H}_0''$ can be based on functionals of
the following process: set, for each $n\in\mathbb{N}^*$,
\begin{equation}\label{alphatilde}
\widetilde{\alpha}_n^{(p)}(s,t):=\frac{\lfloor ns\rfloor(n-\lfloor ns\rfloor)}{n^{3/2}}
\left(\mathbb{F}_{\lfloor ns\rfloor}^{(p)-}(t)
-\mathbb{F}_{n-\lfloor ns\rfloor}^{(p)+}(t)\right)
\quad\text{for}\quad s\in(0,1),\,t\in\mathbb{R},
\end{equation}
where $\mathbb{F}_k^{(p)-}$ is the $p$-fold integrated empirical d.f. based upon
the $k$ first observations while $\mathbb{F}_{n-k}^{(p)+}$ is that based upon
the $(n-k)$ last ones. In (\ref{alphatilde}) we extend the definition of
$\mathbb{F}_k^{(p)-}$ and $\mathbb{F}_k^{(p)+}$ to the case where $k=0$ by
setting $\mathbb{F}_0^{(p)-}=\mathbb{F}_0^{(p)+}=0$, so that
$\widetilde{\alpha}_n^{(p)}(s,t)=0$ if $s\in(0,1/n)$.

We can define the r.v.'s $X_1,\ldots ,X_{\lfloor ns\rfloor}$ and
$X_{\lfloor ns\rfloor+1},\ldots ,X_n$ on a probability space on which we can
simultaneously construct two Kiefer processes
$\{\mathbb{K}_1(s,u): s\in\mathbb{R}, u\in[0,1]\}$ and
$\{\mathbb{K}_2(s,u): s\in\mathbb{R}, u\in[0,1]\}$ such that the ``restricted''
processes $\{\mathbb{K}_1(s,u): s\in[1,n/2], u\in[0,1]\}$ and
$\{\mathbb{K}_2(s,u): s\in[n/2,n], u\in[0,1]\}$ are independent. It turns out
that a natural approximation of
$\big\{\widetilde{\alpha}_n^{(p)}:n\in\mathbb{N}^*\big\}$
is given by the sequence of Gaussian processes
$\Big\{\overset{\text{\tiny o}}{\mathbb{K}}\vphantom{K}_n^{(p)}(s,F(t)) :
s\in[0,1], t\in\mathbb{R}, n\in\mathbb{N}^*\Big\}$ defined by
\begin{equation}\label{Knp}
\overset{\text{\tiny o}}{\mathbb{K}}\vphantom{K}_n^{(p)}(s,u)
:=\frac{1}{p!}\,u^p\,\overset{\text{\tiny o}}{\mathbb{K}}_n(s,u)
\quad\text{for}\quad s,u\in[0,1],\,n\in\mathbb{N}^*,
\end{equation}
where, for each $n\in\mathbb{N}^*$, $\overset{\text{\tiny o}}{\mathbb{K}}_n:=
\big\{\overset{\text{\tiny o}}{\mathbb{K}}_n(s,u):s,u\in[0,1]\big\}$
is the Gaussian process defined by
$$
\overset{\text{\tiny o}}{\mathbb{K}}_n(s,u):=\!
\begin{cases}
\frac{1}{\sqrt{n}}\big[\,\mathbb{K}_2(\lfloor ns\rfloor,u)
-s(\mathbb{K}_1(\lfloor n/2\rfloor,u)+\mathbb{K}_2(\lfloor n/2\rfloor,u))\big]
&
\!\!\! \text{for } s\in\!\big[0,\frac{1}{2}\big], u\in[0,1],
\\[1ex]
\frac{1}{\sqrt{n}}\big[\!-\mathbb{K}_1(\lfloor n(1-s)\rfloor,u)
+(1-s)(\mathbb{K}_1(\lfloor n/2\rfloor,u)+\mathbb{K}_2(\lfloor n/2\rfloor,u))\big]
&
\!\!\! \text{for } s\in\!\big[\frac{1}{2},1\big], u\in[0,1].
\end{cases}
$$
More precisely, we have the following result.
%%%%%%%%%%%%%%%%%%%%
\begin{theorem}\label{theorem1}
On a suitable probability space, it is possible to define
$\big\{\widetilde{\alpha}_n^{(p)}:n\in\mathbb{N}^*\big\}$,
jointly with a sequence of Gaussian processes
$\Big\{\overset{\text{\rm\tiny o}}{\mathbb{K}}\vphantom{K}_n^{(p)}:
n\in\mathbb{N}^*\Big\}$ as above,
such that, under $\mathcal{H}_0''$, with probability~$1$, as $n \to\infty$,
$$
\sup_{s\in (0,1)}\sup_{t\in\mathbb{R}}\left| \widetilde{\alpha}_n^{(p)}(s,t)
-\overset{\text{\rm\tiny o}}{\mathbb{K}}\vphantom{K}_n^{(p)}(s,F(t))\right|
=\mathcal{O}\!\left(\frac{(\log n)^2}{\sqrt{n}}\right)\!.
$$
\end{theorem}
%%%%%%%%%%%%%%%%%%%%
According to \cite{Cosorgo1997}, a way to test change-point is to use the
following statistics:
\begin{equation}\label{proempri}
\sigma_n^{(p)}:=\sup_{s\in (0,1)}\sup_{t\in\mathbb{R}}
\left|\widetilde{\alpha}_n^{(p)}(s,t)\right|.
\end{equation}

The corollary below is a consequence of Theorem~\ref{theorem1} which
can be proved by following exactly the same lines of \cite{Bouzebdaetal2015}.
%%%%%%%%%%%%%%%%%%%%
\begin{corollary}\label{coroltau}
If $\mathcal{H}_0''$ holds true, then we have the convergence in distribution,
as $n\to\infty$,
$$
\sigma_n^{(p)}\stackrel{\mathcal{L}}{\longrightarrow} \sup_{s,u\in[0,1]}
\left|\overset{\text{\rm\tiny o}}{\mathbb{K}}\vphantom{K}^{(p)}(s,u)\right|,
$$
where $\overset{\text{\rm\tiny o}}{\mathbb{K}}\vphantom{K}^{(p)}
=\big\{\overset{\text{\rm\tiny o}}{\mathbb{K}}\vphantom{K}^{(p)}(s,u):s,u\in[0,1]\big\}$
is a Gaussian process with mean zero and covariance
$$
\mathbb{E}\!\left(\overset{\text{\rm\tiny o}}{\mathbb{K}}\vphantom{K}^{(p)}(s,u)
\,\overset{\text{\rm\tiny o}}{\mathbb{K}}\vphantom{K}^{(p)}(s',u')\right)
=\frac{1}{p!^2}\,u^pu'^p(u\wedge u'-uu')(s\wedge s'-ss').
$$
\end{corollary}
%%%%%%%%%%%%%%%%%%%%
One has $\overset{\text{\tiny o}}{\mathbb{K}}\vphantom{K}^{(p)}(s,u) = \frac{1}{p!}\,u^p\,
\overset{\text{\tiny o}}{\mathbb{K}}(s,u)$ where
$\overset{\text{\tiny o}}{\mathbb{K}}$ is a tied-down Kiefer process.
We refer to \cite{Csorgo1997} for more details on the process
$\mathbb{K}^{(p)}$ in the case where $p=1$.

Actually, according to \cite{Cosorgo1997}, the most appropriate way to test
change-point is to use the following weighted statistic:
\begin{equation}\label{proempri2}
\sigma_{n,w}^{(p)}:=\sup_{s\in (0,1)}\sup_{t\in\mathbb{R}}
\frac{\left|\widetilde{\alpha}_n^{(p)}(s,t)\right|}{w\left(\lfloor ns\rfloor /n\right)}
\end{equation}
where $w$ is a positive function defined on $(0,1)$, increasing in a neighborhood
of zero and decreasing in a neighborhood of one satisfying the condition
$$
I(w,\boldsymbol{\varepsilon}):= \int_0^1\exp\left(-\frac{\boldsymbol{\varepsilon}
w^2(s)}{s(1-s)}\right)\,\frac{ds}{s(1-s)} <\infty
$$
for some constant $\boldsymbol{\varepsilon}>0$. For a history and further
applications of $I(w,\boldsymbol{\varepsilon})$, we refer to \cite{Csorgoho1993},
Chapter 4. From \cite{Szyszkowicz(1992)}, an example of such function $w$ is
given by
$$
w(t):=\left( t(1-t)\log \log \frac{1}{t(1-t)}\right)^{\!1/2}
\quad\text{for}\quad t\in(0,1).
$$
By using similar techniques to those which are developed in \cite{Csorgo1997},
one may show that
$$
\sigma_{n,w}^{(p)}\stackrel{\mathcal{L}}{\longrightarrow} \sup_{s,u\in[0,1]}
\frac{\left|\overset{\text{\rm\tiny o}}{\mathbb{K}}\vphantom{K}^{(p)}(s,u)\right|}{w(s)}.
$$
For more details, we refer to \cite{Bouzebda2014AN}.
%%%%%%%%%%%%%%%%%%%%
\begin{remark}
As in \cite{Szysz(1994)}, we mention that the statistic given by
(\ref{proempri}) should be more powerful for detecting changes that occur in
the middle, i.e., near $n/2$, where $k/n(1-k/n)$ reaches its maximum, than for
the ones occurring near the end points. The advantage of using the weighted
statistic defined in (\ref{proempri2}) is the detection of changes that occur
near the end points, while retaining the sensitivity to possible changes in
the middle as well.
\end{remark}
%%%%%%%%%%%%%%%%%%%%

We hope that the results presented in Sections~\ref{sectiontwo} and
\ref{sectionchnae} will be the prototypes of other various applications.

%%%%%%%%%%%%%%%%%%%%%%%%%%%%%%%%%%%%%%%%%%%%%%%%%%%%%%%%%%%%%%%%%%%%%%%%%%%%%%%
\section{Strong approximation of the integrated empirical process when
parameters are estimated}\label{Section5}
%%%%%%%%%%%%%%%%%%%%%%%%%%%%%%%%%%%%%%%%%%%%%%%%%%%%%%%%%%%%%%%%%%%%%%%%%%%%%%%

In this section, we are interested in the strong approximation of the
integrated empirical process when parameters are estimated. Our approach is in
the same spirit of \cite{Burke-Csorgo}. Let us introduce, for each
$n\in\mathbb{N}^*$, the \emph{$p$-fold integrated estimated empirical process}
$\widehat{\alpha}_n^{(p)}$:
\begin{equation}\label{alphabarchap}
\widehat{\alpha}_n^{(p)}(t):=\sqrt{n}\left(\mathbb{F}_n^{(p)}(t)
-F^{(p)}\big(t,\widehat{\boldsymbol{\theta}}_n\big)\right)
\quad\text{for}\quad t\in\mathbb{R},
\end{equation}
where $\big\{\widehat{\boldsymbol{\theta}}_n: n\in\mathbb{N}^*\big\}$ is a
sequence of estimators of a parameter $\boldsymbol{\theta}$ from a family of
d.f.'s $\{ F(t,\boldsymbol{\theta}): t\in\mathbb{R}, \,\boldsymbol{\theta}
\in\boldsymbol{\Theta}\}$ ($\boldsymbol{\Theta}$ being a subset of
$\mathbb{R}^d$ and $d$ a fixed positive integer) related to a sequence of
i.i.d. r.v.'s $\{X_i:i\in\mathbb{N}^*\}$. Let us mention that a general study
of the weak convergence of the estimated empirical process was carried out by
\cite{Durbin1973}. For a more recent reference, we may refer to \cite{Genz2006}
where the authors investigated the empirical processes with estimated
parameters under auxiliary information and provided some results regarding the
bootstrap in order to evaluate the limiting laws.

Let us introduce some notations.
%%%%%%%%%%%%%%%%%%%%
\begin{enumerate}[label=(\thesection.\arabic*)]
\item
The transpose of a vector $V$ of $\mathbb{R}^d$ will be denoted by $V^\top$.
\item
The norm $\| \cdot\|$ on $\mathbb{R}^d$ is defined by
$$
\| (y_1,\ldots,y_d)\|:=\max_{1\leq i \leq d}|y_i|.
$$
\item
For a function $(t,\boldsymbol{\theta})\mapsto g(t,\boldsymbol{\theta})$ where
$\boldsymbol{\theta}=(\theta_1, \ldots,\theta_d)\in \mathbb{R}^d$,
$\nabla_{\boldsymbol{\theta}}g(t,\boldsymbol{\theta}_0)$
denotes the vector in $\mathbb{R}^d$ of partial derivatives
$\big((\partial g/\partial\theta_1) (t,\boldsymbol{\theta}),
\ldots,(\partial g/\partial\theta_d) (t,\boldsymbol{\theta}))\big)$
evaluated at $\boldsymbol{\theta}=\boldsymbol{\theta}_0$,
and $\nabla_{\boldsymbol{\theta}}^2g(t,\boldsymbol{\theta})$
denotes the $d\times d$ matrix of second order partial derivatives
$\big((\partial^2g/\partial\theta_i\partial\theta_j) (t,\boldsymbol{\theta}))
\big)_{1\leq i,j\leq d}$.
\item
For a vector-valued function $x\mapsto V(x)=(v_1(x),\dots,v_d(x))$ defined on $\mathbb{R}$,
$\int V$ denotes the vector
$$\left(\int_{\mathbb{R}} v_1(x_{1})\,dx_{1},\dots,\int_{\mathbb{R}} v_d(x_{d})\,dx_{d}\right).$$
\end{enumerate}
%%%%%%%%%%%%%%%%%%%%

Next, we write out the set of all conditions (those of \cite{Burke-Csorgo})
which we will use in the sequel.
%%%%%%%%%%%%%%%%%%%%
\begin{enumerate}[label=(\roman*)]
\item
The estimator $\widehat{\boldsymbol{\theta}}_n$ admits the following form:
for each $n\in\mathbb{N}^*$,
$$
\sqrt{n}\left(\widehat{\boldsymbol{\theta}}_n-\boldsymbol{\theta}_0\right)
=\frac{1}{\sqrt{n}}\sum_{i=1}^n l(X_i,\boldsymbol{\theta}_0)
+\boldsymbol{\varepsilon}_n,
$$
where $\boldsymbol{\theta}_0$ is the theoretical true value of
$\boldsymbol{\theta}$, $l(\cdot,\boldsymbol{\theta}_0)$ is a measurable
$d$-dimensional vector-valued function, and $\boldsymbol{\varepsilon}_n$
converges to zero as $n\to\infty$ in a manner to be specified later on.
Notice that
$$
\frac{1}{\sqrt{n}}\sum_{i=1}^n l(X_i,\boldsymbol{\theta}_0)
=\sqrt{n}\int_{-\infty}^{\infty} l(s,\boldsymbol{\theta}_0)\,d\mathbb{F}_n(s).
$$

\item
The mean value of $l(X_i,\boldsymbol{\theta}_0)$ vanishes:
$$\mathbb{E}\!\left(l(X_i,\boldsymbol{\theta}_0)\right)=0.$$
\item
The matrix $M(\boldsymbol{\theta}_0):=\mathbb{E}\!
\left(l(X_i,\boldsymbol{\theta}_0)^\top l(X_i,\boldsymbol{\theta}_0)\right)$
is a finite nonnegative definite $d\times d$ matrix.
\item
The vector-valued function $(t,\boldsymbol{\theta})\mapsto
\nabla_{\boldsymbol{\theta}}F(t,\boldsymbol{\theta})$ is uniformly
continuous in $t\in\mathbb{R}$ and $\boldsymbol{\theta} \in\mathbf{V}$, where
$\mathbf{V}$ is the closure of a given neighborhood of $\boldsymbol{\theta}_0$.
\item
Each component of the vector-valued function $t\mapsto
l(t,\boldsymbol{\theta}_0)$ is of bounded variation in $t$ on each finite interval of $\mathbb{R}$.
\item
The vector-valued function $t\mapsto\nabla_{\boldsymbol{\theta}}
F(t,\boldsymbol{\theta}_0)$ is uniformly bounded in
$t\in\mathbb{R}$, and the vector-valued function
$(t,\boldsymbol{\theta})\mapsto \nabla^2_{\boldsymbol{\theta}}F(t,\boldsymbol{\theta})$
is uniformly bounded in $t\in\mathbb{R}$ and $\boldsymbol{\theta} \in\mathbf{V}$.
\item
Set
$$
\ell(s,\boldsymbol{\theta}_0):=l\!\left(F^{-1}(s,\boldsymbol{\theta}_0),\boldsymbol{\theta}_0\right)
\quad\text{for}\quad s\in(0,1)
$$
where
$$
F^{-1}(s,\boldsymbol{\theta}_0)=\inf\{ t\in\mathbb{R}:F(t,\boldsymbol{\theta}_0)\geq s\}.
$$
The limiting relations below hold:
$$
\lim_{s\searrow 0} \sqrt{s \log \log (1/s)}\,
\left\|\ell(s,\boldsymbol{\theta}_0)\right\| =0
$$
and
$$
\lim_{s\nearrow 1} \sqrt{(1-s) \log \log [1/(1-s)]}\,
\left\|\ell(s,\boldsymbol{\theta}_0)\right\| =0,
$$
\item
Set
$$
\ell'_s(s,\boldsymbol{\theta}_0):=\frac{\partial\ell}{\partial s}(s,\boldsymbol{\theta}_0)
\quad\text{for}\quad s\in(0,1).
$$
The partial derivative $\ell'_s(s,\boldsymbol{\theta}_0)$ exist for every $s\in (0,1)$ and
the bounds below hold: there is a positive constant $C$ such that
$$
s\left\|\ell'_s(s,\boldsymbol{\theta}_0)\right\| \leq C \quad\text{for all }
s\in\big(0,\textstyle{\frac{1}{2}}\big)
$$
and
$$
(1-s)\left\|\ell'_s(s,\boldsymbol{\theta}_0)\right\| \leq C \quad\text{for all }
s\in\big(\textstyle{\frac{1}{2}},1\big).
$$
\end{enumerate}
%%%%%%%%%%%%%%%%%%%%

Now, we state an analogous result to Theorem 3.1 of \cite{Burke-Csorgo}. For each
$n\in\mathbb{N}^*$, let $\{G_n(t):t\in\mathbb{R}\}$ be the process defined by
\begin{align*}
G_n(t)
&
:=\frac{1}{\sqrt{n}}\left( \mathbb{K}(n,F(t,\boldsymbol{\theta}_0))
- \left(\int_{\mathbb{R}} l(s,\boldsymbol{\theta}_0)\,
d_s\mathbb{K}(n,F(s,\boldsymbol{\theta}_0))\right)
\nabla_{\boldsymbol{\theta}}F(t,\boldsymbol{\theta}_0)^\top\right)
\\
&
=\frac{1}{\sqrt{n}}\left(\mathbb{K}(n,F(t,\boldsymbol{\theta}_0))
- \mathbf{W}(n)\nabla_{\boldsymbol{\theta}}F(t,\boldsymbol{\theta}_0)^\top\right)
\quad\text{for}\quad t\in\mathbb{R},
\end{align*}
where we set
$$
\mathbf{W}(\tau):=\int_{\mathbb{R}} l(s,\boldsymbol{\theta}_0)
\,d_s\mathbb{K}(\tau,F(s,\boldsymbol{\theta}_0))
\quad\text{for}\quad \tau\geq 0.
$$
The process $\{\mathbf{W}(\tau): \tau\geq 0\}$ is a $d$-dimensional Brownian
motion with a covariance matrix of rank that of $M(\boldsymbol{\theta}_0)$.
The estimated empirical process given by $\widehat{\alpha}_n^{(p)}(t)$ defined
by (\ref{alphabarchap}) will be approximated by the sequence of processes
$\big\{G_n^{(p)}:n\in\mathbb{N}^*\big\}$ defined by
\begin{equation}\label{Gnp}
G_n^{(p)}(t):=\frac{1}{p!}\,F(t,\boldsymbol{\theta}_0)^p\,G_n(t)
\quad\text{for}\quad t\in\mathbb{R},
\end{equation}
as described in the next theorem. Set
\begin{equation}\label{epsilon}
\boldsymbol{\varepsilon}_n^{(p)}:=\sup_{t\in\mathbb{R}}\left|
\widehat{\alpha}_n^{(p)}(t) -G_n^{(p)}(t)\right|\!.
\end{equation}
%%%%%%%%%%%%%%%%%%%%
\begin{theorem}\label{theoremBurkeetal}
Suppose that the sequence of estimators $\big\{\widehat{\boldsymbol{\theta}}_n:
n\in\mathbb{N}^*\big\}$ satisfies conditions {\rm (i), (ii)} and {\rm (iii)}.
Then, as $n\to\infty$,
%%%%%%%%%%%%%%%%%%%%
\begin{enumerate}[label=(\alph*)]
\item
$\boldsymbol{\varepsilon}_n^{(p)}\stackrel{\mathbb{P}}{\longrightarrow} 0$
if Conditions {\rm (iv), (v)} hold and $\boldsymbol{\varepsilon}_n
\stackrel{\mathbb{P}}{\longrightarrow} 0$;
\item
$\boldsymbol{\varepsilon}_n^{(p)}\stackrel{\text{a.s.}}{\longrightarrow} 0$
if Conditions {\rm (vi)--(viii)} hold and $\boldsymbol{\varepsilon}_n
\stackrel{\text{a.s.}}{\longrightarrow} 0$;
\item
$\boldsymbol{\varepsilon}_n^{(p)}=\mathcal{O}(\max(h(n),n^{-\boldsymbol{\epsilon}}))$
for some $\boldsymbol{\epsilon}>0$ if Conditions {\rm(vi)--(viii)} hold and
$\boldsymbol{\varepsilon}_n=\mathcal{O}(h(n))$ for some function $h$ satisfying $h(n)>0$
and $h(n)\to 0.$
\end{enumerate}
%%%%%%%%%%%%%%%%%%%%
\end{theorem}
%%%%%%%%%%%%%%%%%%%%

The limiting Gaussian process $G_n^{(p)}$ of Theorem~\ref{theoremBurkeetal}
depends crucially on $F$ and also on the true theoretical value
$\boldsymbol{\theta}_0$. In general, Theorem~\ref{theoremBurkeetal} cannot be
used to test the composite hypothesis :
$$
F\in\{ F(t,\boldsymbol{\theta}): t\in\mathbb{R}, \boldsymbol{\theta}\in\boldsymbol{\Theta}\}.
$$
In order to circumvent this problem, \cite{Burke-Csorgo} proposed an
approximate solution, they introduce another process:
$$
\widehat{G}_n(t):=\frac{1}{\sqrt{n}}\left( \mathbb{K}\big(n,F\big(t,
\widehat{\boldsymbol{\theta}}_n\big)\big)
- \mathbf{W}(n)\nabla_{\boldsymbol{\theta}}F\big(t,
\widehat{\boldsymbol{\theta}}_n\big)^\top\right)\!.
$$
Under some regularity conditions, \cite{Burke-Csorgo} show that (see Theorem~3.2
therein), as $n\to\infty$,
$$
\sup_{t\in\mathbb{R}}\left|\widehat{G}_n(t)-G_n(t)\right|
\stackrel{\mathbb{P}}{\longrightarrow }0.
$$
Setting $\widehat{G}_n^{(p)}(t) := \frac{1}{p!}\,F\big(t,
\widehat{\boldsymbol{\theta}}_n\big)^p\,\widehat{G}_n(t)$,
one can show that, as $n\to\infty$,
\begin{equation}\label{eqreferrrz11}
\sup_{t\in\mathbb{R}}\left|\widehat{G}_n^{(p)}(t)-G_n^{(p)}(t)\right|
\stackrel{\mathbb{P}}{\longrightarrow }0.
\end{equation}
Consequently, we have, as $n\to\infty$,
$$
\sup_{t\in\mathbb{R}}\left| \widehat{\alpha}_n^{(p)}(t)
-\widehat{G}_n^{(p)}(t)\right| \stackrel{\mathbb{P}}{\longrightarrow }0.
$$

%%%%%%%%%%%%%%%%%%%%%%%%%%%%%%%%%%%%%%%%%%%%%%%%%%%%%%%%%%%%%%%%%%%%%%%%%%%%%%%
\section{Local time of the integrated empirical process}\label{section343}
%%%%%%%%%%%%%%%%%%%%%%%%%%%%%%%%%%%%%%%%%%%%%%%%%%%%%%%%%%%%%%%%%%%%%%%%%%%%%%%

In this section, we are mainly concerned with the behavior of the local time of
the $p$-fold integrated empirical process. This behavior can be characterized
by using a representation that expresses the integrated empirical process in
terms of a partial sums process, see (\ref{sj}) below. Let us recall
the definition of the process $\beta_n$ given in (\ref{uniformprocesse}) and
let us introduce the modified \textit{$p$-fold integrated uniform empirical process}
$\widetilde{\beta}_n^{(p)}$ defined, for each $n\in\mathbb{N}^*$, by
\begin{align*}
\widetilde{\beta}_n^{(p)}(u):=&\;\int_0^u dv_1\int_0^{v_1} dv_{p-1} \dots
\int_0^{v_{p-1}} \beta_n(v_p)\,dv_p
\\
=&\;\frac{1}{(p-1)!} \int_0^u (u-v)^{p-1}\,\beta_n(v)\,dv\quad\text{for}\quad u\in[0,1].
\end{align*}

In this part, we fo cus on the particular r.v. $\mathcal{A}_n^{(p)}:=\widetilde{\beta}_n^{(p)}(1)$.
It is easily seen that the representation below holds:
$$
\mathcal{A}_n^{(p)}=\sqrt{n}\Bigg(\frac{1}{n}
\sum_{i=1}^n\frac{1}{p!}\left(1-U_i\right)^p
-\frac{1}{(p+1)!} \Bigg)=\frac{S_n^{(p)}}{\sqrt{n}}
$$
where $\big\{S_n^{(p)}:n\in\mathbb{N}^*\big\}$ is the following partial sums
process where the summands are i.i.d. r.v.'s with mean zero:
\begin{equation}\label{sj}
S_n^{(p)}:=\frac{1}{p!}\sum_{i=1}^n\left(\left(1-U_i\right)^p-\frac{1}{p+1}\right)\!.
\end{equation}
This is a random walk with continuously distributed jumps. In the particular
case where $p=1$, we retrieve the representation provided by
\cite{HenzeNikitin2002} p.~185, namely
$$
\mathcal{A}_n^{(1)}=\frac{S_n^{(1)}}{\sqrt{n}}\quad\text{with}\quad
S_n^{(1)}:=\sum_{i=1}^n\left(\frac{1}{2}-U_i \right)\!.
$$
Notice that we are dealing with a sum of strongly non-lattice r.v.'s as, i.e.,
in p.~210 of \cite{BassKhoshnevisan1993}. Indeed, we easily check that the
characteristic function $\chi^{(p)}$ of the $\left(1-U_i\right)^p-1/(p+1)$'s,
namely
$$
\chi^{(p)}(z):=\int_0^1 \exp(\mathrm{i}z(u^p-1/(p+1)))\,du
=\frac{\exp(-\mathrm{i}z/(p+1))}{pz^{1/p}} \int_0^z \exp(
\mathrm{i}v)\frac{dv}{v^{1-1/p}}
$$
satisfies the conditions
$$
\forall z\in\mathbb{R}^*,\,\big|\chi^{(p)}(z)\big|< 1\quad\text{and}\quad
\limsup_{|z|\to\infty}\big|\chi^{(p)}(z)\big|< 1.
$$

Next, we fix a neighborhood $I$ of $0$, e.g., $I=[-1/2,1/2]$, and we define the
local time
\begin{equation}\label{localtime}
\lambda^{(p)}(x,n):=\sum_{i=1}^n\mathbbm{1}_I\big(S_i^{(p)}-x\big)
\quad\text{for}\quad x\in\mathbb{R},\, n\in\mathbb{N}^*.
\end{equation}
The local time $\lambda^{(p)}(x,n)$ represents the number of visits of the
random walk $\big\{S_n^{(p)}:n\in\mathbb{N}^*\big\}$ in the neighborhood
$x+I$ of $x$ up to discrete time $n$.
Our aim is to obtain the rate of the approximation
of the self-intersection local time
$$
L_n^{(p)}(t):=\sum_{1\leq i<j \leq \lfloor nt\rfloor }\int_{\mathbb{R}}
\mathbbm{1}_{I}\big(S_i^{(p)}-x\big)\,\mathbbm{1}_{I}\big(S_j^{(p)}-x\big)\,dx
$$
by the integrated local time of some standard Wiener process.
The quantity $L_n^{(p)}(t)$ enumerates in a certain manner the couples $(i,j)$
of distinct and ordered indices up to time $\lfloor nt\rfloor$ such that
$S_i-S_j$ is less than the diameter of $I$.

To this aim, we recall that, if $\{\mathbb{W}(t): t\geq 0\}$ is the standard Wiener
process with $\mathbb{W}(0)=0$, then its local time process
$\{l(x,t): t \geq 0, x\in\mathbb{R}\}$ is defined as
\begin{equation}\label{Lb}
l(x,t):=\lim_{\boldsymbol{\varepsilon} \searrow  0}\frac{1}{2\boldsymbol{\varepsilon}}
\int_0^t\mathbbm{1}_{\left\{ x-\boldsymbol{\varepsilon}< \mathbb{W}(s) < x
+\boldsymbol{\varepsilon} \right\}}\,ds\quad\text{for}\quad x\in\mathbb{R},\,t \geq 0.
\end{equation}

Following exactly the same lines of \cite{Bouzebdaetal2015}, we can prove
the two following results.
%%%%%%%%%%%%%%%%%%%%
\begin{theorem}
We have, with probability~$1$, as $n \to\infty$,
$$
\sup_{t\in[0,1]} \left|L_n^{(p)}(t)
-\frac{1}{2}\,n^{3/2}\int_{\mathbb{R}} l_n(x,t)^2\,dx\right|
=\mathcal{O}\!\left(n^{5/4}(\log n)^{1/2}(\log \log n)^{1/4}\right)\!,
$$
where $l_n$ is the normalized local time
$$
l_n(x,t):=\frac{1}{\sqrt{n}}\,l\!\left(\sqrt{n}\,x,\lfloor nt\rfloor\right)\!.
$$
\end{theorem}
%%%%%%%%%%%%%%%%%%%%
%
%%%%%%%%%%%%%%%%%%%%
\begin{corollary}
We have, with probability $1$, for any $t\in(0,1]$, there exist two positive
constants $\kappa_1$ and $\kappa_2$ such that, almost surely, for large enough $n$,
$$
\kappa_1\,\frac{\lfloor nt\rfloor^{3/2}}{\sqrt{\log\log n}}\leq L_n^{(p)}(t)
\leq \kappa_2\,\lfloor nt\rfloor^{3/2}\sqrt{\log\log n}.
$$
In particular, for any $t\in(0,1]$, almost surely, as $n \to\infty$,
$$
L_n^{(p)}(t)=\frac{1}{2}\,\lfloor nt\rfloor^{3/2+o(1)}.
$$
\end{corollary}
%%%%%%%%%%%%%%%%%%%%

\section{Simulation results for testing the uniformity }\label{simulation}
In this section, series of experiments are conducted in order to examine the performance
of the proposed statistical tests. More precisely, we have undertaken numerical
illustrations regarding the power of these statistical tests in finite sample situations.
The computing program codes are implemented in \texttt{R}. In this section, we considered four uniformity tests for different values of $p=0,1,2,3$,
by making use of the \emph{$p$-fold integrated
Kolmogorov-Smirnov statistic}
$$
\mathbf{S}_n^{(p)}:=\sup_{t\in\mathbb{R}}
\left|\sqrt{n}\left(\mathbb{F}_n^{(p)}(t)-F_0^{(p)}(t)\right)\right|.
$$
Recall that for $p=0$, the statistic $\mathbf{S}_n^{(p)}$ is  the classical Kolmogorov Smirnov statistic.
The simulations involve
random samples of size $n = 10,20,40,100$; they are drawn from the underlying distributions below,
for each $p \in\{0,1, 2, 3\}$ and based on 10000 replications.
In power comparison, we considered the following distribution functions $F(\cdot)$ as alternatives.
\begin{align*}
A_{k} : F(x)&=1-(1-x)^{k} \quad\mbox{for }0\leq x\leq  1,\quad k=1.5, 2;\\
B_{k} : F(x)&=\begin{cases}
    2^{k-1}x^{k}  & \mbox{for }\;  0\leq x<0.5, \\
   1-2^{k-1}(1-x)   & \mbox{for }\; 0.5\leq x\leq  1,\end{cases}\quad k=1.5,2,3;\\
C_{k} : F(x)&=\begin{cases}
    0.5-2^{k-1}(0.5-x)^{k}  &\mbox{for }\;   0\leq x<0.5 ,\\
   0.5+2^{k-1}(0.5-x)^{k}   & \mbox{for }\;  0.5\leq x\leq  1,\end{cases}\quad k=1.5,2,3.
\end{align*}
These alternatives were used by \cite{Stephens1974} in his study of power comparisons of several
tests for uniformity. According to Stephens, alternative $A$ gives points closer to zero than
expected under the hypothesis of uniformity. Alternative $B$ gives points near 0.5 and alternative
$C$ gives two points close to 0 and 1. Also, these alternatives were used by \cite{Dudewicz1981}
and \cite{Alizadeh2017} in their study of power comparisons of some uniformity tests.
The densities of the alternatives $A_{k}$,   $B_{k}$, and $C_{k}$ are depicted in the
Figure~\ref{alternative} below.

\begin{figure}[!ht]
\begin{center}
\centerline{\includegraphics[width=9cm]{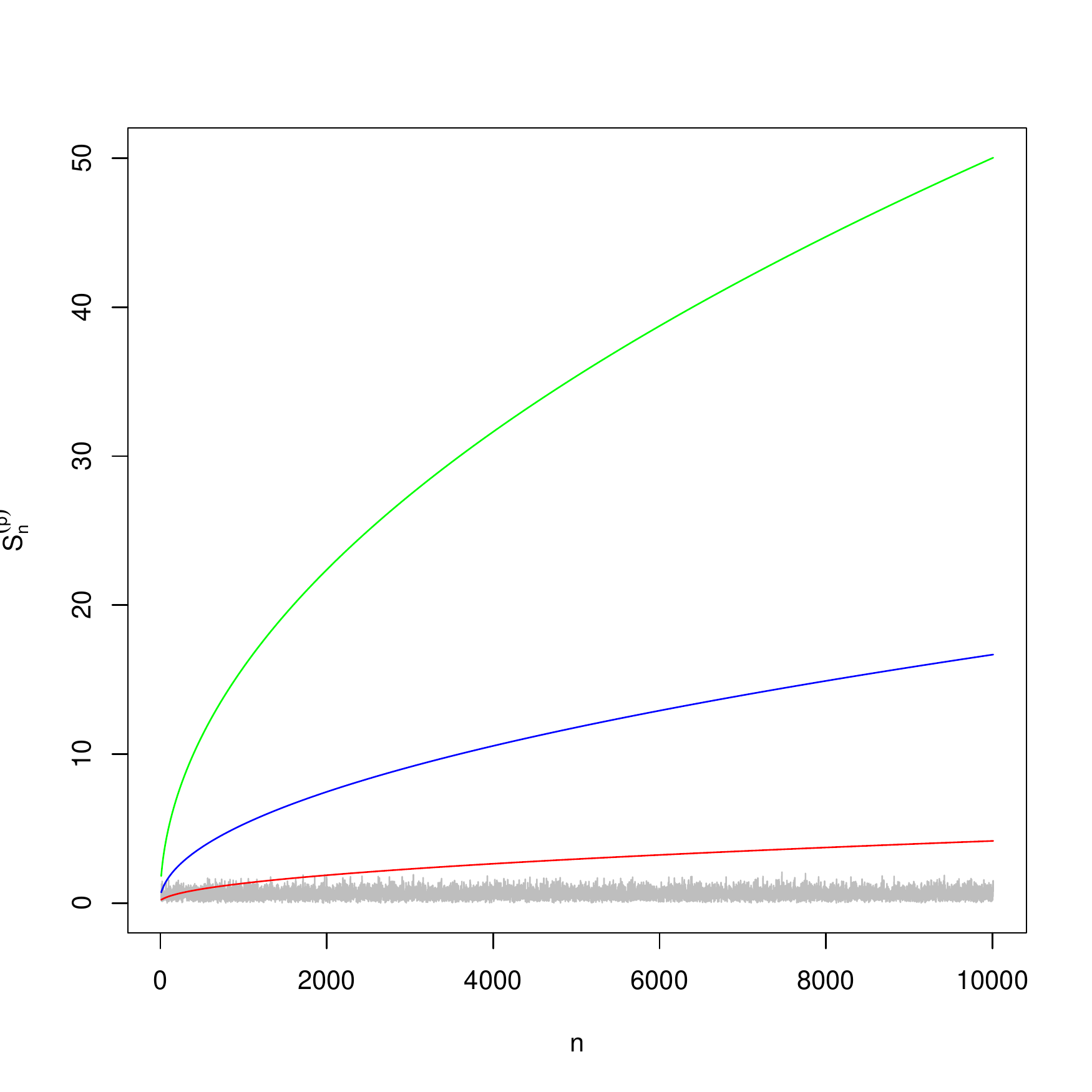}}
\end{center}
\caption{The statistics $\mathbf{ S}_{n}^{(p)}$ $p=0,1,2,3$ as functions of the sample size.
(``{\textcolor{gray}{--}}'': $\mathbf{ S}_{n}^{(0)}$ ), (``{\textcolor{green}{--}}'':  $\mathbf{ S}_{n}^{(1)}$),
(``{\textcolor{blue}{--}}'':  $\mathbf{ S}_{n}^{(2)}$), (``{\textcolor{red}{--}}'':  $\mathbf{ S}_{n}^{(3)}$).}
\label{normallocation2}
\end{figure}

\begin{figure}[!ht]
\begin{center}
\centerline{\includegraphics[width=17cm]{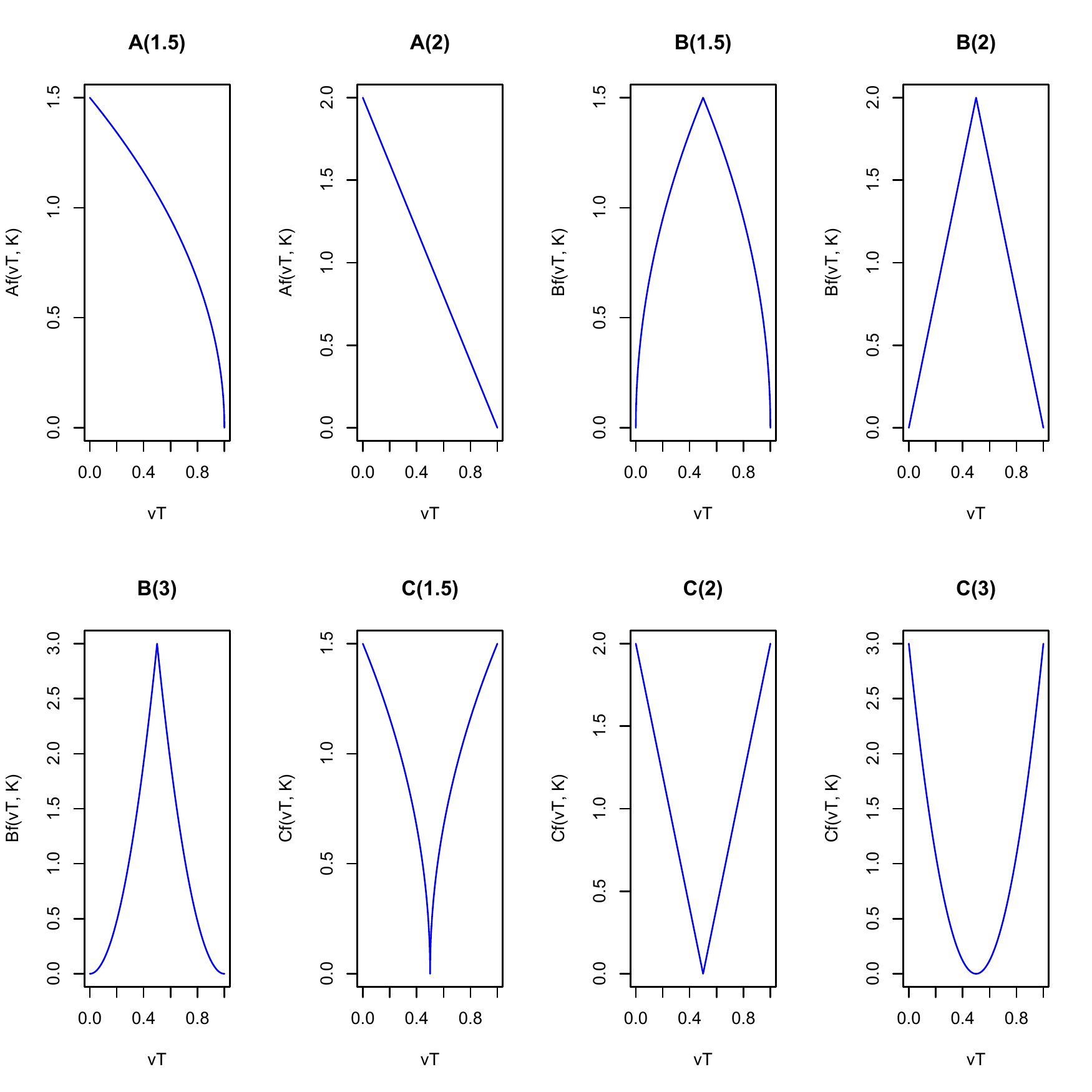}}
\end{center}
\caption{Densities of $A_{k}$, $B_{k}$, and $C_{k}$ families} \label{normallocation2zer}
\label{alternative}\end{figure}

%%%%%%%%%%%%%%%%%%%%%%%%%%%%%%%%%%%%%%%%%%%%%%%%%%%%%%%%%%%%%%%%%%%%%%%%%%%%%%%
\FloatBarrier
\section*{\underline{$n=10$}}

\begin{enumerate}
  \item
\begin{table}[!ht]
\caption{Powers (in percentages) of goodness of fit tests for the Uniform distribution on (0,1) for $n =10$
and $\alpha=0.01$}
\begin{center}

\begin{tabular}{lclclcl}
  \hline
  $n$ & Samples & Level \\
     \hline
  10 & 10000 & 0.01 \\
   \hline
\end{tabular}
~~\\[1ex]

\begin{tabular}{rlclclclclclcl}
  \hline
 & Alternative & $\mathbf{S}_n^{(0)}$ & $\mathbf{S}_n^{(1)}$ & $\mathbf{S}_n^{(2)}$ & $\mathbf{S}_n^{(3)}$ \\
  \hline
  & $A_{1.5}$ &  4 & 7 & 7 & 0  \\
  & $A_{2}$ &  15 & 24 & 21 & 0 \\
  & $B_{1.5}$ & 0 & 1 & 1 & 0  \\
  & $B_{2}$ & 0 & 3 & 3 & 0 \\
  & $B_{3}$ &  0 & 7 & 9 & 0\\
  & $C_{1.5}$ & 3 & 1 & 0 & 3 \\
  & $C_{2}$ &  6 & 1 & 1 & 8 \\
  & $C_{3}$ & 15 & 4 & 5 & 19  \\

   \hline
\end{tabular}
\end{center}
\end{table}

  \item
\begin{table}[!ht]
\caption{Powers (in percentages) of goodness of fit tests for the Uniform distribution on (0,1) for $n =10$
and $\alpha=0.05$}
\begin{center}

\begin{tabular}{lclclcl}
  \hline
  $n$ & Samples & Level \\
     \hline
  10 & 10000 & 0.05 \\
   \hline
\end{tabular}
~~\\[1ex]

\begin{tabular}{rlclclclclclcl}
  \hline
 & Alternative & $\mathbf{S}_n^{(0)}$ & $\mathbf{S}_n^{(1)}$ & $\mathbf{S}_n^{(2)}$ & $\mathbf{S}_n^{(3)}$ \\
  \hline
  & $A_{1.5}$ &   15 & 25 & 22 & 0  \\
  & $A_{2}$ &  38 & 54 & 48 & 0 \\
  & $B_{1.5}$ & 3 & 9 & 11 & 0 \\
  & $B_{2}$ & 4 & 17 & 20 & 0 \\
  & $B_{3}$ &  8 & 36 & 44 & 0 \\
  & $C_{1.5}$ & 11 & 5 & 5 & 13\\
  & $C_{2}$ &  19 & 8 & 9 & 22\\
  & $C_{3}$ & 36 & 18 & 21 & 38 \\
   \hline
\end{tabular}
\end{center}\label{Tab2}
\end{table}

  \item
\begin{table}[!ht]
\caption{Powers (in percentages) of goodness of fit tests for the Uniform distribution on (0,1) for $n =10$
and $\alpha=0.10$}
\begin{center}

\begin{tabular}{lclclcl}
  \hline
  $n$ & Samples & Level \\
     \hline
  10 & 10000 & 0.10 \\
   \hline
\end{tabular}
~~\\[1ex]

\begin{tabular}{rlclclclclclcl}
  \hline
 & Alternative & $\mathbf{S}_n^{(0)}$ & $\mathbf{S}_n^{(1)}$ & $\mathbf{S}_n^{(2)}$ & $\mathbf{S}_n^{(3)}$ \\
  \hline
  & $A_{1.5}$ &  25 & 37 & 32 & 0  \\
  & $A_{2}$ &   51 & 68 & 61 & 0\\
  & $B_{1.5}$ &  8 & 18 & 19 & 3\\
  & $B_{2}$ & 12 & 31 & 32 & 1\\
  & $B_{3}$ &  24 & 58 & 61 & 0 \\
  & $C_{1.5}$ & 19 & 11 & 11 & 23\\
  & $C_{2}$ &  30 & 16 & 16 & 32 \\
  & $C_{3}$ & 53 & 29 & 32 & 50\\
   \hline
\end{tabular}

\end{center}
\end{table}

\end{enumerate}

%%%%%%%%%%%%%%%%%%%%%%%%%%%%%%%%%%%%%%%%%%%%%%%%%%%%%%%%%%%%%%%%%%%%%%%%%%%%%%%
\FloatBarrier

\section*{\underline{$n=20$}}

\begin{enumerate}
  \item
\begin{table}[!ht]
\caption{Powers (in percentages) of goodness of fit tests for the Uniform distribution on (0,1) for $n =20$
and $\alpha=0.01$}
\begin{center}

\begin{tabular}{lclclcl}
  \hline
  $n$ & Samples & Level \\
     \hline
  20 & 10000 & 0.01 \\
   \hline
\end{tabular}
~~\\[1ex]

\begin{tabular}{rlclclclclclcl}
  \hline
 & Alternative & $\mathbf{S}_n^{(0)}$ & $\mathbf{S}_n^{(1)}$ & $\mathbf{S}_n^{(2)}$ & $\mathbf{S}_n^{(3)}$ \\
  \hline
  & $A_{1.5}$ &  9 & 15 & 14 & 0  \\
  & $A_{2}$ &  39 & 52 & 49 & 0 \\
  & $B_{1.5}$ &  0 & 3 & 4 & 0   \\
  & $B_{2}$ &  1 & 8 & 13 & 0 \\
  & $B_{3}$ &  6 & 33 & 44 & 0 \\
  & $C_{1.5}$ &  4 & 1 & 1 & 6  \\
  & $C_{2}$ &  10 & 3 & 3 & 15  \\
  & $C_{3}$ & 31 & 14 & 16 & 40  \\

   \hline
\end{tabular}
\end{center}
\end{table}

  \item \begin{table}[!ht]
\caption{Powers (in percentages) of goodness of fit tests for the Uniform distribution on (0,1) for $n =20$
and $\alpha=0.05$}
\begin{center}

\begin{tabular}{lclclcl}
  \hline
  $n$ & Samples & Level \\
     \hline
  20 & 10000 & 0.05 \\
   \hline
\end{tabular}
~~\\[1ex]

\begin{tabular}{rlclclclclclcl}
  \hline
 & Alternative & $\mathbf{S}_n^{(0)}$ & $\mathbf{S}_n^{(1)}$ & $\mathbf{S}_n^{(2)}$ & $\mathbf{S}_n^{(3)}$ \\
  \hline
  & $A_{1.5}$ &   28 & 42 & 37 & 0   \\
  & $A_{2}$ &   70 & 83 & 77 & 0  \\
  & $B_{1.5}$ & 6 & 15 & 16 & 0  \\
  & $B_{2}$ &  13 & 34 & 38 & 0  \\
  & $B_{3}$ & 42 & 74 & 78 & 0  \\
  & $C_{1.5}$ & 16 & 8 & 7 & 18 \\
  & $C_{2}$ &   31 & 17 & 17 & 36 \\
  & $C_{3}$ & 67 & 42 & 42 & 63 \\
   \hline
\end{tabular}
\end{center}
\end{table}

  \item
\begin{table}[!ht]
\caption{Powers (in percentages) of goodness of fit tests for the Uniform distribution on (0,1) for $n =20$
and $\alpha=0.10$}
\begin{center}

\begin{tabular}{lclclcl}
  \hline
  $n$ & Samples & Level \\
     \hline
  20 & 10000 & 0.10 \\
   \hline
\end{tabular}
~~\\[1ex]

\begin{tabular}{rlclclclclclcl}
  \hline
 & Alternative & $\mathbf{S}_n^{(0)}$ & $\mathbf{S}_n^{(1)}$ & $\mathbf{S}_n^{(2)}$ & $\mathbf{S}_n^{(3)}$ \\
  \hline
  & $A_{1.5}$ &   41 & 54 & 50 & 0   \\
  & $A_{2}$   &   81 & 89 & 86 & 0 \\
  & $B_{1.5}$ &   12 & 24 & 27 & 1\\
  & $B_{2}$   &   27 & 48 & 53 & 0 \\
  & $B_{3}$   &   67 & 86 & 89 & 0 \\
  & $C_{1.5}$ &   25 & 14 & 14 & 30\\
  & $C_{2}$   &   45 & 27 & 29 & 50 \\
  & $C_{3}$   &   82 & 55 & 56 & 75\\
   \hline
\end{tabular}
\end{center}
\end{table}

\end{enumerate}

%%%%%%%%%%%%%%%%%%%%%%%%%%%%%%%%%%%%%%%%%%%%%%%%%%%%%%%%%%%%%%%%%%%%%%%%%%%%%%%
\FloatBarrier
\section*{\underline{$n=40$}}

\begin{enumerate}
  \item
  \begin{table}[!ht]
\caption{Powers (in percentages) of goodness of fit tests for the Uniform distribution on (0,1) for $n =40$
and $\alpha=0.01$}
\begin{center}

\begin{tabular}{lclclcl}
  \hline
  $n$ & Samples & Level \\
     \hline
  40 & 10000 & 0.01 \\
   \hline
\end{tabular}
~~\\[1ex]

\begin{tabular}{rlclclclclclcl}
  \hline
 & Alternative & $\mathbf{S}_n^{(0)}$ & $\mathbf{S}_n^{(1)}$ & $\mathbf{S}_n^{(2)}$ & $\mathbf{S}_n^{(3)}$ \\
  \hline
  & $A_{1.5}$ &  25 & 37 & 35 & 0    \\
  & $A_{2}$   &  82 & 90 & 87 & 10  \\
  & $B_{1.5}$ &   1 & 7 & 10 & 0 \\
  & $B_{2}$   &   7 & 29 & 37 & 0 \\
  & $B_{3}$   &  49 & 83 & 89 & 6 \\
  & $C_{1.5}$ &   7 & 3 & 3 & 8 \\
  & $C_{2}$   &  24 & 13 & 13 & 27 \\
  & $C_{3}$   &  70 & 46 & 47 & 68\\
   \hline
\end{tabular}
\end{center}
\end{table}

  \item
  \begin{table}[!ht]
\caption{Powers (in percentages) of goodness of fit tests for the Uniform distribution on (0,1) for $n =40$
and $\alpha=0.05$}
\begin{center}

\begin{tabular}{lclclcl}
  \hline
  $n$ & Samples & Level \\
     \hline
  40 & 10000 & 0.05 \\
   \hline
\end{tabular}
~~\\[1ex]

\begin{tabular}{rlclclclclclcl}
  \hline
 & Alternative & $\mathbf{S}_n^{(0)}$ & $\mathbf{S}_n^{(1)}$ & $\mathbf{S}_n^{(2)}$ & $\mathbf{S}_n^{(3)}$ \\
  \hline
  & $A_{1.5}$ &  51 & 65 & 63 & 3    \\
  & $A_{2}$   &  95 & 98 & 97 & 28  \\
  & $B_{1.5}$ &  11 & 24 & 30 & 0\\
  & $B_{2}$   &  38 & 62 & 70 & 3\\
  & $B_{3}$   &  92 & 97 & 98 & 26  \\
  & $C_{1.5}$ &  24 & 13 & 14 & 27\\
  & $C_{2}$   &  56 & 36 & 37 & 55   \\
  & $C_{3}$   &  96 & 77 & 75 & 87\\
   \hline
\end{tabular}
\end{center}
\end{table}

  \item
  \begin{table}[!ht]
\caption{Powers (in percentages) of goodness of fit tests for the Uniform distribution on (0,1) for $n =40$
and $\alpha=0.10$}
\begin{center}

\begin{tabular}{lclclcl}
  \hline
  $n$ & Samples & Level \\
     \hline
  40 & 10000 & 0.10 \\
   \hline
\end{tabular}
~~\\[1ex]

\begin{tabular}{rlclclclclclcl}
  \hline
 & Alternative & $\mathbf{S}_n^{(0)}$ & $\mathbf{S}_n^{(1)}$ & $\mathbf{S}_n^{(2)}$ & $\mathbf{S}_n^{(3)}$ \\
  \hline
  & $A_{1.5}$ &   65 & 77 & 74 & 5  \\
  & $A_{2}$   &   98 & 99 & 98 & 40 \\
  & $B_{1.5}$ &   23 & 37 & 42 & 1 \\
  & $B_{2}$   &   60 & 75 & 80 & 6\\
  & $B_{3}$   &   98 & 99 & 99 & 43\\
  & $C_{1.5}$ &   35 & 21 & 21 & 40 \\
  & $C_{2}$   &   74 & 50 & 49 & 69\\
  & $C_{3}$   &   99 & 89 & 84 & 92\\
   \hline
\end{tabular}
\end{center}
\end{table}

\end{enumerate}

%%%%%%%%%%%%%%%%%%%%%%%%%%%%%%%%%%%%%%%%%%%%%%%%%%%%%%%%%%%%%%%%%%%%%%%%%%%%%%%
\FloatBarrier

\section*{\underline{$n=100$}}

\begin{enumerate}
  \item
  \begin{table}[!ht]
  \caption{Powers (in percentages) of goodness of fit tests for the Uniform distribution on (0,1) for $n =100$
and $\alpha=0.01$}
\begin{center}

\begin{tabular}{lclclcl}
  \hline
  $n$ & Samples & Level \\
     \hline
  100 & 10000 & 0.01 \\
   \hline
\end{tabular}
~~\\[1ex]

\begin{tabular}{rlclclclclclcl}
  \hline
 & Alternative & $\mathbf{S}_n^{(0)}$ & $\mathbf{S}_n^{(1)}$ & $\mathbf{S}_n^{(2)}$ & $\mathbf{S}_n^{(3)}$ \\
  \hline
 & $A_{1.5}$ &   70 & 83 & 81 & 35 \\
  & $A_{2}$   &  99 & 100 & 100 & 97  \\
  & $B_{1.5}$ &   7 & 25 & 33 & 6 \\
  & $B_{2}$   &  56 & 83 & 89 & 52 \\
  & $B_{3}$   &  99 & 100 & 100 & 99 \\
  & $C_{1.5}$ &  17 & 13 & 12 & 24\\
  & $C_{2}$   &  71 & 52 & 50 & 68 \\
  & $C_{3}$   &  99 & 97 & 95 & 98\\
   \hline
\end{tabular}
\end{center}
\end{table}

  \item
  \begin{table}[!ht]
  \caption{ Powers (in percentages) of goodness of fit tests for the Uniform distribution on (0,1) for $n =100$
and $\alpha=0.05$}
\begin{center}

\begin{tabular}{lclclcl}
  \hline
  $n$ & Samples & Level \\
     \hline
  100 & 10000 & 0.05 \\
   \hline
\end{tabular}
~~\\[1ex]

\begin{tabular}{rlclclclclclcl}
  \hline
 & Alternative & $\mathbf{S}_n^{(0)}$ & $\mathbf{S}_n^{(1)}$ & $\mathbf{S}_n^{(2)}$ & $\mathbf{S}_n^{(3)}$ \\
  \hline
  & $A_{1.5}$ &  91 & 96 & 94 & 58  \\
  & $A_{2}$   & 100 & 100 & 100 & 99  \\
  & $B_{1.5}$ &  34 & 54 & 61 & 17\\
  & $B_{2}$   &  94 & 97 & 98 & 76\\
  & $B_{3}$   & 100 & 100 & 100 & 99 \\
  & $C_{1.5}$ &  48 & 34 & 31 & 48\\
  & $C_{2}$   &  96 & 81 & 75 & 86 \\
  & $C_{3}$   & 100 & 99 & 99 & 99  \\
   \hline
\end{tabular}
\end{center}
\end{table}

  \item

\begin{table}[!ht]
  \caption{ Powers (in percentages) of goodness of fit tests for the Uniform distribution on (0,1) for $n =100$
and $\alpha=0.10$}
\begin{center}

\begin{tabular}{lclclcl}
  \hline
  $n$ & Samples & Level \\
     \hline
  100 & 10000 & 0.10 \\
   \hline
\end{tabular}
~~\\[1ex]

\begin{tabular}{rlclclclclclcl}
  \hline
 & Alternative & $\mathbf{S}_n^{(0)}$ & $\mathbf{S}_n^{(1)}$ & $\mathbf{S}_n^{(2)}$ & $\mathbf{S}_n^{(3)}$ \\
  \hline
  & $A_{1.5}$ & 95 & 98 & 97 & 70    \\
  & $A_{2}$   &100 & 100 & 100 & 99    \\
  & $B_{1.5}$ & 53 & 66 & 72 & 27 \\
  & $B_{2}$   & 98 & 98 & 99 & 86  \\
  & $B_{3}$   &100 & 100 & 100 & 100   \\
  & $C_{1.5}$ & 66 & 46 & 44 & 62\\
  & $C_{2}$   & 99 & 90 & 85 & 93  \\
  & $C_{3}$   &100 & 100 & 99 & 99  \\
   \hline
\end{tabular}
\end{center}
\end{table}

\end{enumerate}
%%%%%%%%%%%%%%%%%%%%%%%%%%%%%%%%%%%%%%%%%%%%%%%%%%%%%%%%%%%%%%%%%%%%%%%%%%%%%%%

We computed the power values of the proposed tests using Monte Carlo simulation and
next these values are compared with each other.
As it is expected, the power values of tests strongly depend on the type of alternatives.
Based on our simulation study, against alternative $A_{k}$, the proposed tests $\mathbf{S}_n^{(1)}$
and $\mathbf{S}_n^{(2)}$ have a good power while against alternative $B_{k}$ the tests have the
highest power compared with the classical Kolmogorov-Smirnov. The test based on $\mathbf{S}_n^{(3)}$
has a good power for alternative $C_{k}$. In order to extract methodological recommendations
for the use of the proposed statistics in this work, it would be interesting
to conduct extensive Monte Carlo experiments  to compare our procedures with other
alternatives presented in the literature, but this would go
far beyond the scope of the present paper.

%%%%%%%%%%%%%%%%%%%%%%%%%%%%%%%%%%%%%%%%%%%%%%%%%%%%%%%%%%%%%%%%%%%%%%%%%%%%%%%
\section{Mathematical developments}\label{appe}
%%%%%%%%%%%%%%%%%%%%%%%%%%%%%%%%%%%%%%%%%%%%%%%%%%%%%%%%%%%%%%%%%%%%%%%%%%%%%%%

This section is devoted to the proofs of our results. The previously
displayed notations continue to be used in the sequel.

%%%%%%%%%%%%%%%%%%%%%%%%%%%%%%%%%%%%%%%%%%%%%%%%%%%%%%%%%%%%%%%%%%%%%%%%%%%%%%%
\subsection{Some bounds for the empirical process, Brownian bridge and the
Kiefer process}
%%%%%%%%%%%%%%%%%%%%%%%%%%%%%%%%%%%%%%%%%%%%%%%%%%%%%%%%%%%%%%%%%%%%%%%%%%%%%%%

Let us immediately point out an obvious fact which will be used several times
thereafter:
$$
0\leq F(t)\leq 1\quad\text{and}\quad 0\leq\mathbb{F}_n(t)\leq 1
\quad\text{for}\quad t\in\mathbb{R},\,n\in\mathbb{N}^*
$$
which obviously entails that, for any $p\in\mathbb{N}$,
\begin{equation}\label{inequality}
0\leq F^{(p)}(t)\leq 1\quad\text{and}\quad 0\leq\mathbb{F}_n^{(p)}(t)\leq 1
\quad\text{for}\quad t\in\mathbb{R},\,n\in\mathbb{N}^*.
\end{equation}
Similarly, we have, for any $p\in\mathbb{N}$,
\begin{equation}\label{inequality-U}
0\leq\mathbb{U}_n^{(p)}(u)\leq 1 \quad\text{for}\quad u\in[0,1],\,n\in\mathbb{N}^*.
\end{equation}

We also mention some bounds that we will use further. By appealing to
Chung's law of the iterated logarithm for the empirical process, see
\cite{Chung1949}, which stipulates that
$$
\limsup_{n\to\infty} \frac{\sup_{t\in\mathbb{R}} |\alpha_n(t)|}
{\sqrt{\log \log n}}=\frac{1}{\sqrt{2}} \quad\text{a.s.},
$$
we see that, with probability~$1$, as $n\to\infty$,
\begin{equation}\label{empililaeae}
\sup_{t\in\mathbb{R}} |\alpha_n(t)|=\mathcal{O}\!\left(\!\sqrt{\log \log n}\,\right)\!.
\end{equation}
Moreover, by \cite{KMT1975}, on a suitable probability space, we can define
the uniform empirical process $\{\beta_n:n\in\mathbb{N}^*\}$, in combination with a
sequence of Brownian bridges
$\left\{\mathbb{B}_n:n\in\mathbb{N}^*\right\}$ together with a Kiefer process
$\{\mathbb{K}(s,u): s\geq 0, u\in[0,1]\}$, such that, with probability~$1$,
as $n\to\infty$,
\begin{equation}\label{approxbis}
\sup_{u\in[0,1]} \left|\beta_n(u)- \mathbb{B}_n(u)\right|
=\mathcal{O}\!\left(\frac{\log n}{\sqrt{n}}\right)
\end{equation}
and
$$
\max_{1\leq k\leq n}\sup_{u\in[0,1]}\left|\sqrt{k}\,\beta_k(u)
-\mathbb{K}(k,u)\right|=\mathcal{O}\!\left((\log n)^2\right)
$$
from which we extract,with probability~$1$, as $n\to\infty$,
\begin{equation}\label{equation11KIE}
\sup_{u\in[0,1]}\left|\beta_n(u)-\frac{1}{\sqrt{n}}\,\mathbb{K}(n,u)\right|
=\mathcal{O}\!\left(\frac{(\log n)^2}{\sqrt{n}}\right)\!.
\end{equation}
As a result, by putting (\ref{empililaeae}) into (\ref{approxbis}) and
(\ref{equation11KIE}), one derives the following bounds: with probability~$1$,
as $n\to\infty$,
\begin{equation}\label{estimbridge}
\sup_{u\in[0,1]} |\mathbb{B}_n(u)| =\mathcal{O}\!\left(\!\sqrt{\log \log n}\,\right)
\quad\text{and}\quad
\sup_{u\in[0,1]}|\mathbb{K}(n,u)| =\mathcal{O}\!\left(\!\sqrt{n \log \log n}\,\right)\!.
\end{equation}
Notice that the second bound in (\ref{estimbridge}) comes also from the law of
the iterated logarithm for the Kiefer process; see \cite{Csorgo1981REVESZ}, p.~81.

%%%%%%%%%%%%%%%%%%%%%%%%%%%%%%%%%%%%%%%%%%%%%%%%%%%%%%%%%%%%%%%%%%%%%%%%%%%%%%%
\subsection{Proof of Proposition~\ref{expFp}}
%%%%%%%%%%%%%%%%%%%%%%%%%%%%%%%%%%%%%%%%%%%%%%%%%%%%%%%%%%%%%%%%%%%%%%%%%%%%%%%

We begin by making an observation: because of the hypothesis that the df $F$ is continuous,
the sampled variables $X_1, X_2,\ldots,X_n$ are almost surely all different.
Then, we can define with probability $1$ the order statistics
$$
X_{1,n} < X_{2,n} < \cdots< X_{n,n}
$$
associated with $X_1, X_2,\ldots,X_n$. Notice that
the event $\big\{X_{i,n} \leq t\big\}$ is equal to $\{n\mathbb{F}_n(t) \leq i\}$.
Hence, we can write that, for any function $f$, with probability $1$,
\begin{equation}\label{general}
\int_{-\infty}^{t} f(s)\,d\mathbb{F}_n(s)
=\frac{1}{n} \sum_{i=1}^n f(X_i)\mathbbm{1}_{\{X_i\leq t\}}
=\frac{1}{n} \sum_{i=1}^{n\mathbb{F}_n(t)} f\!\left(X_{i,n}\right).
\end{equation}

Before proving (\ref{repFnp1}), we first show by induction that
\begin{equation}\label{rec}
\mathbb{F}_n^{(p)}(t)=\frac{1}{n^{p+1}} \# \left\{(i_1,\dots,i_{p+1})
\in\mathbb{N}^{p+1}: 1\leq i_1\leq \dots \leq i_{p+1}\leq n\mathbb{F}_{n}(t)\right\}\!.
\end{equation}
Of course, (\ref{rec}) holds for $p=0$. Pick now a positive integer $p$ and
suppose that
$$
\mathbb{F}_n^{(p-1)}(t)=\frac{1}{n^p}
\# \left\{(i_1,\dots,i_p)\in\mathbb{N}^p: 1\leq i_1\leq \dots \leq i_p
\leq n\mathbb{F}_{n}(t)\right\}\!.
$$
By Definition~\ref{def}, we see that the family of functions $\mathbb{F}_n^{(p)}$
can be recursively defined by $\mathbb{F}_n^{(0)}=\mathbb{F}_n$ and, for any
$p\in\mathbb{N}^*$ and any $t\in\mathbb{R}$, by
$$
\mathbb{F}_n^{(p)}(t)=\int_{-\infty}^{t}\mathbb{F}_n^{(p-1)}(s)\,d\mathbb{F}_n(s).
$$
Therefore, by (\ref{general}) and remarking that $\mathbb{F}_{n}(X_{i,n})=i/n$, a.s.,
\begin{align*}
\mathbb{F}_n^{(p)}(t)&=\frac{1}{n} \sum_{i=1}^n \mathbb{F}_n^{(p-1)}(X_i)
\mathbbm{1}_{\{X_i\leq t\}}
=\frac{1}{n} \sum_{i=1}^{n\mathbb{F}_n(t)} \mathbb{F}_n^{(p-1)}(X_{i,n})
\\
&=\frac{1}{n^{p+1}} \sum_{i=1}^{n\mathbb{F}_n(t)}
\# \left\{(i_1,\dots,i_p)\in\mathbb{N}^p: 1\leq i_1\leq \dots
\leq i_p \leq n\mathbb{F}_{n}(X_{i,n})\right\}
\\
&=\frac{1}{n^{p+1}} \sum_{i=1}^{n\mathbb{F}_n(t)}
\# \left\{(i_1,\dots,i_p)\in\mathbb{N}^p: 1\leq i_1\leq \dots \leq i_p \leq i\right\}
\\
&=\frac{1}{n^{p+1}} \# \left\{(i_1,\dots,i_p,i_{p+1})\in\mathbb{N}^{p+1}:
1\leq i_1\leq \dots \leq i_p \leq i_{p+1}\leq n\mathbb{F}_{n}(t)\right\}\!.
\end{align*}
Hence, (\ref{rec}) is valid for any $p\in\mathbb{N}$.

Now, we observe that the cardinality in (\ref{rec}) is nothing but the number
of combinations with repetitions of $p+1$ integers lying between $1$ and
$n\mathbb{F}_{n}(t)$, which coincides with the number of combinations without
repetition of $p+1$ integers lying between $1$ and $n\mathbb{F}_{n}(t)+p$.
This is the result concerning $\mathbb{F}_n^{(p)}$ announced in (\ref{repFnp1}).
Finally, the formula concerning $F^{(p)}$ can be easily obtained by induction too.
The proof of Proposition~\ref{expFp} is finished.
\hfill$\Box$

In the proposition below, we provide a representation of $\mathbb{F}_n^{(p)}$
by means of $\mathbb{F}_n$.
%%%%%%%%%%%%%%%%%%%%
\begin{proposition}\label{repFnp3}
The integrated empirical d.f.\/ $\mathbb{F}_n^{(p)}$ can be expressed by means of\/
$\mathbb{F}_n$ as follows: with probability~$1$,
\begin{equation}\label{repFnp2}
\mathbb{F}_n^{(p)}(t)=\frac{\mathbb{F}_n(t)^{p+1}}{(p+1)!}
+\sum_{k=1}^p a_k^{(p)}\,\frac{\mathbb{F}_n(t)^k}{n^{p-k+1}}
\quad\text{for}\quad t\in\mathbb{R},\,n\in\mathbb{N}^*,
\end{equation}
where the coefficients $a_k^{(p)}$, $1\leq k\leq p$, are positive integers.
\end{proposition}
%%%%%%%%%%%%%%%%%%%%
\begin{proof}
By expanding the combination in (\ref{repFnp1}), we get that, a.s.,
$$
\mathbb{F}_n^{(p)}(t)=\frac{1}{(p+1)!\,n^{p+1}} \, \prod_{i=0}^p (n\mathbb{F}_n(t)+i)
=\frac{\mathbb{F}_n(t)}{(p+1)!\,n^p} \, \prod_{i=1}^p (n\mathbb{F}_n(t)+i)
\quad\text{for}\quad t\in\mathbb{R},\,n\in\mathbb{N}^*.
$$
Since  $\prod_{i=1}^p (x+i)$ is a polynomial of degree $p$ with coefficients in $\mathbb{N}$,
(\ref{repFnp2}) immediately follows.
%
%Appealing to the classical expansion
%$$
%\prod_{i=1}^p (x+i)=\sum_{k=0}^p\Big[{p+1\atop k+1}\Big]\,x^k,
%$$
%where $\big[{\cdot\atop\cdot}\big]$ are the
%unsigned Stirling numbers of the first kind (see, e.g., http://en.wikipedia.org/wiki/Stirling\_number).
%%Riordan, J. (2002).  {\em Introduction to Combinatorial Analysis}. Dover Publications, Inc., Mineola, NY.
%We immediately derive (\ref{repFnp2}) by setting $a_k^{(p)}:=\big[{p+1\atop k+1}\big]$.
\end{proof}

In the proposition below, we rely $\alpha_n^{(p)}$ to $\alpha_n$.
%%%%%%%%%%%%%%%%%%%%
\begin{proposition}
The $p$-fold integrated empirical process $\alpha_n^{(p)}$ is related to
the empirical process $\alpha_n$ according to, with probability~$1$,
\begin{equation}\label{alphanp}
\alpha_n^{(p)}(t)=\frac{1}{p!}\,F(t)^p\,\alpha_n(t)
+\sum_{k=2}^{p+1} b_k^{(p)}\,\frac{F(t)^{p+1-k}}{n^{(k-1)/2}}\,\alpha_n(t)^k
+\sum_{k=1}^p a_k^{(p)}\,\frac{\mathbb{F}_n(t)^k}{n^{p-k+1/2}}
\quad\text{for}\quad t\in\mathbb{R}, \,n\in\mathbb{N}^*,
\end{equation}
where the coefficients $a_k^{(p)}$, $1\leq k\leq p$, are those of
Proposition~\ref{repFnp3} and the $b_k^{(p)}$, $2\leq k\leq p+1$, are positive real
numbers less than $1$.
Similarly,
\begin{equation}\label{betanp}
\beta_n^{(p)}(u)=\frac{1}{p!}\,u^p\,\beta_n(u)
+\sum_{k=2}^{p+1} b_k^{(p)}\,\frac{u^{p+1-k}}{n^{(k-1)/2}}\,\beta_n(u)^k
+\sum_{k=1}^p a_k^{(p)}\,\frac{\mathbb{U}_n(u)^k}{n^{p-k+1/2}}
\quad\text{for}\quad u\in[0,1], \,n\in\mathbb{N}^*.
\end{equation}
\end{proposition}
%%%%%%%%%%%%%%%%%%%%
\begin{proof}
By Definition~\ref{def} and Formulae (\ref{repFnp1}) and (\ref{repFnp2}), we
write that, a.s., for $t\in\mathbb{R},\,n\in\mathbb{N}^*$,
$$
\alpha_n^{(p)}(t)=\frac{1}{(p+1)!}\,\sqrt{n}\left(\mathbb{F}_n(t)^{p+1}-F(t)^{p+1}\right)
+\sum_{k=1}^p a_k^{(p)}\,\frac{\mathbb{F}_n(t)^k}{n^{p-k+1/2}}.
$$
Applying the elementary identity below obtained by writing $a=(a-b)+b$ and
using the binomial theorem
$$
a^{p+1}-b^{p+1}
=(p+1)b^p(a-b)+\sum_{k=2}^{p+1} {p+1\choose k} b^{p+1-k} (a-b)^k
$$
to $a=\mathbb{F}_n(t)$ and $b=F(t)$ yields, a.s., for $t\in\mathbb{R},
\,n\in\mathbb{N}^*$,
\begin{align}
\alpha_n^{(p)}(t)
=&\;
\frac{1}{p!}\,\sqrt{n}\,F(t)^p\left(\mathbb{F}_n(t)-F(t)\right)
\nonumber\\
&
+\frac{1}{(p+1)!}\,\sqrt{n} \,\sum_{k=2}^{p+1} {p+1\choose k} F(t)^{p+1-k}
(\mathbb{F}_n(t)-F(t))^k
\nonumber\\
&+\sum_{k=1}^p a_k^{(p)}\,\frac{\mathbb{F}_n(t)^k}{n^{p-k+1/2}}.
\label{H-N}
\end{align}
Substituting $\mathbb{F}_n(t)-F(t)=\alpha_n^{(p)}(t)/\sqrt{n}$ into (\ref{H-N})
gives (\ref{alphanp}) by setting $b_k^{(p)}:={p+1\choose k}/(p+1)!$.
\end{proof}

%%%%%%%%%%%%%%%%%%%%%%%%%%%%%%%%%%%%%%%%%%%%%%%%%%%%%%%%%%%%%%%%%%%%%%%%%%%%%%%
\subsection{Proof of Theorem~\ref{lem2b}}
%%%%%%%%%%%%%%%%%%%%%%%%%%%%%%%%%%%%%%%%%%%%%%%%%%%%%%%%%%%%%%%%%%%%%%%%%%%%%%%

Set $A_p:=\sum_{k=1}^p a_k^{(p)}>0$ and $I=[0,d/n]$ or $[1-d/n,1]$.
Making use of (\ref{betanp}) together with (\ref{inequality-U}) and the fact
that $0\leq b_k^{(p)}\leq 1$, it is clear that, a.s.,
for any $d,n\in\mathbb{N}^*$ such that $d\leq n$,
$$
\sup_{u\in I}\left|\beta_n^{(p)}(u)-\mathbb{B}_n^{(p)}(u)\right|
\leq \sup_{u\in I}\big|\beta_n(u)-\mathbb{B}_n(u)\big|
+\frac{1}{\sqrt{n}}\sum_{k=2}^{p+1} \frac{1}{n^{k/2-1}} \sup_{u\in[0,1]}|\beta_n(u)|^k
+\frac{A_p}{\sqrt{n}}.
$$
Therefore,
\begin{align*}
\lqn{\mathbb{P}\bigg\{\sup_{u\in I}\left|\beta_n^{(p)}(u)-\mathbb{B}_n^{(p)}(u)\right|
\geq \frac{1}{\sqrt{n}}\,(c_1\log d+x)\bigg\}}
&
\leq \mathbb{P}\bigg\{ \sup_{u\in I}|\beta_n(u)-\mathbb{B}_n(u)|
+\frac{1}{\sqrt{n}}\,\sum_{k=2}^{p+1} \frac{1}{n^{k/2-1}}\,\sup_{u\in[0,1]}|\beta_n(u)|^k
\geq \frac{1}{\sqrt{n}}\,(c_1\log d+x-A_p)\bigg\}\!.
\end{align*}
Now, using the elementary inequality
$$
\mathbb{P}\!\left\{\sum_{k=1}^r \xi_k\geq \sum_{k=1}^r a_k\right\}
\leq \sum_{k=1}^r\mathbb{P}\{\xi_k \geq a_k\},
$$
which is valid for any positive
integer $r$, any r.v.'s $\xi_1,\dots,\xi_r$ and any real numbers $a_1,\dots,a_r$,
we obtain
\begin{align}
\lqn{\mathbb{P}\bigg\{\sup_{u\in I}\left|\beta_n^{(p)}(u)-\mathbb{B}_n^{(p)}(u)\right|
\geq \frac{1}{\sqrt{n}}\,(c_1\log d+x)\bigg\}}
\leq &\;
\mathbb{P}\bigg\{ \sup_{u\in I}|\beta_n(u)
-\mathbb{B}_n(u)|\geq \frac{1}{\sqrt{n}}\,(c_1\log d+\frac{x}{p+1}-A_p)\bigg\}
\nonumber\\
&+\sum_{k=2}^{p+1} \mathbb{P}\bigg\{\sup_{u\in[0,1]}|\beta_n(u)|^k\geq
\frac{xn^{k/2-1}}{p+1}\bigg\}\!.
\label{appp1ref}
\end{align}
On the other hand, the inequality of \cite{Dvoretzky1956} stipulates that there
exists a positive constant $c_4$ such that, for any $x>0$ and any
$n\in\mathbb{N}^*$,
\begin{equation}\label{equaDvoretzky1956}
\mathbb{P}\left\{\sup_{t\in\mathbb{R}}\left|\mathbb{F}_n(t)-F(t)\right|
\geq\frac{x}{\sqrt{n}}\,\right\}\leq c_4\,\exp({-2x^2}).
\end{equation}
Actually (\ref{equaDvoretzky1956}) simply reads,
by means of $\beta_n$, for any $x>0$ and any $n\in\mathbb{N}^*$, as
$$
\mathbb{P}\bigg\{\sup_{u\in[0,1]} |\beta_n(u)| \geq x\bigg\} \leq c_4\,\exp({-2x^2}).
$$
Then,
\begin{align}
\sum_{k=2}^{p+1} \mathbb{P}\bigg\{\sup_{u\in[0,1]}|\beta_n(u)|^k\geq \frac{xn^{k/2-1}}{p+1}\bigg\}
&
\leq c_4 \sum_{k=2}^{p+1} \exp\!\left(-2\left(\frac{x{\,n^{k/2-1}}}{p+1}
\right)^{\!\!2/k}\right)
\nonumber\\
&
\leq c_4 \!\left(\exp({-[2/(p+1)]\,x})+
\sum_{k=3}^{p+1} \exp\!\left(-\frac{2}{p+1}\,x^{2/k}n^{1-2/k}\right)\!\right)\!.
\label{appp1}
\end{align}
Now, by putting (\ref{mason-inequality}) and (\ref{appp1}) into (\ref{appp1ref}), we
immediately complete the proof of Theorem~\ref{lem2b} with
$$
B_p:=c_2\,\exp({c_3A_p})+c_4\quad\text{and}\quad C_p:=\min(c_3,2)/(p+1).
$$
\hfill$\Box$

%%%%%%%%%%%%%%%%%%%%%%%%%%%%%%%%%%%%%%%%%%%%%%%%%%%%%%%%%%%%%%%%%%%%%%%%%%%%%%%
\subsection{Proof of Corollary~\ref{corol22}}
%%%%%%%%%%%%%%%%%%%%%%%%%%%%%%%%%%%%%%%%%%%%%%%%%%%%%%%%%%%%%%%%%%%%%%%%%%%%%%%

The functional $\Phi$ being Lipschitz, there exists a positive constant $L$
such that, for any functions $v,w$,
\begin{equation}\label{lip}
|\Phi(v)-\Phi(w)|\leq L \sup_{t\in\mathbb{R}} |v(t)-w(t)|.
\end{equation}
%inequality that we will use in the form
%\begin{equation}\label{lip}
%\Phi(w)-L \sup_{t\in\mathbb{R}} |v(t)-w(t)|\leq \Phi(v) \leq \Phi(w)+L \sup_{t\in\mathbb{R}} |v(t)-w(t)|.
%\end{equation}
Let us choose for $v,w$ the processes $$V_n:=\alpha_n^{(p)}(\cdot) \quad\mbox{and}\quad
W_n:=\mathbb{B}_n^{(p)}(F(\cdot)).$$
Applying the elementary inequality
$$
|\mathbb{P}(A)-\mathbb{P}(B)|
\leq \mathbb{P}(A\backslash B) + \mathbb{P}(B\backslash A)
=\mathbb{P}\big((A\backslash B)\cup (B\backslash A)\big)
$$
to the events $A=\{\Phi(V_n)\leq x\}$ and $B=\{\Phi(W_n)\leq x\}$ provides,
for any $x\in\mathbb{R}$ and any $n\in\mathbb{N}^*$,
$$
\big|\mathbb{P} \{\Phi(V_n)\leq x\}-\mathbb{P} \{\Phi(W_n)\leq x\}\big|
\leq \mathbb{P} \{\Phi(V_n)\leq x\leq \Phi(W_n) \quad\mbox{or}\quad\Phi(W_n)\leq x\leq \Phi(V_n)\}.
$$
Now, applying to the elementary fact that [$a\le x\le b$ or $b\le x\le a$
implies $|b-x|\le |b-a|$] to the numbers $a=\Phi(V_n)$ and $b=\Phi(W_n)$
\begin{align*}
\mathbb{P} \{\Phi(V_n)\leq x\leq \Phi(W_n)\;\mbox{or}\;\Phi(W_n)\leq x\leq \Phi(V_n)\}&
\leq\mathbb{P} \!\left\{|\Phi(W_n)-x|\leq |\Phi(W_n)-\Phi(V_n)|\right\}\!,
\end{align*}
from which, due to (\ref{lip}), we deduce that
\begin{equation}\label{estimPhiVW}
\big|\mathbb{P} \{\Phi(V_n)\leq x\}-\mathbb{P} \{\Phi(W_n)\leq x\}\big|
\leq \mathbb{P} \!\left\{|\Phi(W_n)-x| \leq L \sup_{t\in\mathbb{R}} |V_n(t)-W_n(t)|\right\}\!.
\end{equation}
On the other hand, by choosing $x=c \log n$ for a large enough constant $c$ in
(\ref{estimation-a}) and putting
$\epsilon_n:=(c+c_1)$ $\log n/\sqrt{n}$,
we obtain the estimate below valid for large enough $n$:
\begin{equation}\label{estimationbis}
\mathbb{P}\!\left\{ \sup_{t\in\mathbb{R}} \left|V_n(t)-W_n(t)\right|
\geq \epsilon_n\right\}\leq \frac{B_p}{n}
=o\!\left(\frac{\log n}{\sqrt{n}}\right)\!.
\end{equation}
Now, by (\ref{estimPhiVW}), we write
\begin{align}
\lqn{
\big|\mathbb{P} \{\Phi(V_n)\leq x\}-\mathbb{P} \{\Phi(W_n)\leq x\}\big|}
\leq&\;
\mathbb{P} \!\left\{ \sup_{t\in\mathbb{R}}
\left|V_n(t)-W_n(t)\right|<\epsilon_n, |\Phi(W_n)-x| \leq
L \sup_{t\in\mathbb{R}} |V_n(t)-W_n(t)|\right\}
\nonumber\\
&
+\mathbb{P}\!\left\{ \sup_{t\in\mathbb{R}} \left|V_n(t)-W_n(t)\right|
\geq \epsilon_n,|\Phi(W_n)-x| \leq
L \sup_{t\in\mathbb{R}} |V_n(t)-W_n(t)|\right\}
\nonumber\\
\leq&\;
\mathbb{P} \!\left\{|\Phi(W_n)-x| \leq L \epsilon_n\right\}
+\mathbb{P}\!\left\{ \sup_{t\in\mathbb{R}} \left|V_n(t)-W_n(t)\right|
\geq \epsilon_n\right\}\!.
\label{estimationter}
\end{align}
Noticing that the distribution of $\mathbb{B}_n$ does not depend on $n$,
which entails the equality
$$
\mathbb{P} \!\left\{|\Phi(W_n)-x| \leq L \epsilon_n\right\}
=\mathbb{P} \!\left\{|\Phi(W)-x| \leq L \epsilon_n\right\}
$$
where $$W:=F(\cdot)^p\,\mathbb{B}(F(\cdot))/p!,$$
and recalling the assumption that the r.v.
$\Phi(W)$ admits a density function bounded by $M$ say,
we get that, for any $x\in\mathbb{R}$ and any $n\in\mathbb{N}^*$,
\begin{equation}\label{estimationquater}
\mathbb{P} \!\left\{|\Phi(W_n)-x| \leq L \epsilon_n\right\}\leq 2LM\epsilon_n.
\end{equation}
Finally, putting (\ref{estimationbis}) and  (\ref{estimationquater})
into (\ref{estimationter}) leads to (\ref{estimationfunctional}),
which completes the proof of Corollary~\ref{corol22}.
An alternative proof of a similar result may be found in \cite{Shorack1986} pp. 502--503.
\hfill$\Box$

%%%%%%%%%%%%%%%%%%%%%%%%%%%%%%%%%%%%%%%%%%%%%%%%%%%%%%%%%%%%%%%%%%%%%%%%%%%%%%%
\subsection{Proof of Corollary~\ref{corollaryapprox}}\label{section-cor-approx}
%%%%%%%%%%%%%%%%%%%%%%%%%%%%%%%%%%%%%%%%%%%%%%%%%%%%%%%%%%%%%%%%%%%%%%%%%%%%%%%

Set, for any $n\in\mathbb{N}^*$,
$$
\pi_n:=\mathbb{P}\!\left\{\sup_{t\in\mathbb{R}} \left|\alpha_n^{(p)}(t)
-\mathbb{B}_n^{(p)}(F(t))\right| \geq c_6\,\frac{\log n}{\sqrt{n}}\right\}\!.
$$
Applying (\ref{estimation-a}) to $x=c'\log n$ for a sufficiently large constant $c'$
yields, for large enough $n$,
\begin{align*}
\pi_n&\leq B_p \sum_{k=2}^{p+1}\exp\!\left(-C_pc'^{2/k}\,(\log n)^{2/k}n^{1-3/k}\right)
\\&=\mathcal{O}\!\left(\frac{1}{n^2}\right),
\end{align*}
where $c_6$ is a positive constant.
Hence the series {   $(\sum_{n\geq 1} \pi_n)$} is convergent and by appealing to
Borel-Cantelli lemma, we get that
$$
\mathbb{P}\!\left(\limsup_{n\in\mathbb{N}^*}\left\{ \sup_{t\in\mathbb{R}}
\left|\alpha_n^{(p)}(t)-\mathbb{B}_n^{(p)}(F(t))\right|
\geq c_6\,\frac{\log n}{\sqrt{n}}\right\}\right)=0,
$$
which clearly implies Corollary~\ref{corollaryapprox}.
\hfill$\Box$

%%%%%%%%%%%%%%%%%%%%%%%%%%%%%%%%%%%%%%%%%%%%%%%%%%%%%%%%%%%%%%%%%%%%%%%%%%%%%%%
\subsection{Proof of Theorem~\ref{kieferapproximation}}
%%%%%%%%%%%%%%%%%%%%%%%%%%%%%%%%%%%%%%%%%%%%%%%%%%%%%%%%%%%%%%%%%%%%%%%%%%%%%%%

In view of (\ref{alphanp}), we write that, a.s., for any $s\in(0,1)$, any
$t\in\mathbb{R}$ and any $k\in\mathbb{N}^*$,
\begin{align}
\sqrt{k}\,\alpha_k^{(p)}(t)-\mathbb{K}^{(p)}(k,F(t))
=&\;
\frac{1}{p!}\,F(t)^p\left[\sqrt{k}\,\alpha_k(t)-\mathbb{K}(k,F(t))\right]
\nonumber\\
&
+\sum_{i=2}^{p+1} b_i^{(p)}\,\frac{F(t)^{p+1-i}}{k^{(i-1)/2}}\,\alpha_k(t)^i
+\sum_{i=1}^p a_i^{(p)}\,\frac{\mathbb{F}_k(t)^i}{k^{p-i+1/2}}.
\end{align}
Then, by using (\ref{inequality}) together with the fact that $0\leq b_i^{(p)}\leq 1$,
we deduce that, a.s., for any $n\in\mathbb{N}^*$,
\begin{align}
\lqn{
\max_{1\leq k\leq n}\sup_{t\in\mathbb{R}}\left|\sqrt{k}\,\alpha_k^{(p)}(t)
-\mathbb{K}^{(p)}(k,F(t))\right|}
&
\leq \frac{1}{p!}\,\max_{1\leq k\leq n}\sup_{t\in\mathbb{R}}\left|\sqrt{k}\,\alpha_k(t)
-\mathbb{K}(k,F(t))\right|+\sum_{i=2}^{p+1} \max_{1\leq k\leq n}\!
\left(\frac{1}{k^{(i-1)/2}}\sup_{t\in\mathbb{R}}|\alpha_k(t)|^i\right)
\nonumber\\
&+\sum_{i=1}^p \max_{1\leq k\leq n} \frac{a_i^{(p)}}{k^{p-i+1/2}}
\nonumber\\
&
\leq \max_{1\leq k\leq n}\sup_{t\in\mathbb{R}}\left|\sqrt{k}\,\alpha_k(t)
-\mathbb{K}(k,F(t))\right|+\sum_{i=2}^{p+1} \left(\max_{1\leq k\leq n}
\sup_{t\in\mathbb{R}}|\alpha_k(t)|\right)^{\!i}+A_p.
\label{referthe25}
\end{align}
Finally, by putting (\ref{equation11KIE}) and (\ref{empililaeae}) into
(\ref{referthe25}), we completes the proof of Theorem~\ref{kieferapproximation}.
\hfill$\Box$

%%%%%%%%%%%%%%%%%%%%%%%%%%%%%%%%%%%%%%%%%%%%%%%%%%%%%%%%%%%%%%%%%%%%%%%%%%%%%%%
\subsection{Proof of Corollary~\ref{corollarylli}}
%%%%%%%%%%%%%%%%%%%%%%%%%%%%%%%%%%%%%%%%%%%%%%%%%%%%%%%%%%%%%%%%%%%%%%%%%%%%%%%

From Theorem~\ref{kieferapproximation}, we deduce that, with probability $1$, as $n\to\infty$,
$$
\sup_{t\in\mathbb{R}} \big|\alpha_n^{(p)}(t)\big|
=\frac{1}{\sqrt{n}}\,\sup_{t\in\mathbb{R}} \big|\mathbb{K}^{(p)}(n,F(t))\big|
+\mathcal{O}\!\left(\frac{(\log n)^2}{\sqrt{n}}\right)\!.
$$
Therefore, a.s.,
\begin{align}
\limsup_{n\to\infty}\frac{\sup_{t\in\mathbb{R}}
\big|\alpha_n^{(p)}(t)\big|}{\sqrt{\log \log n}}
&
=\limsup_{n\to\infty}\frac{\sup_{t\in\mathbb{R}}
\big|\mathbb{K}^{(p)}(n,F(t))\big|}{\sqrt{n\log \log n}}
\nonumber\\
&
=\limsup_{n\to\infty}
\frac{\sup_{u\in[0,1]} |u^p\,\mathbb{K}(n,u)|}{p!\,\sqrt{n\log \log n}}
\nonumber\\
&
=\frac{\sqrt{2}}{p!}\sup_{u\in[0,1]} \sqrt{\mathrm{Var}(u^p\,\mathbb{K}(1,u))}.
\label{lil}
\end{align}
In the last equality, we have used Strassen's law of the iterated logarithm
for Gaussian processes; see for instance \cite{Oodaira1973},
or Corollary 1.15.1 of \cite{Csorgo1981REVESZ}.
Observing that $\mathrm{Var}(u^p\,\mathbb{K}(1,u))=u^{2p+1}(1-u)$ and that
$$
\sup_{u\in[0,1]} u^{2p+1}(1-u)=\frac{(2p+1)^{2p+1}}{(2p+2)^{2p+2}},
$$
(\ref{lil}) readily implies (\ref{lli}) which proves Corollary~\ref{corollarylli}.
\hfill$\Box$

%%%%%%%%%%%%%%%%%%%%%%%%%%%%%%%%%%%%%%%%%%%%%%%%%%%%%%%%%%%%%%%%%%%%%%%%%%%%%%%
\subsection{Proof of Corollary~\ref{ecooooo}}
%%%%%%%%%%%%%%%%%%%%%%%%%%%%%%%%%%%%%%%%%%%%%%%%%%%%%%%%%%%%%%%%%%%%%%%%%%%%%%%

We work under Hypothesis $\mathcal{H}_0$. Let us introduce the $p$-fold
integrated empirical process related to the d.f. $F_0$:
$$
\alpha_{0,n}^{(p)}(t):=\sqrt{n}\left(\mathbb{F}_n^{(p)}(t)-F_0^{(p)}(t)\right)
\quad\text{for}\quad t\in\mathbb{R},\,n\in\mathbb{N}^*.
$$
By the triangular inequality, we plainly have
$$
\left|\mathbf{S}_n^{(p)}-\sup_{t\in\mathbb{R}}
\big|\mathbb{B}_n^{(p)}(F_0(t))\big|\right|
\leq
\sup_{t\in\mathbb{R}}\left|\alpha_{0,n}^{(p)}(t)-\mathbb{B}_n^{(p)}(F_0(t))\right|
$$
from which together with (\ref{approx}) we deduce (\ref{kol}).

Similarly,
\begin{align}
\left|\mathbf{T}_n^{(p)}-\int_{\mathbb{R}}
\big[\mathbb{B}_n^{(p)}(F_0(t))\big]^2\,dF_0(t)\right|
\leq&\;
\int_{\mathbb{R}}\left|\alpha_{0,n}^{(p)}(t)^2
-\big[\mathbb{B}_n^{(p)}(F_0(t))\big]^2\right| dF_0(t)
\nonumber\\
\leq&\;
\sup_{t\in\mathbb{R}}\left|\alpha_{0,n}^{(p)}(t)-\mathbb{B}_n^{(p)}(F_0(t))\right|
\nonumber\\
&
\times \left(\sup_{t\in\mathbb{R}}\left|\alpha_{0,n}^{(p)}(t)\right|
+\sup_{t\in\mathbb{R}}\big|\mathbb{B}_n^{(p)}(F_0(t))\big|\right)\!.
\label{vonmisesbis}
\end{align}
In the last inequality above appears the supremum
$$
\sup_{t\in\mathbb{R}}\big|\mathbb{B}_n^{(p)}(F_0(t))\big|
\leq\sup_{u\in[0,1]}\big|\mathbb{B}_n(u)\big|.
$$
Then, by putting (\ref{approx}), (\ref{lli}) and (\ref{estimbridge}) into
(\ref{vonmisesbis}), we immediately deduce (\ref{vonmises}). The proof of
Corollary~\ref{ecooooo} is finished.
\hfill$\Box$

%%%%%%%%%%%%%%%%%%%%%%%%%%%%%%%%%%%%%%%%%%%%%%%%%%%%%%%%%%%%%%%%%%%%%%%%%%%%%%%
\subsection{Proof of Theorem~\ref{theweithed1}}
%%%%%%%%%%%%%%%%%%%%%%%%%%%%%%%%%%%%%%%%%%%%%%%%%%%%%%%%%%%%%%%%%%%%%%%%%%%%%%%

We imitate the proof of Theorem 2.1 in \cite{Mason1999}.
Let us introduce
\begin{align*}
\Delta_{n,\nu}^{(p)}(d)&:=\sup_{u\in[d/n,1-d/n]}\frac{\big|\beta_n^{(p)}(u)
-\mathbb{B}_n^{(p)}(u)\big|}{\big(u(1-u)\big)^{\nu}},
\\
\Delta_{n,\nu}^{1,(p)}(d)&:=\sup_{u\in[d/n,1]}\frac{\big|\beta_n^{(p)}(u)
-\mathbb{B}_n^{(p)}(u)\big|}{u^{\nu}},
\\
\Delta_{n,\nu}^{2,(p)}(d)&:=\sup_{u\in[0,1-d/n]}\frac{\big|\beta_n^{(p)}(u)
-\mathbb{B}_n^{(p)}(u)\big|}{(1-u)^{\nu}}.
\end{align*}
We clearly have
\begin{align*}
\Delta_{n,\nu}^{1,(p)}(d)
&
= \max\!\left(\sup_{u\in[d/n,1/2]}\frac{\big|\beta_n^{(p)}(u)
-\mathbb{B}_n^{(p)}(u)\big|}{\big(u(1-u)\big)^{\nu}},
\sup_{u\in[1/2,1-d/n]}\frac{\big|\beta_n^{(p)}(u)
-\mathbb{B}_n^{(p)}(u)\big|}{\big(u(1-u)\big)^{\nu}}\right)
\\
&
\leq 2^{\nu} \max\!\left( \sup_{u\in[d/n,1/2]}\frac{\big|\beta_n^{(p)}(u)
-\mathbb{B}_n^{(p)}(u)\big|}{u^{\nu}},
\sup_{u\in[1/2,1-d/n]}\frac{\big|\beta_n^{(p)}(u)
-\mathbb{B}_n^{(p)}(u)\big|}{(1-u)^{\nu}}\right)
\end{align*}
which entails that
\begin{equation}\label{ineg-Delta-max}
\Delta_{n,\nu}^{(p)}(d)\leq 2^{\nu}  \max\!\left(\Delta_{n,\nu}^{1,(p)}(d),
\Delta_{n,\nu}^{2,(p)}(d)\right)\!.
\end{equation}
Hence, it suffices to derive inequalities of the form~(\ref{ineg-Delta1})
for the auxiliary quantities $\Delta_{n,\nu}^{(p)1}(d)$ and $\Delta_{n,\nu}^{(p)2}(d)$.
First, notice that if we have for each $i\in\{d,\dots,n-1\}$,
$$
\sup_{u\in[0,(i+1)/n]} \big|\beta_n^{(p)}(u)-\mathbb{B}_n^{(p)}(u)\big|<x(i/n)^{\nu},
$$
then
\begin{align*}
\Delta_{n,\nu}^{1,(p)}(d)
&
=\max_{i=d,\dots,n-1} \!\left(\sup_{u\in[i/n,(i+1)/n]}
\frac{\big|\beta_n^{(p)}(u)-\mathbb{B}_n^{(p)}(u)\big|}{u^{\nu}}\right)
\\
&
\leq\max_{i=d,\dots,n-1} \!\bigg(\!\!\left(\frac{n}{i}\right)^{\nu}\sup_{u\in[0,(i+1)/n]}
\big|\beta_n^{(p)}(u)-\mathbb{B}_n^{(p)}(u)\big|\bigg)<x.
\end{align*}
Hence, we derive the inequality
\begin{equation}\label{ineg-Delta-delta1}
\mathbb{P}\big\{\Delta_{n,\nu}^{1,(p)}(d)\geq x\big\}
\leq\sum_{i=d}^{n-1} \delta_{i,n},
\end{equation}
where
$$
\delta_{i,n}:=\mathbb{P}\bigg\{\sup_{u\in[0,(i+1)/n]}
\big|\beta_n^{(p)}(u)-\mathbb{B}_n^{(p)}(u)\big|\geq x(i/n)^{\nu}\bigg\}.
$$
Quite similarly,
\begin{equation}\label{ineg-Delta-delta2}
\mathbb{P}\big\{\Delta_{n,\nu}^{2,(p)}(d)\geq x\big\}
\leq\sum_{i=0}^{n-d-1} \delta_{i,n}.
\end{equation}
Set now (under the assumption $\nu>0$)
$$
a_{\nu}:=\max\!\left(\!1,c_1\sup_{i\in\mathbb{N}^*}\frac{\log (i+1)}{i^{\nu}}\right),
$$
where the constant $c_1$ is that of (\ref{mason-inequality}).
Of course, $a_{\nu}\geq 1$ and, for any $i\in\mathbb{N}^*$,
$$a_{\nu}\,i^{\nu}\geq c_1\log (i+1).$$ Then, by writing
$$x=(x-a_{\nu}n^{\nu-1/2})+a_{\nu}n^{\nu-1/2}$$ in $\delta_{i,n}$ above
and appealing to (\ref{estimation-b}), we have, for $i\geq d$, any
$n\in\mathbb{N}^*$ and large enough $x$ (recall the assumption $\nu<1/2$),
\begin{align}
\delta_{i,n}&\leq \mathbb{P}\bigg\{ \sup_{u\in[0,(i+1)/n]}
\big|\beta_n^{(p)}(u)-\mathbb{B}_n^{(p)}(u)\big|
\geq \frac{1}{\sqrt{n}} \left(c_1\log(i+1)+i^{\nu}(xn^{1/2-\nu}-a_{\nu})\right)\!\bigg\}
\nonumber\\
&
\leq B_p \sum_{k=2}^{p+1}\exp\!\left(-C_p\,(xn^{1/2-\nu}-a_{\nu})^{2/k}i^{2\nu/k}n^{1-2/k}\right)\!.
\label{ineg-delta}
\end{align}
In order to split the quantity $(xn^{1/2-\nu}-a_{\nu})^{2/k}$, we use the elementary inequality
$$(a+b)^{\mu}\geq 2^{\mu-1}\big(a^{\mu}+b^{\mu}\big),$$
which is valid for any $a,b\geq0$ and any $\mu\in(0,1]$.
If, in addition, $a$ and $b$ satisfy $a\geq 4b$, which entails
that $$a-2b\geq a/2,$$ then
$$(a-b)^{\mu}\geq 2^{\mu-1}\big((a-2b)^{\mu}+b^{\mu}\big)
\geq \frac12\big(a^{\mu}+b^{\mu}\big).$$ Thanks to this last inequality,
we can write, for any $i\geq d$, any $k\leq p+1$ and any $x\geq 4a_{\nu}n^{\nu-1/2}$,
that
\begin{align}
\lqn{\exp\!\left(-C_p\,(xn^{1/2-\nu}-a_{\nu})^{2/k}i^{2\nu/k}n^{1-2/k}\right)}
\nonumber
&
\leq\exp\!\left(-C_p''\,a_{\nu}^{2/k}i^{2\nu/k}n^{1-2/k}\right)
\exp\!\left(-C_p''\,x^{2/k}d^{2\nu/k}n^{1-(2\nu+1)/k}\right)
\nonumber\\
&
\leq\exp\!\left(-C_p''\,i^{2\nu/(p+1)}\right)
\exp\!\left(-C_p''\,x^{2/k}d^{2\nu/k}n^{1-(2\nu+1)/k}\right),
\label{ineg-exp}
\end{align}
where we set $C_p''=C_p/2$.
Consequently, by putting (\ref{ineg-exp}) into (\ref{ineg-delta}), and next this latter into
(\ref{ineg-Delta-delta1}) and (\ref{ineg-Delta-delta2}), we get
\begin{equation}\label{ineg-Delta12}
\mathbb{P}\big\{\Delta_{n,\nu}^{1,(p)}(d)\geq x\big\}
\quad\text{and}\quad
\mathbb{P}\big\{\Delta_{n,\nu}^{2,(p)}(d)\geq x\big\}
\leq \frac12\,B_{p,\nu} \sum_{k=2}^{p+1}
\exp\!\left(-C_p''\,x^{2/k}d^{2\nu/k}n^{1-(2\nu+1)/k}\right)\!,
\end{equation}
with the constant
$$
B_{p,\nu}:=2B_p\sum_{i=0}^{\infty} \exp\!\left(-C_p''i^{2\nu/(p+1)}\right)<+\infty.
$$
Finally, by (\ref{ineg-Delta-max}),
\begin{equation}\label{ineg-Delta-Delta1}
\mathbb{P}\big\{\Delta_{n,\nu}^{(p)}(d)\geq x\big\}
\leq \mathbb{P}\!\left\{\Delta_{n,\nu}^{1,(p)}(d)\geq \frac{x}{2}\right\}
+\mathbb{P}\!\left\{\Delta_{n,\nu}^{2,(p)}(d)\geq \frac{x}{2}\right\},
\end{equation}
and by putting (\ref{ineg-Delta12}) into (\ref{ineg-Delta-Delta1}), we
obtain (\ref{ineg-Delta1}) with $C_p'=C_p''/2^{2/(p+1)}$.
\hfill$\Box$

%%%%%%%%%%%%%%%%%%%%%%%%%%%%%%%%%%%%%%%%%%%%%%%%%%%%%%%%%%%%%%%%%%%%%%%%%%%%%%%
\subsection{Proof of Corollary \ref{corollary-O}}
%%%%%%%%%%%%%%%%%%%%%%%%%%%%%%%%%%%%%%%%%%%%%%%%%%%%%%%%%%%%%%%%%%%%%%%%%%%%%%%
%$\xi_n=\mathcal{O}_{\mathbb{P}}(1)$ signifie
%$$\forall \epsilon>0,\exists \alpha(\epsilon)>0,\forall n\in\mathbb{N}^*,\quad\mathbb{P}\{|\xi_n|>\alpha(\epsilon)\}<\epsilon.%$$
Applying (\ref{ineg-Delta1}) to $x=c''\log n$ for a sufficiently large constant $c''$ yields, for large enough $n$,
\begin{align*}
\mathbb{P}\bigg\{\Delta_{n,\nu}^{(p)}(d)\geq c''\frac{\log n}{n^{1/2-\nu}}\bigg\}
\leq&\; B_{p,\nu} \sum_{k=2}^{p+1} \exp\!\left(-C_p'c''^{2/k}\,d^{2\nu/k}(\log n)^{2/k}n^{1-2/k}\right)
\\=&\;\mathcal{O}\!\left(\frac{1}{n^2}\right)\!.
\end{align*}
We conclude with the aid of Borel-Cantelli lemma as in Subsection~\ref{section-cor-approx}.
\hfill$\Box$

%%%%%%%%%%%%%%%%%%%%%%%%%%%%%%%%%%%%%%%%%%%%%%%%%%%%%%%%%%%%%%%%%%%%%%%%%%%%%%%
\subsection{Proof of Corollary~\ref{ximn}}
%%%%%%%%%%%%%%%%%%%%%%%%%%%%%%%%%%%%%%%%%%%%%%%%%%%%%%%%%%%%%%%%%%%%%%%%%%%%%%%

For each $m,n\in\mathbb{N}^*$, let $\alpha_m^{1,(p)}$ and $\alpha_n^{2,(p)}$
denote the empirical processes
respectively associated with the samples $X_1,\ldots, X_m$ and $Y_1,\ldots,Y_n$.
By replacing $\mathbb{F}_m^{(p)}(t)$ by $\alpha_m^{1,(p)}(t)/\sqrt{m}+F^{(p)}(t)$
and $\mathbb{G}_n^{(p)}(t)$ by $\alpha_n^{2,(p)}(t)/\sqrt{n}+G^{(p)}(t)$,
using the binomial theorem and recalling that, under $\mathcal{H}_0'$, $F=G$,
we write
\begin{align*}
\boldsymbol{\xi}_{m,n}^{(p,q)}(t)
&
=\sqrt{\frac{mn}{m+n}}\left[\left(\frac{\alpha_m^{1,(p)}(t)}{\sqrt{m}}+F^{(p)}(t)\right)^{\!\!q}
-\left(\frac{\alpha_n^{2,(p)}(t)}{\sqrt{n}}+F^{(p)}(t)\right)^{\!\!q}\right]
\\
&
=\sqrt{\frac{mn}{m+n}}\sum_{k=1}^q {q\choose k}\big(F^{(p)}(t)\big)^{q-k}
\left[\left(\frac{\alpha_m^{1,(p)}(t)}{\sqrt{m}}\right)^{\!\!k}
-\left(\frac{\alpha_n^{2,(p)}(t)}{\sqrt{n}}\right)^{\!\!k}\right]
\\
&
=q\big(F^{(p)}(t)\big)^{q-1} \left(\sqrt{\frac{n}{m+n}}\,\alpha_m^{1,(p)}(t)
-\sqrt{\frac{m}{m+n}}\,\alpha_n^{2,(p)}(t)\right) +\Delta_{m,n}(t)
\end{align*}
where
$$
\Delta_{m,n}(t)=\sqrt{\frac{mn}{m+n}}\sum_{k=2}^q {q\choose k}\big(F^{(p)}(t)\big)^{q-k}
\left[\left(\frac{\alpha_m^{1,(p)}(t)}{\sqrt{m}}\right)^{\!\!k}
-\left(\frac{\alpha_n^{2,(p)}(t)}{\sqrt{n}}\right)^{\!\!k}\right]\!.
$$
By (\ref{lli}) and (\ref{inequality}), it is easily seen that, with probability $1$, as $m,n\to\infty$,
\begin{equation}\label{Delta}
\sup_{t\in\mathbb{R}}|\Delta_{m,n}(t)|=\mathcal{O}\!\left(\frac{(\log\log m)^{q/2}}{\sqrt{m}}\right)
+\mathcal{O}\!\left(\frac{(\log\log n)^{q/2}}{\sqrt{n}}\right)\!.
\end{equation}
On the other hand, by Corollary~\ref{corollaryapprox}, we can construct two sequences of
Brownian bridges $\big\{\mathbb{B}_m^1:m\in\mathbb{N}^*\big\}$
and $\big\{\mathbb{B}_n^2:n\in\mathbb{N}^*\big\}$ such that, with probability $1$, as $m,n\to\infty$,
\begin{align}
\sup_{t\in\mathbb{R}} \left|\alpha_m^{1,(p)}(t)-\frac{1}{p!}\,
F(t)^p\,\mathbb{B}_m^1(F(t))\right|
&
=\mathcal{O}\!\left(\frac{\log m}{\sqrt{m}}\right)\!,
\nonumber\\[-1ex]
\label{ecart1}\\[-1ex]
\sup_{t\in\mathbb{R}} \left|\alpha_n^{2,(p)}(t)-\frac{1}{p!}\,
F(t)^p\,\mathbb{B}_n^2(F(t))\right|
&
=\mathcal{O}\!\left(\frac{\log n}{\sqrt{n}}\right)\!.
\nonumber
\end{align}
Setting $\mathbb{B}_{m,n}^{(p,q)}$ as in Corollary~\ref{ximn}, we have
\begin{align}
\boldsymbol{\xi}_{m,n}^{(p,q)}(t)-\mathbb{B}_{m,n}^{(p,q)}(t)=
&\;
\frac{q}{(p+1)!^{q-1}}\,F(t)^{(p+1)(q-1)}\left[
\sqrt{\frac{n}{m+n}}\left(\alpha_m^{1,(p)}(t)-\frac{1}{p!}\,F(t)^p\,
\mathbb{B}_m^1(F(t))\right)\right.
\nonumber\\
&
\left.-\sqrt{\frac{m}{m+n}}\left(\alpha_n^{2,(p)}(t)-\frac{1}{p!}\,F(t)^p\,
\mathbb{B}_n^2(F(t))\right)\right]+\Delta_{m,n}(t).
\label{ecart3}
\end{align}
By putting (\ref{Delta}) and (\ref{ecart1}) into (\ref{ecart3}),
we deduce the result announced in Corollary~\ref{ximn}.
\hfill$\Box$

%%%%%%%%%%%%%%%%%%%%%%%%%%%%%%%%%%%%%%%%%%%%%%%%%%%%%%%%%%%%%%%%%%%%%%%%%%%%%%%
\subsection{Proof of Theorem~\ref{theoremK}}
%%%%%%%%%%%%%%%%%%%%%%%%%%%%%%%%%%%%%%%%%%%%%%%%%%%%%%%%%%%%%%%%%%%%%%%%%%%%%%%
Let us introduce, for each $k\in\{1,\dots,K\}$, the $p$-fold integrated empirical
process associated with the d.f. $F^k$
$$
\alpha_n^{k,(p)}(t):=\sqrt{n}\left(\mathbb{F}_n^{k,(p)}(t)-F^{k,(p)}(t)\right)
\quad\text{for}\quad t\in\mathbb{R},\,n\in\mathbb{N}^*.
$$
By recalling (\ref{DD}) and making use of the most well-known variance formula
$$
\sum_{k=1}^K n_k(x_k-\bar{x})^2
=\sum_{k=1}^K n_k(x_k-x_0)^2 -\boldsymbol{|n|}\left(\bar{x}-x_0\right)^2,
$$
where we have denoted $$\boldsymbol{|n|}=\sum_{k=1}^K n_k ~~\mbox{and}~~
\bar{x}=\frac{1}{\boldsymbol{|n|}}\sum_{k=1}^K n_kx_k, ~~\mbox{
for} ~~\boldsymbol{n}=(n_1,\dots,n_K),$$ we rewrite
$\boldsymbol{\xi}_{K,\boldsymbol{n}}^{(p)}(t)$ under Hypothesis $\mathcal{H}_0^K$ as
\begin{align*}
\boldsymbol{\xi}_{K,\boldsymbol{n}}^{(p)}(t)
&
=\sum_{k=1}^K n_k\!
\left(\mathbb{F}_{n_k}^{k,(p)}(t)-F_0^{(p)}(t)\right)^{\!2}
-\frac{1}{\boldsymbol{|n|}}\left(\sum_{k=1}^K n_k\!
\left(\mathbb{F}_{n_k}^{k,(p)}(t)-F_0^{(p)}(t)\right)\right)^{\!\!2}
\\
&=\sum_{k=1}^K \alpha_{n_k}^{k,(p)}(t)^2
-\left(\sum_{k=1}^K \sqrt{\frac{n_k}{\boldsymbol{|n|}}}\,
\alpha_{n_k}^{k,(p)}(t)\right)^{\!\!2}.
\end{align*}
Next, setting $\mathbb{B}_{K,\boldsymbol{n}}^{(p)}$ as in Theorem~\ref{theoremK},
we have
\begin{equation}\label{ecartK}
\boldsymbol{\xi}_{K,\boldsymbol{n}}^{(p)}(t)-\mathbb{B}_{K,\boldsymbol{n}}^{(p)}(t)
=\Delta_{\boldsymbol{n}}^1(t)-\Delta_{\boldsymbol{n}}^2(t),
\end{equation}
where we put, for any $t\in\mathbb{R}$ and any
$\boldsymbol{n}=(n_1,\dots,n_K)\in\mathbb{N}^*$,
\begin{align*}
\Delta_{\boldsymbol{n}}^1(t)&=\sum_{k=1}^K \left(\alpha_{n_k}^{k,(p)}(t)^2
-\frac{F_0(t)^{2p}}{p!^2} \,\mathbb{B}_{n_k}^k\!\big(F_0(t)\big)^2\right)\!,
\\
\Delta_{\boldsymbol{n}}^2(t)&=\left(\sum_{k=1}^K \sqrt{\frac{n_k}{\boldsymbol{|n|}}}\,
\alpha_{n_k}^{k,(p)}(t)\right)^{\!\!2}
-\left(\frac{F_0(t)^p}{p!} \sum_{k=1}^K\sqrt{\frac{n_k}{\boldsymbol{|n|}}}\,
\mathbb{B}_{n_k}^k\!\big(F_0(t)\big)\right)^{\!\!2}\!.
\end{align*}
By setting, for any $k\in\{1,\dots,K\}$, any $t\in\mathbb{R}$ and any
$\boldsymbol{n}=(n_1,\dots,n_K)\in\mathbb{N}^*$,
$$
\delta_{k,\boldsymbol{n}}(t)=\alpha_{n_k}^{k,(p)}(t)-\frac{F_0(t)^p}{p!}
\,\mathbb{B}_{n_k}^k\!\big(F_0(t)\big),$$
and
$$
\epsilon_{k,\boldsymbol{n}}(t)=\alpha_{n_k}^{k,(p)}(t)+\frac{F_0(t)^p}{p!}
\,\mathbb{B}_{n_k}^k\!\big(F_0(t)\big)
$$
and writing $\Delta_{\boldsymbol{n}}^1(t)$ and $\Delta_{\boldsymbol{n}}^2(t)$ as
$$
\Delta_{\boldsymbol{n}}^1(t)=\sum_{k=1}^K \delta_{k,\boldsymbol{n}}(t)
\,\epsilon_{k,\boldsymbol{n}}(t),$$
and
$$
\Delta_{\boldsymbol{n}}^2(t)=\sum_{k=1}^K \sqrt{\frac{n_k}{\boldsymbol{|n|}}}
\,\delta_{k,\boldsymbol{n}}(t) \,\sum_{k=1}^K\sqrt{\frac{n_k}{\boldsymbol{|n|}}}
\,\epsilon_{k,\boldsymbol{n}}(t),
$$
we derive the following inequalities:
\begin{align}
\sup_{t\in\mathbb{R}} |\Delta_{\boldsymbol{n}}^1(t)|
&
\leq\sum_{k=1}^K \left(\sup_{t\in\mathbb{R}}
\left|\delta_{k,\boldsymbol{n}}(t)\right|\right)\!
\left(\sup_{t\in\mathbb{R}} \left|\epsilon_{k,\boldsymbol{n}}(t)\right|\right)\!,
\nonumber\\[-1ex]
\label{delta1}\\[-1ex]
\sup_{t\in\mathbb{R}} \big|\Delta_{\boldsymbol{n}}^2(t)\big|
&
\leq\left(\sum_{k=1}^K \sup_{t\in\mathbb{R}}
\left|\delta_{k,\boldsymbol{n}}(t)\right|\right)\!
\left(\sum_{k=1}^K \sup_{t\in\mathbb{R}} \left|\epsilon_{k,\boldsymbol{n}}(t)\right|\right)\!.
\nonumber
\end{align}
By (\ref{corollaryapprox}), (\ref{lli}) and (\ref{estimbridge}),
we get the bounds, a.s., for each $k\in\{1,\dots,K\}$, as $n_k\to\infty$,
\begin{equation}\label{deltaeps}
\sup_{t\in\mathbb{R}} \left|\delta_{k,\boldsymbol{n}}(t)\right|
=\mathcal{O}\!\left(\frac{\log n_k}{\sqrt{n_k}}\right)
\quad\text{and}\quad\sup_{t\in\mathbb{R}} \left|\epsilon_{k,\boldsymbol{n}}(t)\right|
=\mathcal{O}\!\left(\!\sqrt{\log\log n_k}\,\right)\!.
\end{equation}
Finally, by putting (\ref{deltaeps}) into (\ref{delta1}),
and next into (\ref{ecartK}), we finish the proof of Theorem~\ref{theoremK}.
\hfill$\Box$

%%%%%%%%%%%%%%%%%%%%%%%%%%%%%%%%%%%%%%%%%%%%%%%%%%%%%%%%%%%%%%%%%%%%%%%%%%%%%%%
\subsection{Proof of Theorem~\ref{theorem1}}
%%%%%%%%%%%%%%%%%%%%%%%%%%%%%%%%%%%%%%%%%%%%%%%%%%%%%%%%%%%%%%%%%%%%%%%%%%%%%%%

In the computations below, the superscript ``$-$'' in the quantities
$\mathbb{F}$, $\mathbb{F}^{(p)}$, $\alpha$ and $\beta$ refers to the $k$ first
observations while the superscript ``$+$'' refers to the $(n-k)$ last ones.
By (\ref{alpha}) and (\ref{alphatilde}), we write that, for  $n\in \mathbb{N}^*$,
$s\in(0,1)$ and $t\in\mathbb{R}$,
\begin{align*}
\widetilde{\alpha}_n^{(p)}(s,t)
&
=\frac{\lfloor ns\rfloor(n-\lfloor ns\rfloor)}{n^{3/2}}
\left[ \left(\mathbb{F}_{\lfloor ns\rfloor}^{(p)-}(t)-F^{(p)}(t)\right)
-\left(\mathbb{F}_{n-\lfloor ns\rfloor}^{(p)+}(t)-F^{(p)}(t)\right) \right]
\\
&
=\frac{\lfloor ns\rfloor(n-\lfloor ns\rfloor)}{n^{3/2}}\left(
\frac{\alpha_{\lfloor ns\rfloor}^{(p)-}(t)}{\sqrt{\lfloor ns\rfloor}}
-\frac{\alpha_{n-\lfloor ns\rfloor}^{(p)+}(t)}{\sqrt{n-\lfloor ns\rfloor}}\right)\!.
\end{align*}
Using (\ref{alphanp}), we derive, a.s., for any  $n\in \mathbb{N}^*$, any
$s\in(0,1)$ and any $t\in\mathbb{R}$, the form
\begin{equation}\label{equation78}
\widetilde{\alpha}_n^{(p)}(s,t)=\mathrm{I}_n(s,t)-\mathrm{II}_n(s,t)
+\mathrm{III}_n(s,t)+\mathrm{IV}_n(s,t)
\end{equation}
where
\begin{align*}
\mathrm{I}_n(s,t)
&
=\frac{\lfloor ns\rfloor(n-\lfloor ns\rfloor)}{p!\,n^{3/2}}\,
\frac{\alpha_{\lfloor ns\rfloor}^-(t)}{\sqrt{\lfloor ns\rfloor}}\,F(t)^p ,
\\
\mathrm{II}_n(s,t)
&
=\frac{\lfloor ns\rfloor(n-\lfloor ns\rfloor)}{p!\,n^{3/2}}\,
\frac{\alpha_{n-\lfloor ns\rfloor}^+(t)}{\sqrt{n-\lfloor ns\rfloor}}\,F(t)^p ,
\\
\mathrm{III}_n(s,t)
&
=\frac{\lfloor ns\rfloor(n-\lfloor ns\rfloor)}{n^{3/2}}\,
\sum_{k=2}^{p+1} b_k^{(p)}\,F(t)^{p+1-k}
\left(\frac{\alpha_{\lfloor ns\rfloor}^-(t)^k}{\lfloor ns\rfloor^{k/2}}
-\frac{\alpha_{n-\lfloor ns\rfloor}^+(t)^k}{(n-\lfloor ns\rfloor)^{k/2}}\right)\!,
\\
\mathrm{IV}_n(s,t)
&
=\frac{\lfloor ns\rfloor(n-\lfloor ns\rfloor)}{n^{3/2}}\,
\sum_{k=1}^p a_k^{(p)}\left(
\frac{\mathbb{F}_{\lfloor ns\rfloor}^-(t)^k}{\lfloor ns\rfloor^{p-k+1}}
-\frac{\mathbb{F}_{n-\lfloor ns\rfloor}^+(t)^k}{(n-\lfloor ns\rfloor)^{p-k+1}}\right)\!.
\end{align*}
Concerning $\mathrm{III}_n$, we have the estimate below:
\begin{align}
\left|\mathrm{III}_n(s,t)\right|
&
\leq \frac{\lfloor ns\rfloor(n-\lfloor ns\rfloor)}{n^{3/2}}\,\sum_{k=2}^{p+1}
\left[\left(\frac{\big|\alpha_{\lfloor ns\rfloor}^-(t)\big|}{\sqrt{\lfloor ns\rfloor}}\right)^{\!\!k}
+\left(\frac{\big|\alpha_{n-\lfloor ns\rfloor}^+(t)\big|}{\sqrt{n-\lfloor ns\rfloor}}\right)^{\!\!k}\right]
\nonumber\\
&
\leq \frac{1}{\sqrt{n}}\,\alpha_{\lfloor ns\rfloor}^-(t)^2\,\sum_{k=0}^{p-1}
\left(\frac{\big|\alpha_{\lfloor ns\rfloor}^-(t)\big|}{\sqrt{\lfloor ns\rfloor}}\right)^{\!\!k}
+\frac{1}{\sqrt{n}}\,\alpha_{n-\lfloor ns\rfloor}^+(t)^2\,\sum_{k=0}^{p-1}
\left(\frac{\big|\alpha_{n-\lfloor ns\rfloor}^+(t)\big|}{\sqrt{n-\lfloor ns\rfloor}}\right)^{\!\!k}
\label{III}
\end{align}
We learn from (\ref{inequality}) that $|\alpha_n(t)/\sqrt{n}|
=|\mathbb{F}_n(t)-F(t)|\leq 1$ for any $t\in\mathbb{R}$ and any
$n\in\mathbb{N}^*$ and, of course, similar inequalities hold for
$\alpha_{\lfloor ns\rfloor}^-$ and $\alpha_{n-\lfloor ns\rfloor}^+$.
We deduce that both sums displayed in (\ref{III}) are not greater than $p$ and
by (\ref{empililaeae}), with probability~$1$, as $n \to\infty$, uniformly in
$s$ and $t$,
\begin{equation}\label{IIIbis}
\left|\mathrm{III}_n(s,t)\right|
\leq \frac{p}{\sqrt{n}}\,\left(\alpha_{\lfloor ns\rfloor}^-(t)^2
+\alpha_{n-\lfloor ns\rfloor}^+(t)^2\right)
=\mathcal{O}\!\left(\frac{\log \log n}{\sqrt{n}}\right)\!.
\end{equation}
Concerning $\mathrm{IV}_n$, we have the estimate below:
\begin{equation}\label{IV}
\left|\mathrm{IV}_n(s,t)\right|
\leq A_p'\!\left(\frac{n-\lfloor ns\rfloor}{n^{3/2}}\,
\sum_{k=1}^p \frac{\mathbb{F}_{\lfloor ns\rfloor}^-(t)^k}{\lfloor ns\rfloor^{p-k}}
+\frac{\lfloor ns\rfloor}{n^{3/2}}\, \sum_{k=1}^p
\frac{\mathbb{F}_{n-\lfloor ns\rfloor}^+(t)^k}{(n-\lfloor ns\rfloor)^{p-k}}\right)
\end{equation}
where we set $A_p':=\max_{1\leq i\leq p}a_i^{(p)}>0$.
Because of (\ref{inequality}) and the convention that
$\mathbb{F}_{\lfloor ns\rfloor}^-=0$ if $s\in(0,1/n)$,
we see that both sums displayed in (\ref{IV}) are not greater than $p$ and,
as $n \to\infty$, uniformly in $s$ and $t$,
\begin{equation}\label{IVbis}
\mathrm{IV}_n(s,t)=\mathcal{O}\!\left(\frac{1}{\sqrt{n}}\right)\!.
\end{equation}
As a byproduct, we get from (\ref{equation78}), (\ref{IIIbis}) and (\ref{IVbis})
that, with probability~$1$, as $n \to\infty$, uniformly in $s$ and $t$,
\begin{equation}\label{alphawide}
\widetilde{\alpha}_n^{(p)}(s,t)=\mathrm{I}_n(s,t)-\mathrm{II}_n(s,t)
+\mathcal{O}\!\left(\frac{\log \log n}{\sqrt{n}}\right)\!.
\end{equation}

Next, it is convenient to introduce, for $n\in \mathbb{N}^*$ and $s,u\in(0,1)$,
\begin{align*}
\gamma_n(u)
&
=\sqrt{n}\,\beta_n(u)=\sum_{i=1}^n\left(\mathbbm{1}_{\{ U_i\leq u\}}-u\right)\!,
\\
\gamma_{\lfloor ns\rfloor}^-(u)
&
=\sqrt{\lfloor ns\rfloor}\,\beta_{\lfloor ns\rfloor}^-(u)
=\sum_{i=1}^{\lfloor ns\rfloor}\left(\mathbbm{1}_{\{ U_i\leq u\}}-u\right)\!,
\\
\gamma_{n-\lfloor ns\rfloor}^+(u)
&
=\sqrt{n-\lfloor ns\rfloor}\,\beta_{n-\lfloor ns\rfloor}^+(u)
=\sum_{i=\lfloor ns\rfloor+1}^n \left(\mathbbm{1}_{\{ U_i\leq u\}}-u\right)\!,
\\
\mathrm{I}_n'(s,u)&=\frac{\lfloor ns\rfloor(n-\lfloor ns\rfloor)}{p!\,n^{3/2}}
\,\frac{\beta_{\lfloor ns\rfloor}^-(u)}{\sqrt{\lfloor ns\rfloor}}\,u^p
=\frac{n-\lfloor ns\rfloor}{p!\,n^{3/2}} \,u^p \,\gamma_{\lfloor ns\rfloor}^-(u),
\\
\mathrm{II}_n'(s,u)&=\frac{\lfloor ns\rfloor(n-\lfloor ns\rfloor)}{p!\,n^{3/2}}
\,\frac{\beta_{n-\lfloor ns\rfloor}^+(u)}{\sqrt{n-\lfloor ns\rfloor}}\,u^p
=\frac{\lfloor ns\rfloor}{p!\,n^{3/2}} \,u^p \,\gamma_{n-\lfloor ns\rfloor}^+(u),
\end{align*}
and
$$
\delta_n(s,u)=\mathrm{I}_n'(s,u)-\mathrm{II}_n'(s,u).
$$
Then, by (\ref{identity}), we plainly have the following equalities:
\begin{align*}
\mathrm{I}_n'(s,F(t))&=\mathrm{I}_n(s,t),\\
\mathrm{II}_n'(s,F(t))&=\mathrm{II}_n(s,t),\\
\gamma_n(u)&=\gamma_{\lfloor ns\rfloor}^-(u)+\gamma_{n-\lfloor ns\rfloor}^+(u),
\end{align*}
and we rewrite (\ref{alphawide}), with probability~$1$, as $n \to\infty$,
uniformly in $s$ and $t$, as
\begin{equation}\label{equatioreffree}
\widetilde{\alpha}_n^{(p)}(s,t)= \delta_n(s,F(t))
+\mathcal{O}\!\left(\frac{\log \log n}{\sqrt{n}}\right)\!.
\end{equation}

Now, let us rewrite $\delta_n(s,u)$ as
\begin{align}
\delta_n(s,u)&=\frac{u^p}{p!\,\sqrt{n}}\left( \gamma_{\lfloor ns\rfloor}^-(u)
-\frac{\lfloor ns\rfloor}{n}\,\gamma_n(u)\right)\nonumber\\
&=\frac{u^p}{p!\,\sqrt{n}}\left( \frac{n-\lfloor ns\rfloor}{n}\,
\gamma_n(u)-\gamma_{n-\lfloor ns\rfloor}^+(u)\right)\!.\label{h-n15,31}
\end{align}
We know from \cite{KMT1975} and \cite{Csorgo1997} that, with probability $1$,
as $n\to\infty$,
\begin{align}
\sup_{s\in[0,1/2]}\sup_{u\in[0,1]}\left| \gamma_{\lfloor ns\rfloor}^-(u)
-\mathbb{K}_2(\lfloor ns\rfloor,u)\right| &=\mathcal{O}\!\left((\log n)^2\right)\!,
\nonumber\\[-1ex]
\label{h-n15,36}\\[-1ex]
\sup_{s\in[1/2,1]}\sup_{u\in[0,1]}\left| \gamma_{n-\lfloor ns\rfloor}^+(u)
-\mathbb{K}_1(\lfloor ns\rfloor,u)\right|&=\mathcal{O}\!\left((\log n)^2\right)\!.
\nonumber
\end{align}
In particular, we have, with probability $1$, as $n\to\infty$,
\begin{align}
\sup_{u\in[0,1]}\left| \gamma_{\lfloor n/2\rfloor}^-(u)
-\mathbb{K}_2(\lfloor n/2\rfloor,u)\right| &=\mathcal{O}\!\left((\log n)^2\right)\!,
\label{h-n15,36a}\\[1ex]
\sup_{u\in[0,1]}\left| \gamma_{n-\lfloor n/2\rfloor}^+(u)
-\mathbb{K}_1(\lfloor n/2\rfloor,u)\right|&=\mathcal{O}\!\left((\log n)^2\right)\!.
\label{h-n15,36b}
\end{align}
As a byproduct, by adding (\ref{h-n15,36a}) and (\ref{h-n15,36b}), we readily
infer that, with probability $1$, as $n\to\infty$,
\begin{equation}\label{prth7}
\sup_{u\in[0,1]}\left| \gamma_n(u)-\big( \mathbb{K}_1(\lfloor n/2\rfloor,u)
+\mathbb{K}_2(\lfloor n/2\rfloor,u)\big)\right|=\mathcal{O}\!\left((\log n)^2\right)\!.
\end{equation}
Recall the definition of the process
$\overset{\text{\tiny o}}{\mathbb{K}}\vphantom{K}_n^{(p)}$ given by (\ref{Knp}).
From (\ref{h-n15,31}), (\ref{h-n15,36}), and (\ref{prth7}), we deduce that,
with probability $1$, as $n\to\infty$,
\begin{equation}\label{ecart}
\sup_{s,u\in[0,1]} \big|\delta_n(s,u)
-\overset{\text{\tiny o}}{\mathbb{K}}\vphantom{K}_n^{(p)}(s,u)\big|
=\mathcal{O}\!\left(\frac{(\log n)^2}{\sqrt{n}}\right)\!.
\end{equation}
Finally, we conclude by using the triangle inequality
\begin{align}
\sup_{s\in[0,1]}\sup_{t\in\mathbb{R}} \big|\widetilde{\alpha}_n^{(p)}(s,t)
-\overset{\text{\tiny o}}{\mathbb{K}}\vphantom{K}_n^{(p)}(s,F(t))\big|\leq
&\sup_{s\in[0,1]}\sup_{t\in\mathbb{R}} \big|\widetilde{\alpha}_n^{(p)}(s,t)
-\delta_n(s,F(t))\big|
\nonumber\\
&\;+\sup_{s,u\in[0,1]} \big|\delta_n(s,u)
-\overset{\text{\tiny o}}{\mathbb{K}}\vphantom{K}_n^{(p)}(s,u)\big|
\label{ecartbis}
\end{align}
and next by putting (\ref{equatioreffree}) and (\ref{ecart}) into (\ref{ecartbis}).
The proof of Theorem~\ref{theorem1} is completed.
\hfill$\Box$

%%%%%%%%%%%%%%%%%%%%%%%%%%%%%%%%%%%%%%%%%%%%%%%%%%%%%%%%%%%%%%%%%%%%%%%%%%%%%%%
\subsection{Proof of Theorem~\ref{theoremBurkeetal}}
%%%%%%%%%%%%%%%%%%%%%%%%%%%%%%%%%%%%%%%%%%%%%%%%%%%%%%%%%%%%%%%%%%%%%%%%%%%%%%%

Recall the definition of $\widehat{\alpha}_n^{(p)}(t)$ given by (\ref{alphabarchap})
and write it as follows: by using (\ref{alphanp}) \textit{mutatis mutandis},
a.s., for any $t\in\mathbb{R}$ and any $n\in\mathbb{N}^*$,
\begin{equation}\label{alphanpestim}
\widehat{\alpha}_n^{(p)}(t)
=\frac{1}{p!}\,F\big(t,\widehat{\boldsymbol{\theta}}_n\big)^p\,\widehat{\alpha}_n(t)
+\sum_{k=2}^{p+1} b_k^{(p)}\,F\big(t,\widehat{\boldsymbol{\theta}}_n\big)^{p+1-k}
\,\frac{\widehat{\alpha}_n(t)^k}{n^{(k-1)/2}}
+\sum_{k=1}^p a_k^{(p)}\,\frac{\mathbb{F}_n(t)^k}{n^{p-k+1/2}}.
\end{equation}
Substituting $\widehat{\alpha}_n(t)=\big(\widehat{\alpha}_n(t)-G_n(t)\big)+G_n(t)$
into (\ref{alphanpestim}) and using the binomial theorem yield, a.s.,
for any $t\in\mathbb{R}$ and any $n\in\mathbb{N}^*$,
\begin{align}
\widehat{\alpha}_n^{(p)}(t)-G_n^{(p)}(t)=
&\, \frac{1}{p!}\,F\big(t,\widehat{\boldsymbol{\theta}}_n\big)^p
\big(\widehat{\alpha}_n(t)-G_n(t)\big)
+\frac{1}{p!}\left(F\big(t,\widehat{\boldsymbol{\theta}}_n\big)^p
-F(t,\boldsymbol{\theta}_0)^p\right)G_n(t)
+\sum_{k=1}^p a_k^{(p)}\,\frac{\mathbb{F}_n(t)^k}{n^{p-k+1/2}}
\nonumber\\
&+\sum_{k=2}^{p+1} b_k^{(p)}\,\frac{F\big(t,
\widehat{\boldsymbol{\theta}}_n\big)^{p+1-k}}{n^{(k-1)/2}}
\sum_{i=0}^k {k\choose i} G_n(t)^{k-i} \big(\widehat{\alpha}_n(t)-G_n(t)\big)^i.
\label{alphanpestimbis}
\end{align}
By (\ref{inequality}) and by appealing to the elementary identity
$a^p-b^p=(a-b)\sum_{i=0}^{p-1}a^ib^{p-i-1}$, we extract from
(\ref{alphanpestimbis}) the following inequality: a.s., for any $n\in\mathbb{N}^*$,
\begin{align*}
\sup_{t\in\mathbb{R}} \big|\widehat{\alpha}_n^{(p)}(t)-G_n^{(p)}(t)\big|\leq
&\, \sup_{t\in\mathbb{R}} \big|\widehat{\alpha}_n(t)-G_n(t)\big|
+\sup_{t\in\mathbb{R}} \left|F\big(t,\widehat{\boldsymbol{\theta}}_n\big)
-F(t,\boldsymbol{\theta}_0)\right| \,\sup_{t\in\mathbb{R}} |G_n(t)|
+\frac{A_p}{\sqrt{n}}
\\
&+\frac{A_p''}{\sqrt{n}} \left(\sum_{k=0}^{p+1} \sup_{t\in\mathbb{R}} |G_n(t)|^k\right)\!
\left(\sum_{k=0}^{p+1} \sup_{t\in\mathbb{R}} \big|\widehat{\alpha}_n(t)-G_n(t)\big|^k\right)
\end{align*}
where we set $A_p'':=\max_{0\leq i,k\leq p}{k\choose i}>0$.
Recall the notation (\ref{epsilon}) of $\boldsymbol{\varepsilon}_n^{(p)}$ and set
$$
\boldsymbol{\eta}_n:=\sup_{t\in\mathbb{R}}\left|\widehat{\alpha}_n(t)-G_n(t)\right|\!.
$$
We have thus obtained the inequality
\begin{align}
\boldsymbol{\varepsilon}_n^{(p)}\leq&\; \boldsymbol{\eta}_n
+\sup_{t\in\mathbb{R}}\left|F\big(t,\widehat{\boldsymbol{\theta}}_n\big)
-F(t,\boldsymbol{\theta}_0)\right| \,\sup_{t\in\mathbb{R}}|G_n(t)|
\nonumber
\\&+\frac{A_p''}{\sqrt{n}}\left(\sum_{k=0}^p\boldsymbol{\eta}_n^k\right)\!
\left(\sum_{k=0}^p \sup_{t\in\mathbb{R}}|G_n(t)|^k\right) +\frac{A_p}{\sqrt{n}}.
\label{reffequa}
\end{align}
We know from Theorem 3.1 of \cite{Burke-Csorgo} that $\boldsymbol{\eta}_n$
satisfies the same limiting results than those displayed in Theorem
\ref{theoremBurkeetal} for $\boldsymbol{\varepsilon}_n^{(p)}$.

Next, we need to derive some bounds for $\sup_{t\in\mathbb{R}}|G_n(t)|$ and
$\sup_{t\in\mathbb{R}}\big|F\big(t,\widehat{\boldsymbol{\theta}}_n\big)
-F(t,\boldsymbol{\theta}_0)\big|$ as $n\to\infty$. First, by using
(\ref{estimbridge}) and noticing that the same bound holds true for
$\mathbf{ W}(n)$, and by Condition (iv) and the definition of $G_n(t)$,
we see that, with probability $1$, as $n\to\infty$,
\begin{equation}\label{reffffreco}
\sup_{t\in\mathbb{R}}|G_n(t)|=\mathcal{O}\!\left(\!\sqrt{\log \log n}\right)\!.
\end{equation}
On the other hand, using the one-term Taylor expansion of
$F(\cdot,\boldsymbol{\theta})$ with respect to $\boldsymbol{\theta}_0$,
there exists $\boldsymbol{\theta}_n^*$ lying in the segment
$\big[\boldsymbol{\theta}_0, \widehat{\boldsymbol{\theta}}_n\big]$ such that
\begin{equation}\label{reffffreco23}
F\big(t,\widehat{\boldsymbol{\theta}}_n\big)-F(t,\boldsymbol{\theta}_0)
=\left(\widehat{\boldsymbol{\theta}}_n-\boldsymbol{\theta}_0\right)
\nabla_{\boldsymbol{\theta}}F(t,\boldsymbol{\theta}_n^*)^{\top}.
\end{equation}

In case (a) of Theorem~\ref{theoremBurkeetal},
$\sqrt{n}\big(\widehat{\boldsymbol{\theta}}_n-\boldsymbol{\theta}_0\big)$
is asymptotically normal and then
$n^{1/4}\big(\widehat{\boldsymbol{\theta}}_n-\boldsymbol{\theta}_0\big)$
tends to zero as $n \to\infty$ in probability.
Therefore, by (\ref{reffffreco}) and (\ref{reffffreco23}),
$\sup_{t\in\mathbb{R}}\big|F\big(t,\widehat{\boldsymbol{\theta}}_n\big)
-F(t,\boldsymbol{\theta}_0)\big|\sup_{t\in\mathbb{R}}|G_n(t)|$
also tends to zero as $n\to\infty$, in probability. Putting this into
(\ref{reffequa}), we easily complete the proof of Theorem
\ref{theoremBurkeetal} in this case. In cases (b) and (c) of Theorem
\ref{theoremBurkeetal}, referring to \cite{Burke-Csorgo} p.~779, we have the
following bound for $\widehat{\boldsymbol{\theta}}_n-\boldsymbol{\theta}_0$:
with probability $1$, as $n\to\infty$,
$$
\widehat{\boldsymbol{\theta}}_n-\boldsymbol{\theta}_0
=\mathcal{O}\!\left(\!\sqrt{\frac{\log \log n}{n}}\,\right)\!.
$$
By putting this into (\ref{reffffreco23}) and next in (\ref{reffequa}) with the
aid of (\ref{reffffreco}), we complete the proof of Theorem
\ref{theoremBurkeetal} in these two cases.

Finally, concerning $\widehat{G}_n^{(p)}(t)$, we have
\begin{align*}
\widehat{G}_n^{(p)}(t)-G_n^{(p)}(t)=&\;\frac{1}{p!}\,F\big(t,\widehat{\boldsymbol{\theta}}_n\big)^p
\left(\widehat{G}_n(t)-G_n(t)\right)
\\&+\frac{1}{p!}
\left(F\big(t,\widehat{\boldsymbol{\theta}}_n\big)^p-F(t,{\boldsymbol{\theta}}_0)^p\right)G_n(t),
\end{align*}
from which we deduce
\begin{align*}
\sup_{t\in\mathbb{R}}\left|\widehat{G}_n^{(p)}(t)-G_n^{(p)}(t)\right|\leq&\;
\sup_{t\in\mathbb{R}}\left|\widehat{G}_n(t)-G_n(t)\right|
\\&+\sup_{t\in\mathbb{R}}\left|F\big(t,\widehat{\boldsymbol{\theta}}_n\big)
-F(t,{\boldsymbol{\theta}}_0)\right| \, \sup_{t\in\mathbb{R}}|G_n(t)|.
\end{align*}
Using the same bounds than previously, we immediately derive (\ref{eqreferrrz11}).
\hfill$\Box$

%%%%%%%%%%%%%%%%%%%%%%%%%%%%%%%%%%%%%%%%%%%%%%%%%%%%%%%%%%%%%%%%%%%%%%%%%%%%%%%
\appendix\section{Appendix : other integrated empirical distribution functions}
%%%%%%%%%%%%%%%%%%%%%%%%%%%%%%%%%%%%%%%%%%%%%%%%%%%%%%%%%%%%%%%%%%%%%%%%%%%%%%%

To end up this article, let us point out that a similar analysis may be carried
out with other integrated empirical d.f.s and integrated processes.
For instance, we present below two other families of integrated empirical d.f.'s.
The underlying d.f. $F$ is still assumed to be continuous.

%%%%%%%%%%%%%%%%%%%%
\begin{definition}\label{defbis}
We define the families of integrated d.f.'s and integrated empirical d.f.'s,
for any $p\in\mathbb{N}$, any $n\in\mathbb{N}^*$ and any $t\in\mathbb{R}$, as
$$
\widetilde{F}^{(p)}(t):=\int_{-\infty}^t F(s)^p\,dF(s),\quad
\widetilde{\mathbb{F}}_n^{(p)}(t):=\int_{-\infty}^t \mathbb{F}_n(s)^p\,d\mathbb{F}_n(s)
$$
and
$$
\breve{F}^{(p)}(t):=\int_{-\infty}^t \big(F(t)-F(s)\big)^p\,dF(s),\quad
\breve{\mathbb{F}}_n^{(p)}(t):=\int_{-\infty}^t \big(\mathbb{F}_n(t)
-\mathbb{F}_n(s)\big)^p\,d\mathbb{F}_n(s)
$$
together with the corresponding family of integrated empirical processes as
\begin{align*}
\widetilde{\alpha}_n^{(p)}(t)&:=\sqrt{n}\left(\widetilde{\mathbb{F}}_n^{(p)}(t)
-\widetilde{F}^{(p)}(t)\right)\!,\\
\breve{\alpha}_n^{(p)}(t)&:=\sqrt{n}\left(\breve{\mathbb{F}}_n^{(p)}(t)
-\breve{F}^{(p)}(t)\right)\!.
\end{align*}
\end{definition}
%%%%%%%%%%%%%%%%%%%%

We have, from (\ref{general}), a.s., for any $p\in\mathbb{N}$, any $n\in\mathbb{N}^*$
and any $t\in\mathbb{R}$,
$$
\widetilde{\mathbb{F}}_n^{(p)}(t)
=\frac{1}{n} \sum_{i=1}^{n\mathbb{F}_n(t)} \mathbb{F}_n(X_{i,n})^p
\quad\text{and}\quad
\breve{\mathbb{F}}_n^{(p)}(t)
=\frac{1}{n} \sum_{i=1}^{n\mathbb{F}_n(t)} \big(\mathbb{F}_n(t)-\mathbb{F}_n(X_{i,n})\big)^p.
$$
Since $\mathbb{F}_n(X_{i,n})=i/n$, we obtain the following closed forms.
%%%%%%%%%%%%%%%%%%%%
\begin{proposition}\label{expFptilde}
For each $p\in\mathbb{N}$, we explicitly have, with probability~$1$,
$$
\widetilde{F}^{(p)}(t)=\breve{F}^{(p)}(t)=\frac{F(t)^{p+1}}{p+1}
\quad\text{for}\quad t\in\mathbb{R},
$$
and
\begin{equation}\label{repFnp1bis}
\widetilde{\mathbb{F}}_n^{(p)}(t)=\frac{1}{n^{p+1}} \sum_{i=1}^{n\mathbb{F}_n(t)} i^p,\quad
\breve{\mathbb{F}}_n^{(p)}(t)=\frac{1}{n^{p+1}} \sum_{i=0}^{n\mathbb{F}_n(t)-1} i^p
\quad\text{for}\quad t\in\mathbb{R},\,n\in\mathbb{N}^*.
\end{equation}
\end{proposition}
Observe the relation, a.s. valid, for all $p\in\mathbb{N}^*$, any $n\in\mathbb{N}^*$
and any $t\in\mathbb{R}$,
$$
\breve{\mathbb{F}}_n^{(p)}(t)=\widetilde{\mathbb{F}}_n^{(p)}(t)-\frac{1}{n}\,F_n(t)^p.
$$

%%%%%%%%%%%%%%%%%%%%
\begin{proposition}\label{repFnp3bis}
The empirical d.f.\/ $\widetilde{\mathbb{F}}_n^{(p)}$ can be expressed by means of\/
$\mathbb{F}_n$ as follows: with probability~$1$,
\begin{equation}\label{repFnp2bis}
\widetilde{\mathbb{F}}_n^{(p)}(t)=\frac{\mathbb{F}_n(t)^{p+1}}{p+1}
+\sum_{k=1}^p \widetilde{a}_k^{(p)}\,\frac{\mathbb{F}_n(t)^k}{n^{p-k+1}}
\quad\text{for}\quad t\in\mathbb{R},\,n\in\mathbb{N}^*,
\end{equation}
where the coefficients $\widetilde{a}_k^{(p)}$, $1\leq k\leq p$, are rational numbers.
\end{proposition}
%%%%%%%%%%%%%%%%%%%%
\begin{proof}
Appealing to Bernoulli's formula
$$
\sum_{i=1}^n i^p=\frac{1}{p+1}\,n^{p+1}+\frac{1}{2}\,n^p
+\frac{1}{p+1}\,\sum_{k=1}^{p-1} {p+1 \choose k} B_{p+1-k}n^k
$$
where the $B_k$'s are the Bernoulli numbers
(see, e.g., http://en.wikipedia.org/wiki/Bernoulli\_number),
(\ref{repFnp1bis}) immediately yields (\ref{repFnp2bis}) with the coefficients
$$\widetilde{a}_k^{(p)}:={p+1 \choose k} B_{p+1-k}/(p+1) \quad\mbox{for}\quad k\in\{1,\dots,p-1\},$$
and $\widetilde{a}_p^{(p)}:=1/2$.
\end{proof}
%%%%%%%%%%%%%%%%%%%%

Below, we state the expression of $\widetilde{\alpha}_n^{(p)}$ by means of
$\alpha_n$ analogous to (\ref{alphanp}).
%%%%%%%%%%%%%%%%%%%%
\begin{proposition}
The integrated empirical process $\widetilde{\alpha}_n^{(p)}$ is related to
the empirical process $\alpha_n$ according to, with probability~$1$,
$$
\widetilde{\alpha}_n^{(p)}(t)=F(t)^p\,\alpha_n(t)
+\sum_{k=2}^{p+1} \widetilde{b}_k^{(p)}\,\frac{F(t)^{p+1-k}}{n^{(k-1)/2}}\,\alpha_n(t)^k
+\sum_{k=1}^p \widetilde{a}_k^{(p)}\,\frac{\mathbb{F}_n(t)^k}{n^{p-k+1/2}}
\quad\text{for}\quad t\in\mathbb{R}, \,n\in\mathbb{N}^*,
$$
where the coefficients $\widetilde{a}_k^{(p)}$, $1\leq k\leq p$, are those of
Proposition~\ref{repFnp3bis} and the $\widetilde{b}_k^{(p)}$, $2\leq k\leq p+1$,
are positive real numbers less than $p!$.
\end{proposition}
%%%%%%%%%%%%%%%%%%%%
The coefficients $b_k^{(p)}$ are given by $b_k^{(p)}:={p+1\choose k}/(p+1)$.
\\

More generally, we could define a broader family indexed by polynomials of two variables.
%%%%%%%%%%%%%%%%%%%%
\begin{definition}
We define the family of integrated d.f.'s and integrated empirical d.f.'s,
for any polynomials $\mathbf{P}$ of two variables, any $n\in\mathbb{N}^*$ and any $t\in\mathbb{R}$, as
\begin{align*}
F^{(\mathbf{P})}(t)&:=\int_{-\infty}^t \mathbf{P}(F(s),F(t))\,dF(s),\\
\mathbb{F}_n^{(\mathbf{P})}(t)&:=\int_{-\infty}^t \mathbf{P}(\mathbb{F}_n(s),
\mathbb{F}_n(t))\,\mathbb{F}_n(s),
\end{align*}
together with the corresponding family of integrated empirical processes as
$$
\alpha_n^{\mathbf{P}}(t):=\sqrt{n}\left(\mathbb{F}_n^\mathbf{P}(t)-F^\mathbf{P}(t)\right)\!.
$$
\end{definition}

Below, we state the last result of the paper which is a representation of
$\alpha_n^{\mathbf{P}}$ by means of $\alpha_n$ analogous to (\ref{alphanp}).
This is the key point for deriving bounds similar to those obtained throughout the paper.
%%%%%%%%%%%%%%%%%%%%
\begin{proposition}
The empirical process $\alpha_n^{\mathbf{P}}$ can be express as follows:
with probability~$1$,
\begin{equation}\label{alphaP}
\alpha_n^{\mathbf{P}}(t)=\mathbf{Q}(t)\,\alpha_n(t)+\mathbf{R}_n(t)
\quad\text{for}\quad t\in\mathbb{R}, \,n\in\mathbb{N}^*,
\end{equation}
where $\mathbf{Q}$ is the polynomial function of $F$ defined by
$$
\mathbf{\mathbf{Q}}(t)=\mathbf{P}(F(t),F(t))+\int_0^{F(t)}
\frac{\partial \mathbf{P}}{\partial y}(x,F(t))\,dx
\quad\text{for}\quad t\in\mathbb{R},
$$
and $\mathbf{R}_n$ satisfies the inequality
$$
|\mathbf{R}_n(t)|\leq \frac{\mathbf{A}}{\sqrt{n}}+\mathbf{B}\sum_{k=2}^{\mathbf{d}}
\frac{\alpha_n(t)^k}{n^{(k-1)/2}}
\quad\text{for}\quad t\in\mathbb{R}, \,n\in\mathbb{N}^*,
$$
$\mathbf{A}$, $\mathbf{B}$, $\mathbf{d}$ being three constants depending on $\mathbf{P}$.
\end{proposition}
%%%%%%%%%%%%%%%%%%%%
\begin{proof}
Set $\mathbf{P}(x,y)=\sum_{i=0}^\mathbf{p}\sum_{j=0}^\mathbf{q} \mathbf{a}_{ij}\,x^iy^j$ for some integers $\mathbf{p},\mathbf{q}$
and some coefficients $\mathbf{a}_{ij}$. Then
$$
F^{(\mathbf{P})}(t)=\sum_{i=0}^\mathbf{p}\sum_{j=0}^\mathbf{q} \frac{\mathbf{a}_{ij}}{i+1}\,F(t)^{i+j+1}
$$
and, a.s., for any $n\in\mathbb{N}^*$ and any $t\in\mathbb{R}$,
\begin{align*}
\mathbb{F}_n^{(\mathbf{P})}(t)&=\sum_{i=0}^\mathbf{p}\sum_{j=0}^\mathbf{q}
\mathbf{a}_{ij}\,\mathbb{F}_n(t)^j\,\widetilde{\mathbb{F}}_n^{(i)}(t)
\\&=\sum_{i=0}^\mathbf{p}\sum_{j=0}^\mathbf{q} \frac{\mathbf{a}_{ij}}{i+1}\,\mathbb{F}_n(t)^{i+j+1}
+\sum_{i=0}^\mathbf{p}\sum_{j=0}^\mathbf{q}\sum_{k=1}^i \mathbf{a}_{ij}\,\widetilde{a}_k^{(i)}
\,\frac{\mathbb{F}_n(t)^{j+k}}{n^{i-k+1}}.
\end{align*}

Consequently, a.s., for any $n\in\mathbb{N}^*$ and any $t\in\mathbb{R}$,
\begin{align*}
\alpha_n^{\mathbf{P}}(t)=&\;\sqrt{n}\sum_{i=0}^\mathbf{p}\sum_{j=0}^\mathbf{q}
\frac{\mathbf{a}_{ij}}{i+1}\,\left(\mathbb{F}_n(t)^{i+j+1}-F(t)^{i+j+1}\right)
\\&+\sum_{i=0}^\mathbf{p}\sum_{j=0}^\mathbf{q}\sum_{k=1}^i \mathbf{a}_{ij}\,\widetilde{a}_k^{(i)}
\,\frac{\mathbb{F}_n(t)^{j+k}}{n^{i-k+1/2}}
\end{align*}
which, by using the same method than (\ref{alphanp}), we rewrite as (\ref{alphaP})
with
\begin{align*}
\mathbf{Q}(t)=&\;\sum_{i=0}^\mathbf{p}\sum_{j=0}^\mathbf{q} \frac{i+j+1}{i+1}\,\mathbf{a}_{ij}\,F(t)^{i+j}\\
=&\;\sum_{i=0}^\mathbf{p}\sum_{j=0}^\mathbf{q} \mathbf{a}_{ij}\,F(t)^{i+j}
+\sum_{i=0}^\mathbf{p}\sum_{j=0}^\mathbf{q} \frac{j}{i+1}\,\mathbf{a}_{ij}\,F(t)^{i+j},
\\
\mathbf{R}_n(t)=&\;\sum_{i=0}^\mathbf{p}\sum_{j=0}^\mathbf{q}\sum_{k=1}^i
\mathbf{a}_{ij}\,\widetilde{a}_k^{(i)}\,\frac{\mathbb{F}_n(t)^{j+k}}{n^{i-k+1/2}}
\\&
+\sum_{i=0}^\mathbf{p}\sum_{j=0}^\mathbf{q}\sum_{k=2}^{i+j+1}{i+j+1 \choose k}
\frac{\mathbf{a}_{ij}}{i+1}\,\frac{\alpha_n(t)^k}{n^{(k-1)/2}}\,F(t)^{i+j+1-k}.
\end{align*}
We easily conclude by using (\ref{inequality}).
\end{proof}

%%%%%%%%%%%%%%%%%%%%%%%%%%%%%%%%%%%%%%%%%%%%%%%%%%%%%%%%%%%%%%%%%%%%%%%%%%%%%%%
\section{Some possible extensions}
%%%%%%%%%%%%%%%%%%%%%%%%%%%%%%%%%%%%%%%%%%%%%%%%%%%%%%%%%%%%%%%%%%%%%%%%%%%%%%%

We will work under the following notation borrowed from \cite{zhang1997A}.
Let $X_1,\ldots, X_n$ be a sample of independent and identically distributed random variables
of a random variable $X$ with unknown distribution function $F(\cdot)$.
For the unknown distribution function $F(\cdot)$ underlying the random sample $X_1,\ldots, X_n$
we assume that we have the following auxiliary information:
there exist $r$ ($r\geq 1$) functionally independent functions $\eta_1(\cdot),\ldots,\eta_r(\cdot)$ such that,
by putting $\eta(\cdot)= (\eta_1(\cdot),\ldots,\eta_r(\cdot))^\top$,
\begin{equation}\label{(1.1)}
\mathbb{E}(\eta(X))=0.
\end{equation}
By $r$ linearly independent functions
$\eta_1(\cdot),\ldots,\eta_r(\cdot)$ we mean that any $\eta_i(\cdot)$ can not be expressed as a linear
combination of $\eta_1(\cdot),\ldots,\eta_{i-1}(\cdot),$ $\eta_{i+1}(\cdot),\ldots,\eta_r(\cdot)$ for $i = 1,\ldots, r$.
\noindent Let $p = (p_{1},\ldots,p_{n})$ denote a multinomial distribution on the
points $X_1,\ldots, X_n$ and put
$$
L(p)=\prod_{i=1}^np_{i}.
$$
Under the assumption (\ref{(1.1)}), the profile empirical likelihood function $L$
is defined by
$$
L=\max_{p}L(p)=\max_{p}\prod_{i=1}^np_{i}
$$
where the maximum is taken on the $n$-uples $p = (p_{1},\ldots,p_{n})$ subject to the constraints
$$
\sum_{i=1}^np_{i}=1,\quad\sum_{i=1}^np_{i}\eta(X_i)=0\quad\mbox{and}\quad p_{i}>0
\quad\mbox{for } i=1,\ldots,n.
$$
If $0$ is inside the convex hull of the points $\eta(X_1),\ldots,\eta(X_n)$, then $L$ exists
uniquely. A little calculus of variations shows that
$$
L=\max_{p}\prod_{i=1}^n\widehat{p}_{i},
$$
where
$$
\widehat{p}_{i}=\frac{1}{n}\left(\frac{1}{1+\lambda^\top \eta(X_i)}\right)
\quad\mbox{for } 1\leq i\leq n,
$$
$\lambda=(\lambda_1,\ldots,\lambda_r)^\top$ being the solution of
$$
\sum_{i=1}^n\widehat{p}_i\eta(X_i)
=\frac{1}{n}\sum_{i=1}^n\frac{1}{1+\lambda^\top \eta(X_i)}\eta(X_i)=0.
$$
Now, let
$$
\widehat{F}_n(z)=\sum_{i=1}^n\widehat{p}_i\mathbbm{1}_{\{X_i\leq z\}}
=\frac{1}{n}\sum_{i=1}^n\frac{\mathbbm{1}_{\{X_i\leq z\}}}{1+\lambda^\top \eta(X_i)}.
$$
Then $\widehat{F}_n(\cdot)$ can be viewed as an alternative estimator of $F(\cdot)$ satisfying (\ref{(1.1)}).
On the other hand, in the absence of the knowledge of (\ref{(1.1)}), the profile
empirical likelihood $L$ attains its maximum at the distribution $p = (1/n,\ldots,1/n)$, and
hence $\widehat{F}_n= F_n,$ the standard empirical distribution function.
Let us assume that $\Sigma=\mathbb{E}(\eta(X)\eta^{\top}(X))$ is a positive definite matrix and
$$
\mathbb{E}(\|\eta(X)\|^{2}_{r})=\int_{-\infty}^{\infty}\|\eta(x)\|_{r}^{2}dF(x)<\infty,
$$
where  $\|\cdot\|_{r}$ is the Euclidean norm in $\mathbb{R}^{r}$.
Let us denote by  $\{\mathbb{D}(x):x\in\mathbb{R}\}$ a centered Gaussian process
with sample continuous paths satisfying
\begin{align*}
\mathbb{E}(\mathbb{D}(x)\mathbb{D}(y))=&\;F(x\wedge y)-F(x)F(y)
\\
&-\mathbb{E}\big(\eta^\top(X)\mathbbm{1}_{\{X\leq x\}}\big)\Sigma^{-1}
\mathbb{E}\big(\eta(X)\mathbbm{1}_{\{X\leq y\}}\big)\quad\text{for}\quad x,y\in\mathbb{R}.
\end{align*}
\cite{zhang1997} proved the following functional central limit theorem
for the empirical process pertaining to $\widehat{F}_{n}(\cdot)$:
$$
\sqrt{n}(\widehat{F}_{n}-F)\Rightarrow \mathbb{D} \quad\mbox{in}\quad\mathcal{D}(\mathbb{R}).
$$
Above, ``$\Rightarrow \mathbb{D}\;\mbox{in}\;\mathcal{D}(\mathbb{R})$'' denotes convergence
in distribution in $\mathcal{D}(\mathbb{R})$.

\begin{definition}\label{def}
We define the families of integrated empirical d.f.'s, under auxiliary
information (\ref{(1.1)}), associated with the d.f.~$F$, for any $p\in\mathbb{N}$,
any $n\in\mathbb{N}^*$ and any $t\in\mathbb{R}$, as
$$
 \widehat{\mathbb{F}}_n^{(0)}(t):=\widehat{F}_n(t),
$$
$$
\widehat{\mathbb{F}}_n^{(1)}(t):=\int_{-\infty}^{t} \widehat{F}_n(s)\,d\widehat{F}_n(s),
$$
and, for $p\geq 2$,
$$
\widehat{\mathbb{F}}_n^{(p)}(t):=\int_{-\infty}^{t}\,d\widehat{F}_n(s_1)
\int_{-\infty}^{s_1}\,d\widehat{F}_n(s_2)\dots\int_{-\infty}^{s_{p-1}}
\widehat{F}_n(s_p)\,d\widehat{F}_n(s_p),
$$
together with the corresponding family of integrated empirical processes as
$$
\boldsymbol{\alpha}_n^{(p)}(t):=\sqrt{n}\left(\widehat{\mathbb{F}}_n^{(p)}(t)-F^{(p)}(t)\right)\!.
$$
\end{definition}

\begin{enumerate}
\item[$\bullet$]
In many interesting applications, we may have some partial information about
the distribution of the  population, although we do not know exactly the
underlying distribution function of the sample issued from the population;
for further details and motivation about such problem we refer to \cite{Owen1990,Owen1991,Owen2001}
and \cite{zhang2000}.
One of possible extension is to consider the problem of the strong and
weak convergences of the processes $\boldsymbol{\alpha}_n^{(p)}(\cdot)$ by making effective
use of auxiliary information (\ref{(1.1)}).

\item[$\bullet$]
According to \cite{cheng1995}, \cite{chengParzen1997}, some studies have shown
that a smoothed estimator  may be preferable to
the sample estimator. First, smoothing reduces the random variation in the data,
resulting in a more efficient estimator. Second, smoothing gives a smooth curve for
the quantile function that better displays the interesting features of the population
distribution. Motivated by all these facts, it will be of interest to consider
the smoothed versions of the processes considered in the present paper and to study
their asymptotic properties.
\end{enumerate}
%%%%%%%%%%%%%%%%%%%%

%%%%%%%%%%%%%%%%%%%%%%%%%%%%%%%%%%%%%%%%%%%%%%%%%%%%%%%

\end{document}